\documentclass[12pt,reqno]{amsart}
\usepackage{amssymb}
\usepackage{caption,dsfont,enumerate,esint,float,graphicx,mathtools}
\usepackage{perpage,thmtools}
\usepackage[margin=2.2cm]{geometry}
\usepackage{tikz}

\usepackage[T1]{fontenc}
\usepackage[english]{babel}
\usepackage{lmodern,microtype,hyperref}

\allowdisplaybreaks
\MakePerPage{footnote}

\DeclareMathOperator{\dist}{dist}
\DeclareMathOperator{\Op}{Op}
\DeclareMathOperator{\Rer}{Re}
\DeclareMathOperator{\sign}{sign}
\DeclareMathOperator{\Span}{span}
\DeclareMathOperator{\Tr}{Tr}
\DeclareMathOperator{\Vol}{Vol}

\newcommand{\aw}{\mathrm{aw}}
\newcommand{\bs}{\backslash}
\newcommand{\C}{\mathbb{C}}
\newcommand{\D}{\mathbb{D}}
\newcommand{\eb}{\mathbf{e}}
\newcommand{\Gc}{\mathcal{G}}
\newcommand{\Hb}{\mathbb{H}}
\newcommand{\Hc}{\mathcal{H}}
\newcommand{\Hr}{\mathrm{H}}
\newcommand{\Id}{\mathrm{I}}
\newcommand{\indicator}{\mathds{1}}
\newcommand{\Leb}{\mathrm{Leb}}
\newcommand{\Lip}{\mathrm{Lip}}
\newcommand{\M}{\mathbb{M}}
\newcommand{\N}{\mathbb{N}}
\newcommand{\Nc}{\mathcal{N}}
\newcommand{\ol}{\overline}
\newcommand{\Pc}{\mathcal{P}}
\newcommand{\PSL}{\mathrm{PSL}}
\newcommand{\R}{\mathbb{R}}
\newcommand{\Sc}{\mathcal{S}}
\newcommand{\SL}{\mathrm{SL}}
\newcommand{\T}{\mathbb{T}}
\newcommand{\ve}{\varepsilon}
\newcommand{\vk}{\varkappa}
\newcommand{\vp}{\varphi}
\newcommand{\w}{\omega}
\newcommand{\Xc}{\mathcal{X}}
\newcommand{\Z}{\mathbb{Z}}

\theoremstyle{plain}
\newtheorem{thm}{Theorem}[section]

\newtheorem{cor}[thm]{Corollary}
\newtheorem{lemma}[thm]{Lemma}
\newtheorem{prop}[thm]{Proposition}

\newtheorem{prob}{Problem}

\theoremstyle{definition}
\newtheorem*{defn}{Definition}
\newtheorem*{ex}{Example}

\theoremstyle{remark}
\newtheorem*{rmk}{Remark}

\numberwithin{equation}{section}

\title[Small scale quantum ergodicity in cat maps]{Small scale quantum ergodicity in cat maps}
\author{Xiaolong Han}
\email{xiaolong.han@csun.edu}
\address{Department of Mathematics, California State University, Northridge, CA 91330, USA}


\begin{document}
\begin{abstract}
Small scale quantum ergodicity concerns the distribution of eigenstates, or quasimodes, of ergodic systems in small regions, such as balls whose radii tend to zero with the semiclassical parameter. On negatively curved manifolds, almost all eigenfunctions are known to have mass comparable to the volume on balls at logarithmic scales. Whether these scales can be improved, and which scales are optimal, is a major open problem. For the quantum cat maps, as the semiclassical parameter $N^{-1}\to0$, we prove three groups of results on the equidistribution and the concentration of eigenstates and quasimodes at small scales.

First, for every $N$ and every orthonormal basis of eigenstates, almost all eigenstates are equidistributed on phase space balls of radius $(\log N)^{-\alpha}$ for every $\alpha<\frac12$. The same holds on position space balls of radius $(\log N)^{-\alpha}$ for every $\alpha<1$. For a set of $N$ of full density, on which the periods are long, almost all eigenstates are equidistributed at polynomial scales $N^{-\beta}$ for explicit $\beta>0$. We also obtain small scale versions of quantum unique ergodicity, in which every eigenstate is equidistributed at polynomial scales, along a set of $N$ of full density and for Hecke eigenbases. The proof at logarithmic scales combines time averages up to half the Ehrenfest time, using the exact Egorov theorem, with a moment estimate that uses the bad approximability of the stable and unstable directions.

Second, let $\mu$ be an invariant probability measure and $r_N\to0$. For each $N$, we construct one quasimode of order $O((\log N)^{-1})$ such that these quasimodes have semiclassical measure $\mu$ and are not equidistributed at any scale between $4N^{-\frac18}$ and $r_N$. The quasimodes are built along periodic orbits, which cannot be equidistributed at scales below the inverse square root of their periods. In particular, when $\mu$ is the Lebesgue measure and $r_N\to0$ logarithmically in $N$, these quasimodes are equidistributed at every fixed scale but not at logarithmic scales. Since the construction gives a single quasimode for each $N$, they constitute a zero density subset of orthonormal bases.

Third, we construct orthonormal bases of quasimodes of order $O((\log N)^{-1})$ in which a full density subset of the elements is equidistributed at every fixed scale, but not at the scales $(\log N)^{-\alpha}$ for $\alpha>\frac12$ in phase space and $\alpha>1$ in position space. The bases are obtained by applying the Gram--Schmidt process to projections of coherent states onto spectral subspaces of the quantum cat map.

In particular, for $N$ with short periods, the eigenvalues of the quantum cat map have large multiplicities, and these quasimodes can be chosen to be exact eigenstates. So there are eigenbases in which a full density subset of eigenstates is equidistributed at every fixed scale, as in the quantum ergodicity theorem, and at the scales $(\log N)^{-\alpha}$ for $\alpha<\frac12$ in phase space, but not for $\alpha>\frac12$. In position space, the critical exponent is $1$. Hence, the exponents $\frac12$ and $1$ in the first group are sharp. Even for a full density subset of eigenstates, equidistribution at every fixed scale does not extend beyond the critical logarithmic scales.
\end{abstract}

\maketitle

\newpage
\setcounter{tocdepth}{1}
\tableofcontents

\section{Introduction}\label{sec:intro}
\subsection{Small scale quantum ergodicity}\label{sec:introSSQE}
The central theme of Quantum Chaos is to understand how the eigenstates or quasimodes (that is, approximate eigenstates) of a quantum system are distributed when the underlying classical dynamics is chaotic. One of the most important examples of chaotic dynamical systems is given by the geodesic flow on a compact Riemannian manifold $\M$ with negative sectional curvature; see Anosov \cite{Ano}. Let $\{u_j\}_{j=1}^\infty$ be an orthonormal basis of eigenfunctions in $L^2(\M)$ of the Laplacian $\Delta$ (that is, an eigenbasis):
\[-\Delta u_j=h_j^{-2}u_j\quad\text{with }h_1\ge h_2\ge\cdots\text{ and }h_j\to0\text{ as }j\to\infty.\]
Here, $h_j^{-2}$ is the eigenvalue of $-\Delta$ and $h_j\to0$ is the semiclassical parameter\label{idx:h}. The eigenfunction $u_j$ oscillates at a wavelength comparable to $h_j$, which is the Planck scale in position space. Denote by $B(x,r)\subset\M$ the geodesic ball of radius $r$ centered at $x\in\M$.

The quantum ergodicity theorem\label{idx:QEtheorem} of Shnirelman \cite{Sh}, Zelditch \cite{Z1}, and Colin de Verdi\`ere \cite{CdV} asserts that almost all eigenfunctions are equidistributed. More generally, it holds on every compact manifold with ergodic geodesic flow. In particular, there is $S\subset\N$ of full density such that
\begin{equation}\label{eq:QEmanifold}
\lim_{j\in S,j\to\infty}\int_{B(x,r)}\left|u_j\right|^2\,d\Vol=\frac{\Vol\left(B(x,r)\right)}{\Vol(\M)}\quad\text{for every ball }B(x,r)\subset\M.
\end{equation}
We state the result for balls, to compare it later with equidistribution at small scales, but it holds for every Borel set whose boundary has measure zero. Here, a set $S\subset\N$ is said to have full density\label{idx:fulldensity} if
\[\lim_{J\to\infty}\frac{\left|S\cap\{1,\ldots,J\}\right|}J=1,\]
where $|S|$\label{idx:card} denotes the number of elements of a finite set $S$.

In fact, equidistribution in the quantum ergodicity theorem holds in phase space, which here is the unit cosphere bundle $S^*\M$ of $\M$. Denote by $\mu_L$ the Liouville measure on $S^*\M$, normalized so that $\mu_L(S^*\M)=1$. For $a\in C_0^\infty(T^*\M)$, denote by $\Op_h(a)$ a semiclassical pseudo-differential operator with principal symbol $a$ on the cotangent bundle $T^*\M$ of $\M$. Then there is a full density $S\subset\N$ such that
\begin{equation}\label{eq:QEphase}
\lim_{j\in S,j\to\infty}\left\langle\Op_{h_j}(a)u_j,u_j\right\rangle=\int_{S^*\M}a\,d\mu_L\quad\text{for all }a\in C_0^\infty(T^*\M).
\end{equation}

The quantum unique ergodicity\label{idx:QUE} (QUE) conjecture of Rudnick--Sarnak \cite{RS} asserts that \eqref{eq:QEphase} holds with $S=\N$ for every eigenbasis. It remains open. On arithmetic surfaces, QUE holds for Hecke eigenfunctions, that is, joint eigenfunctions of the Laplacian and the Hecke operators: see Lindenstrauss \cite{L} for compact surfaces, and Lindenstrauss \cite{L} and Soundararajan \cite{So} for the modular surface. See also Brooks--Lindenstrauss \cite{BL}, Holowinsky--Soundararajan \cite{HS}, and Silberman--Venkatesh \cite{SV} for related works, and the surveys of Nonnenmacher \cite{N}, Sarnak \cite{Sa}, and Zelditch \cite{Z3} for the development in this field of QUE.

Almost all eigenfunctions, and conjecturally all of them, are equidistributed at every fixed scale. Figure \ref{fig:hyperbolic} shows an eigenfunction on a compact hyperbolic surface. The eigenfunction oscillates at a wavelength comparable to $h$. As the eigenvalue $h^{-2}\to\infty$, its mass tends to distribute evenly on the surface. The figure also shows that this pattern of ``equidistribution'' looks uniform not only on balls of fixed radius, but also on much smaller balls. Small scale quantum ergodicity asks whether \eqref{eq:QEmanifold} remains valid for balls whose radii shrink as $j\to\infty$, uniformly in the center $x$, and down to which scale.

\begin{figure}[!ht]
\centering
\includegraphics{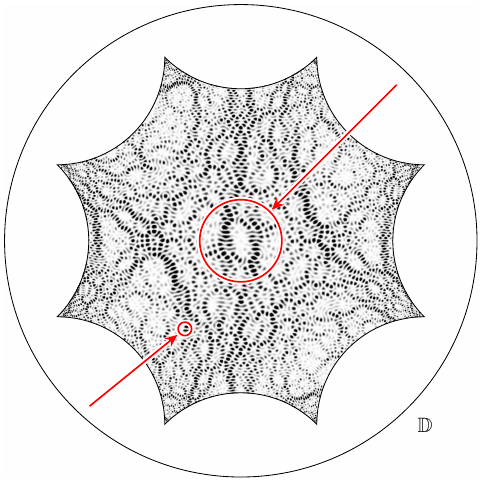}
\caption{An eigenfunction $u$ of the Laplacian on the Bolza surface, the compact hyperbolic surface of genus two obtained from the regular octagon in the Poincar\'e disk $\D$ by identifying opposite sides. It is computed numerically, with eigenvalue $h^{-2}$ about $8000$. Darker means larger $|u|^2$. The arrows point to a ball of radius several wavelengths and to a ball of radius about one wavelength.}\label{fig:hyperbolic}
\end{figure}

The author \cite{Han1} and Hezari--Rivi\`ere \cite{HR1} proved the following.

\begin{thm}[Small scale quantum ergodicity on negatively curved manifolds]\label{thm:HanHR}
Let $\M$ be negatively curved with dimension $d$ and $\alpha\in(0,\frac1{2d})$. For every eigenbasis $\{u_j\}_{j=1}^\infty$, there is $S\subset\N$ of full density such that
\[\int_{B(x,r_j)}\left|u_j\right|^2\,d\Vol\approx\Vol\left(B\left(x,r_j\right)\right)\quad\text{for all }j\in S\text{ and }x\in\M,\]
where $r_j=|\log h_j|^{-\alpha}$. Here, $X\approx Y$ means that $cY\le X\le CY$ for $c=c(\M)>0$ and $C=C(\M)>0$.
\end{thm}

That is, almost all eigenfunctions satisfy a weaker form of equidistribution at small scales: the mass of $u_j$ on these balls is comparable to their volume. These logarithmic scales reflect the limit of Egorov's theorem at the Ehrenfest time, which is of order $|\log h_j|$. Indeed, the proofs follow the geodesic flow only up to this time, during which, by Egorov's theorem, the quantum evolution is approximated by the classical evolution under the geodesic flow.

In arithmetic settings, the Hecke operators give additional tools. For example, on the modular surface $\PSL(2,\Z)\bs\Hb$, the following results are known.
\begin{itemize}
\item Almost all Hecke--Maass cusp forms are equidistributed at the polynomial scales $h_j^\beta$ for every $\beta\in(0,\frac1{42})$, uniformly over injective geodesic balls of these radii. This follows by Chebyshev's inequality from the bound of Luo--Sarnak \cite[Theorem 1.5]{LS} on the mean square of the discrepancy over injective geodesic balls.
\item Under the generalized Lindel\"of hypothesis, all Hecke--Maass cusp forms are equidistributed on balls of radius $h_j^{\tilde\beta}$ with a fixed center, for every $\tilde\beta<\frac13$; see Young \cite[Proposition 1.5]{Y}.
\item Under the same hypothesis, they are equidistributed on almost every ball of radius $h_j^{\tilde\beta}$, for every $\tilde\beta<1$; see Humphries \cite[Theorem 1.17]{Hu}. Humphries also showed that this can fail on balls whose radius is close to the Planck scale.
\item All holomorphic Hecke cusp forms are equidistributed on sets that shrink at a small power of the logarithm of the weight, by the effective QUE of Lester--Matom\"aki--Radziwi\l\l{} \cite{LMR}.
\item Almost all Hecke eigenforms are equidistributed on sets shrinking toward the cusp above the Planck scale, but not below it; see Nordentoft--Petridis--Risager \cite{NPR}.
\end{itemize}

On flat tori, where the geodesic flow is integrable, \eqref{eq:QEphase} fails in general. However, equidistribution in position space \eqref{eq:QEmanifold} and its small scale versions similar to Theorem \ref{thm:HanHR} are known. In particular, Hezari--Rivi\`ere \cite{HR2}, Lester--Rudnick \cite{LR}, and Granville--Wigman \cite{GW} proved that almost all eigenfunctions are equidistributed at polynomial scales, down to almost the Planck scale on two-dimensional tori.

On the other hand, equidistribution fails at the Planck scale $h$. In fact, on every compact manifold, real-valued eigenfunctions are not equidistributed, uniformly in the center, at small constant multiples of the Planck scale; see the author \cite{Han5} for a detailed discussion. Moreover, eigenfunctions with large values fail to be equidistributed at larger scales. For example, Humphries \cite[Theorem 1.14]{Hu} showed that Hecke--Maass eigenforms on the modular surface are not equidistributed at the scales $h|\log h|^A$, for every $A>0$, around Heegner points, where the eigenforms achieve large values. Bourgain constructed eigenfunctions on flat tori of dimension $d$ that are not equidistributed at the scales $o(h^{\frac1{d-1}})$; see Lester--Rudnick \cite[Theorem 4.1]{LR}.

The question is what happens between the logarithmic scales in Theorem \ref{thm:HanHR} and the Planck scale; see Figure \ref{fig:manifoldscales} and Problem \ref{prob:randomwave}.

\begin{figure}[!ht]
\centering
\begin{tikzpicture}
\draw[->] (0,0) -- (11.4,0) node[right] {$r$};
\foreach \x/\l in {0.8/{h},2.2/{h^{\tilde\beta}},4.8/{h^\beta},8.0/{|\log h|^{-\alpha}},10.7/{1}}
  {\draw (\x,0.12) -- (\x,-0.12) node[below] {$\l$};}
\draw (0.8,-1.15) node {\small Planck};
\draw (8.0,-1.15) node {\small logarithmic};
\draw (10.7,-1.15) node {\small macroscopic};
\draw[thick] (8.0,1) -- (10.7,1);
\draw[dashed] (0.8,1) -- (8.0,1);
\draw (0,1) node[left] {\small negatively curved};
\draw (9.35,1) node[above] {\small Theorem \ref{thm:HanHR}};
\draw (4.4,1) node[above] {\small open};
\draw[thick] (4.8,2) -- (10.7,2);
\draw (0,2) node[left] {\small almost all Hecke};
\draw (7.75,2) node[above] {\small Luo--Sarnak};
\draw[thick] (2.2,3) -- (10.7,3);
\draw (0,3) node[left] {\small all Hecke, conditional};
\draw (6.45,3) node[above] {\small Young};
\draw[thick,dotted] (0.8,0.3) -- (0.8,3.3) node[above] {\small \cite{Han5}};
\end{tikzpicture}
\caption{Equidistribution of eigenfunctions at the scale $r$. Solid: proved, in the weaker form of comparable mass for the negatively curved row. Dashed: open. Dotted: no real-valued eigenfunction is equidistributed at small constant multiples of the Planck scale. The two upper rows concern Hecke eigenfunctions on the modular surface: almost all of them by Luo--Sarnak \cite{LS}, and all of them, conditionally and for a fixed center, by Young \cite{Y}.}\label{fig:manifoldscales}
\end{figure}

\subsection{Random waves}\label{sec:introrandom}
The random wave model of Berry \cite{Berr} predicts that the eigenfunctions of chaotic systems behave like random waves, that is, random superpositions of waves of frequency $h^{-1}$. Random waves are equidistributed far below the logarithmic scales, down to large multiples of the Planck scale. On every compact manifold, the author and Tacy \cite{HT} showed that random combinations of eigenfunctions in a growing spectral window are equidistributed on each ball of radius $r$ with high probability if $r/h\to\infty$. De Courcy-Ireland \cite[Theorem 1.1]{CI} showed that, with high probability, random combinations of eigenfunctions in a slowly growing spectral window are equidistributed on all balls of radius $r$ simultaneously if
\[\frac r{h|\log h|}\to\infty\quad\text{for }d\ge3,\quad\text{and}\quad\frac r{h|\log h|^2}\to\infty\quad\text{for }d=2.\]
For random eigenbases, see the author \cite{Han2}.

Together with the random wave model, the results of \cite{HT,CI} predict that the typical eigenfunctions of negatively curved manifolds are equidistributed on balls down to scales just above the Planck scale. We propose the following problem.

\begin{prob}\label{prob:randomwave}
Let $\M$ be negatively curved.
\begin{enumerate}[(i)]
\item Equidistribution for a full density subsequence at small scales: Let $\beta\in(0,1)$. Determine whether for every eigenbasis $\{u_j\}_{j=1}^\infty$ and all radii $r_j\to0$ with $r_j\ge h_j^\beta$, there is $S\subset\N$ of full density with
\[\int_{B(x,r_j)}\left|u_j\right|^2\,d\Vol\approx\Vol\left(B\left(x,r_j\right)\right)\quad\text{for all }j\in S\text{ and }x\in\M.\]
\item Equidistribution for the whole eigenbasis at small scales: Determine the optimal scales $r_j\to0$ such that the estimate in (i) holds for all $j\in\N$ and $x\in\M$, for every eigenbasis.
\end{enumerate}
\end{prob}

In (ii), the optimal scales depend on the sup norms of the eigenfunctions. Indeed, by \cite{Han5}, on a manifold of dimension $d$, uniform equidistribution at the scales $r_j$ requires
\begin{equation}\label{eq:supnorm}
\left\|u_j\right\|_{L^\infty(\M)}\le C\left(\frac{r_j}{h_j}\right)^{\frac d2}.
\end{equation}
On some arithmetic hyperbolic $3$-manifolds, Rudnick--Sarnak \cite{RS} and Mili\'cevi\'c \cite{M} found eigenfunctions with sup norms at least $ch_j^{-\frac12}$. By \eqref{eq:supnorm}, they are not equidistributed at the scales $\approx h_j^{\frac23}$, which is much larger than the Planck scale $h_j$. Those in \cite{RS} have density zero, so they do not contradict (i). On the modular surface, the scales $h_j|\log h_j|^A$ do not suffice for any $A>0$, by Humphries \cite[Theorem 1.14]{Hu}.

For cat maps, we answer the analogue of Problem \ref{prob:randomwave} in Sections \ref{sec:introcat}--\ref{sec:introneg}: positively at polynomial scales under arithmetic conditions, and negatively beyond the logarithmic scales along the integers $N$ with short periods.

\subsection{Quantum cat maps}\label{sec:introcat}
The simplest chaotic dynamical systems are the hyperbolic automorphisms of the torus $\T^2=\R^2/\Z^2$, commonly called cat maps after Arnold--Avez \cite{AA}. Such a map is given by a matrix $M\in\SL(2,\Z)$ with $|\Tr M|>2$, which acts by
\[M\colon\T^2\to\T^2,\quad x\mapsto Mx\pmod{\Z^2}.\]
We fix such a matrix $M$.

The quantization of cat maps was introduced by Hannay--Berry \cite{HB}. It was studied by Degli Esposti \cite{DE}, Degli Esposti--Graffi--Isola \cite{DEGI}, Bouzouina--De Bi\`evre \cite{BouDB}, Zelditch \cite{Z2}, Kurlberg--Rudnick \cite{KR1,KR2}, and Faure--Nonnenmacher--De Bi\`evre \cite{FNDB}, among many others. Let $N\ge2$ be an integer\label{idx:N}. The quantization produces an $N$-dimensional Hilbert space $\Hc_N$ of quantum states on the phase space $\T^2$ and a unitary operator $\hat M$ on $\Hc_N$, the quantum cat map. The semiclassical limit is $N\to\infty$, and $N^{-1}\to0$ plays the role of the semiclassical parameter $h$ above. We recall the construction in Section \ref{sec:quantum}.

We study a state $\psi\in\Hc_N$ with $\|\psi\|=1$ through two probability measures. The position density $|\psi|^2$ is a probability measure on the position space $\T$, supported on $N$ equally spaced points. For a Borel set $E\subset\T$, we write $|\psi|^2(E)$ for the mass of $\psi$ on $E$, that is, the measure of $E$ with respect to $|\psi|^2$. 

In the phase space, we use the Husimi function\label{idx:Husimi} of $\psi$:
\begin{equation}\label{eq:Husimi}
\Hr_\psi(x)=N\left|\left\langle\psi,\Phi_x\right\rangle\right|^2\quad\text{for }x=(q,p)\in\T^2,
\end{equation}
where $\Phi_x\in\Hc_N$ is the coherent state at $x$. The integral of $\Hr_\psi$ over $\T^2$ is $1$. See Section \ref{sec:coherent} for the precise definitions. For a Borel set $E\subset\T^2$, we write $\Hr_\psi(E)$ for the integral of $\Hr_\psi$ over $E$:
\[\Hr_\psi(E)=\int_E\Hr_\psi(x)\,dx.\]

We call a normalized eigenvector of $\hat M$ an eigenstate\label{idx:eigenstate}. An eigenbasis\label{idx:eigenbasis} of $\Hc_N$ is then an orthonormal basis of eigenstates. 

For cat maps, the quantum ergodicity theorem was proved by Bouzouina--De Bi\`evre \cite{BouDB} and Zelditch \cite{Z2}. It asserts that for every eigenbasis $\{\psi_j\}_{j=1}^N$ of $\hat M$ in $\Hc_N$, there is a subset $S(N)\subset\{1,\ldots,N\}$ of full density\label{idx:fulldensityN}, that is,
\[\lim_{N\to\infty}\frac{\left|S(N)\right|}{N}=1,\]
such that
\begin{equation}\label{eq:QEcat}
\lim_{N\to\infty}\max_{j\in S(N)}\left|\int_{\T^2}a\Hr_{\psi_j}\,dx-\int_{\T^2}a\,dx\right|=0\quad\text{for all }a\in C(\T^2).
\end{equation}
We say that QUE holds if \eqref{eq:QEcat} holds with $S(N)=\{1,\ldots,N\}$ for every eigenbasis.

For cat maps, Kurlberg--Rudnick proved QUE along a set of $N$ of full density \cite{KR2} and for Hecke eigenbases \cite{KR1}, under congruence conditions on $M$; see Section \ref{sec:introQE}. Bourgain \cite{Bo} improved the first result, and Ostafe--Shparlinski--Voloch \cite{OSV} gave the explicit rate $N^{-\frac1{60}+o(1)}$ for all eigenstates along a set of $N$ of full density. Kurlberg--Ostafe--Rudnick--Shparlinski \cite{KORS} extended the first result to higher-dimensional cat maps. Degli Esposti--Graffi--Isola \cite{DEGI} proved QUE along a sparse set of primes under the Generalized Riemann Hypothesis.

On the other hand, Faure--Nonnenmacher--De Bi\`evre \cite{FNDB} showed that QUE fails along the integers $N$ for which the period of $M$ modulo $N$ is as short as possible. For these $N$, the eigenspaces of $\hat M$ are large, and Schwartz \cite[Theorem P.3.4]{Schw1} proved that random eigenbases satisfy QUE almost surely, also at some small scales.

In this paper, we study small scale quantum ergodicity in the quantum cat maps, a program initiated by the author in \cite{Han3,Han4}. We first make precise what equidistribution at small scales means for cat maps. Denote by\label{idx:balls} $B_1(q,r)\subset\T$ and $B_2(x,r)\subset\T^2$ the open balls of radius $r$ centered at $q\in\T$ and $x\in\T^2$. For $r\le\frac14$, their measures are $2r$ and $\pi r^2$.

\begin{defn}[Equidistribution at small scales]\label{idx:QEscale}
Let $\Gc\subset\N$ be infinite and $r_N>0$ with $r_N\to0$ as $N\to\infty$ in $\Gc$.
\begin{itemize}
\item We say that almost all eigenstates are equidistributed in phase space at the scales $r_N$ along $\Gc$ if there are $\delta_N>0$ with $\delta_N\to0$ as $N\to\infty$ in $\Gc$ such that, for every $N\in\Gc$ and every eigenbasis $\{\psi_j\}_{j=1}^N$ of $\hat M$ in $\Hc_N$, all but at most $\delta_NN$ of the $\psi_j$ satisfy
\[\sup_{x\in\T^2}\left|\frac{\Hr_{\psi_j}\left(B_2\left(x,r_N\right)\right)}{\pi r_N^2}-1\right|\le\delta_N.\]
\item We say that all eigenstates are equidistributed in phase space at the scales $r_N$ along $\Gc$ if the same holds with all the $\psi_j$ in place of all but at most $\delta_NN$ of them.
\item In position space, we use the same two terms with the left-hand side replaced by
\[\sup_{q\in\T}\left|\frac{\left|\psi_j\right|^2\left(B_1\left(q,r_N\right)\right)}{2r_N}-1\right|.\]
\end{itemize}
\end{defn}

The sequence $\delta_N$ bounds both the proportion of exceptional eigenstates and the discrepancy. It does not depend on the eigenbasis. Since $\delta_N\to0$, the exceptional eigenstates have density zero. The quantum ergodicity theorem concerns fixed radii $r$ independent of $N$. On the other hand, no state is equidistributed in position space at scales $r_N\le(2N)^{-1}$, since $|\psi|^2$ is supported on $N$ points with spacing $N^{-1}$, the Planck scale\label{idx:Planck} in position space (Section \ref{sec:states}). In phase space, the Husimi function does not resolve scales below the Planck scale $N^{-\frac12}$, since a coherent state is concentrated in a ball of radius $N^{-\frac12}$.

We prove three groups of results, stated precisely in Sections \ref{sec:introQE}--\ref{sec:introneg} and summarized in Figure \ref{fig:scales}.
\begin{enumerate}[(I)]
\item For every $N$, almost all eigenstates are equidistributed at the logarithmic scales $(\log N)^{-\alpha}$\label{idx:exponents}, for $\alpha<1$ in position space and $\alpha<\frac12$ in phase space (Theorem \ref{thm:QElog}). Along a set of $N$ of full density, this holds at polynomial scales $N^{-\beta}$ for some explicit $\beta>0$ (Theorem \ref{thm:density}). These are small scale versions of the quantum ergodicity theorem. This group also contains small scale versions of QUE\label{idx:QUEsmall}, in which every eigenstate is equidistributed at polynomial scales. They hold along the same set of $N$ (Theorem \ref{thm:density}) and for Hecke eigenbases (Theorem \ref{thm:Hecke}).
\item Every invariant probability measure $\mu$ is the semiclassical measure of quasimodes of order $O((\log N)^{-1})$, one for each $N$, that fail to be equidistributed at all scales from $4N^{-\frac18}$ up to a prescribed scale $r_N\to0$ (Theorem \ref{thm:quasimodes}).
\item There are orthonormal bases of quasimodes of order $O((\log N)^{-1})$, and eigenbases for $N$ with short periods, with the following property (Theorems \ref{thm:ONquasimodes} and \ref{thm:shortperiod}). A full density subset is equidistributed at every fixed scale, but not at the scales $(\log N)^{-\alpha}$ for $\alpha>1$ in position space and $\alpha>\frac12$ in phase space.
\end{enumerate}
So the critical scales\label{idx:critical} for the equidistribution of almost all eigenstates are $(\log N)^{-1}$ in position space and $(\log N)^{-\frac12}$ in phase space (Theorems \ref{thm:QElog} and \ref{thm:shortperiod}).

\begin{figure}[!ht]
\centering
\begin{tikzpicture}
\draw[->] (0,0) -- (11.4,0) node[right] {$r$};
\foreach \x/\l in {0.8/{N^{-1}},3.6/{N^{-\beta}},7.0/{(\log N)^{-1}},10.7/{1}}
  {\draw (\x,0.12) -- (\x,-0.12) node[below] {$\l$};}
\draw (0.8,-1.15) node {\small Planck};
\draw (7.0,-1.15) node {\small critical};
\draw (10.7,-1.15) node {\small macroscopic};
\draw[thick] (7.0,1) -- (10.7,1);
\draw (0,1) node[left] {\small every $N$};
\draw (8.85,1) node[above] {\small Theorem \ref{thm:QElog}};
\draw[thick] (7.0,2) -- (10.7,2);
\draw[thick,dotted] (0.8,2) -- (7.0,2);
\draw (0,2) node[left] {\small short periods};
\draw (8.85,2) node[above] {\small Theorem \ref{thm:QElog}};
\draw (3.9,2) node[above] {\small fails: Theorem \ref{thm:shortperiod}};
\draw[thick] (3.6,3) -- (10.7,3);
\draw[dashed] (0.8,3) -- (3.6,3);
\draw (0,3) node[left] {\small almost all $N$, Hecke};
\draw (7.15,3) node[above] {\small Theorems \ref{thm:density} and \ref{thm:Hecke}};
\end{tikzpicture}
\caption{Equidistribution in position space at the scale $r$. Solid: almost all eigenstates are equidistributed. Dotted: a full density subset of some eigenbasis is not equidistributed. Dashed: open. Below $(\log N)^{-1}$, the answer depends on $N$. The phase space picture is the same with $N^{-1}$ replaced by $N^{-\frac12}$ and $(\log N)^{-1}$ replaced by $(\log N)^{-\frac12}$.}\label{fig:scales}
\end{figure}

Groups (II) and (III) are complementary. Group (II) gives a single quasimode for each $N$, with an arbitrary invariant semiclassical measure, that is not equidistributed at any scale between $4N^{-\frac18}$ and $r_N$. In an orthonormal basis of quasimodes whose order tends to $0$, however, almost all elements are equidistributed at every fixed scale, that is, their semiclassical measure is the Lebesgue measure $\Leb$ (Proposition \ref{prop:ONQE}). Group (III) shows that a full density subset can still fail to be equidistributed beyond the critical logarithmic scales.

Each group follows from a different mechanism.

Group (I) uses that the stable and unstable directions of $M$ are badly approximable: non-zero integer vectors stay away from them in a quantitative, hyperbolic sense (Lemma \ref{lemma:badapprox}). Averaging along the dynamics up to half the Ehrenfest time, we reduce the moments of an observable over an eigenbasis to counting vanishing sums $M^{t_1}\xi_1+\cdots+M^{t_{2m}}\xi_{2m}$ of integer vectors. Bad approximability forces the times in such a sum to form clusters (Lemma \ref{lemma:cluster}). This gives moments of Gaussian type of every order, and moments of high order give the sharp exponents. The polynomial scales use instead the arithmetic of $N$ or of the Hecke operators.

Group (II) uses that a periodic orbit of prime period $T$ cannot be equidistributed at scales below $T^{-\frac12}$ (Lemma \ref{lemma:orbitscale}). Indeed, if $Tr^2$ is small, then some balls of radius $r$ miss the orbit, while a ball around a point of the orbit has mass about $1/T$, much more than its area. We superpose evolved coherent states along such an orbit, with cosine weights over times up to a quarter of the Ehrenfest time. This gives a quasimode of order $O((\log N)^{-1})$ whose Husimi function is concentrated near the orbit; see Figure \ref{fig:quasimode}.

\begin{figure}[!ht]
\centering
\includegraphics{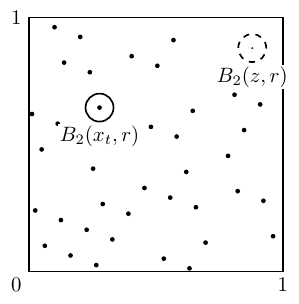}
\caption{The key idea of Theorem \ref{thm:quasimodes}: the periodic orbit of prime period $T=39$ of the Arnold cat map through $(\frac1{79},\frac{49}{79})$ (dots), a ball $B_2(x_t,r)$ centered at a point of the orbit (solid), and a ball $B_2(z,r)$ that misses the orbit (dashed).}\label{fig:quasimode}
\end{figure}

Group (III) uses that a subspace of dimension $d$ contains a unit vector with squared overlap at least $d/N$ with some coherent state, namely the normalized projection of that coherent state onto the subspace. Repeating this in the orthogonal complement of the vectors already chosen gives an orthonormal basis whose $k$-th vector has squared overlap at least $(d-k+1)/N$ with a coherent state (Proposition \ref{prop:coherentbasis}). For $N$ with short periods, $\hat M$ has at most a constant times $\log N$ distinct eigenvalues. So, for each $\eta\in(0,1)$, all but $\eta N$ elements of the resulting eigenbasis put mass at least a constant times $\eta(\log N)^{-1}$ near a point. For $\eta\to0$ slowly, this is much larger than the area of a ball of radius $(\log N)^{-\alpha}$ with $\alpha>\frac12$.

The failure for $N$ with short periods is caused by large spectral degeneracies, while the polynomial scales along a set of $N$ of full density use arithmetic information. Hence, an improvement of the logarithmic scales on negatively curved manifolds has to overcome such degeneracies, or use properties of the Laplacian that cat maps do not have. On a negatively curved manifold of dimension $d$, the best known bound on the multiplicity of the eigenvalue $h^{-2}$, due to B\'erard \cite{Bera}, is $Ch^{-(d-1)}/|\log h|$. So it does not exclude degeneracies analogous to the multiplicities of order $N/\log N$ for $N$ with short periods.

\subsection{Equidistribution at small scales}\label{sec:introQE}
The first main theorem holds for every $N$ and every eigenbasis. It does not use any arithmetic information on $N$.

\begin{thm}[Equidistribution at logarithmic scales]\label{thm:QElog}\leavevmode
\begin{enumerate}[(i)]
\item For each $\alpha\in(0,1)$, almost all eigenstates are equidistributed in position space at the scales $(\log N)^{-\alpha}$ along $\N$.
\item For each $\alpha\in(0,\frac12)$, almost all eigenstates are equidistributed in phase space at the scales $(\log N)^{-\alpha}$ along $\N$.
\end{enumerate}
\end{thm}

Theorem \ref{thm:shortperiod} shows that the exponents $1$ and $\frac12$ in Theorem \ref{thm:QElog} cannot be improved. In comparison, Chang--Zelditch \cite[Theorem 2]{CZ} studied the Toeplitz quantization of symplectic maps with polynomial decay of correlations on K\"ahler manifolds, which include hyperbolic toral automorphisms. They proved that almost all eigensections have comparable mass on all balls of radius $(\log N)^{-\alpha}$, for $\alpha<\frac16$ in complex dimension one.

For special values of $N$ or special eigenbases, the logarithmic scales can be improved to polynomial scales. Here, we use the results of Kurlberg--Rudnick \cite{KR1,KR2} on QUE in cat maps.

Denote by $P(N)$ the period\label{idx:period} of $M$ modulo $N$, that is, the smallest $t\in\N$ with $M^t=\Id\pmod N$. Kurlberg--Rudnick \cite{KR2} showed that there are $c=c(M)>0$ and $\nu_0=\nu_0(M)>0$ such that
\[P(N)\ge cN^{\frac12}\exp\left((\log N)^{\nu_0}\right)\]
for $N$ in a set of full density, and that QUE holds along such a set. The next theorem gives polynomial scales along a set of full density on which this lower bound holds. So the periods are long on this set, in contrast with the short periods in Theorem \ref{thm:shortperiod}.

\begin{thm}[Polynomial scales for a set of $N$ of full density]\label{thm:density}
Suppose that $M$ is congruent modulo $2$ to $\Id$ or to the matrix that exchanges the two coordinates. There are a set $\Gc\subset\N$ of full density and $\beta_0=\beta_0(M)>0$ with the following properties.
\begin{enumerate}[(i)]
\item For each $\beta\in(0,\frac1{12})$, almost all eigenstates are equidistributed in position space at the scales $N^{-\beta}$ along $\Gc$.
\item For each $\beta\in(0,\frac1{16})$, almost all eigenstates are equidistributed in phase space at the scales $N^{-\beta}$ along $\Gc$.
\item All eigenstates are equidistributed at the scales $N^{-\beta_0}$ along $\Gc$, both in position space and in phase space.
\end{enumerate}
\end{thm}

We remark that (iii) of Theorem \ref{thm:density} is a small scale version of the QUE theorem of Kurlberg--Rudnick \cite{KR2}: every eigenstate is equidistributed at the polynomial scales $N^{-\beta_0}$. It follows from the quantitative QUE bound of Bourgain \cite[Theorem 3]{Bo}, in the explicit form of Ostafe--Shparlinski--Voloch \cite[Section 8.2]{OSV}. The parity condition on $M$ is the standing assumption of \cite{KR2}.

Kurlberg--Rudnick \cite{KR1} introduced a commutative group of unitary operators on $\Hc_N$ that commute with $\hat M$, the Hecke operators, in analogy with the Hecke theory on the modular surface. We say that an orthonormal basis of $\Hc_N$ is a Hecke eigenbasis\label{idx:Hecke} if it consists of joint eigenvectors of $\hat M$ and all Hecke operators; see Section \ref{sec:poly}. The next theorem gives polynomial scales for Hecke eigenbases, along all $N$.

\begin{thm}[Polynomial scales for Hecke eigenbases]\label{thm:Hecke}
Suppose that $M=\Id\pmod4$.
\begin{enumerate}[(i)]
\item For each $\beta\in(0,\frac1{10})$, almost all elements of Hecke eigenbases are equidistributed in position space at the scales $N^{-\beta}$ along $\N$.
\item For each $\beta\in(0,\frac1{12})$, almost all elements of Hecke eigenbases are equidistributed in phase space at the scales $N^{-\beta}$ along $\N$.
\item For each $\beta\in(0,\frac1{20})$, all elements of Hecke eigenbases are equidistributed in position space at the scales $N^{-\beta}$ along $\N$. For each $\beta\in(0,\frac1{24})$, the same holds in phase space.
\end{enumerate}
\end{thm}

\subsection{Quasimodes that are not equidistributed at small scales}\label{sec:introquasi}
We say that a state $\psi\in\Hc_N$ with $\|\psi\|=1$ is a quasimode of order $\ve$\label{idx:quasimode} if
\[\left\|\left(\hat M-e^{i\w}\right)\psi\right\|\le\ve\quad\text{for some }\w\in\R.\]
In the special case when $\ve=0$, $\psi$ is an eigenstate of $\hat M$ with eigenvalue $e^{i\w}$.

We say that a sequence of states $\psi_N\in\Hc_N$ with $\|\psi_N\|=1$ has semiclassical measure\label{idx:semiclassical} $\mu$ if $\Hr_{\psi_N}\,dx$ converges weakly to $\mu$ as $N\to\infty$. If the order of the quasimodes tends to $0$, then every semiclassical measure is invariant under $M$ (Lemma \ref{lemma:invariance}). The second main theorem is a converse with quasimodes of order $O((\log N)^{-1})$ that can, moreover, be chosen not to be equidistributed at small scales.

\begin{thm}[Quasimodes that are not equidistributed at small scales]\label{thm:quasimodes}
Let $\mu$ be an $M$-invariant probability measure on $\T^2$. Let $r_N>0$ with $r_N\to0$ as $N\to\infty$. There are $C=C(M)>0$, $\delta_N>0$ with $\delta_N\to0$ as $N\to\infty$, and for each $N$ a quasimode $\psi_N\in\Hc_N$ of order $C(\log N)^{-1}$, such that the following hold.
\begin{enumerate}[(i)]
\item The semiclassical measure of $\{\psi_N\}$ is $\mu$.
\item For all $r\in[4N^{-\frac18},r_N]$,
\begin{align*}
\sup_{q\in\T}\frac{\left|\psi_N\right|^2\left(B_1(q,r)\right)}{2r}\ge\frac1{\delta_N},&\quad\inf_{q\in\T}\frac{\left|\psi_N\right|^2\left(B_1(q,r)\right)}{2r}\le\delta_N,\\
\sup_{x\in\T^2}\frac{\Hr_{\psi_N}\left(B_2(x,r)\right)}{\pi r^2}\ge\frac1{\delta_N},&\quad\inf_{x\in\T^2}\frac{\Hr_{\psi_N}\left(B_2(x,r)\right)}{\pi r^2}\le\delta_N.
\end{align*}
\end{enumerate}
\end{thm}

So the quasimodes in Theorem \ref{thm:quasimodes} are not equidistributed at any scale between $4N^{-\frac18}$ and a prescribed scale $r_N$, which may tend to $0$ arbitrarily slowly. The lower end $4N^{-\frac18}$ can be lowered to $4N^{-\frac12+\beta}$ for each $\beta\in(0,\frac14]$, at the cost of the order $C(\beta\log N)^{-1}$; see the remark on smaller scales in Section \ref{sec:negativequasi}. The most important case is $\mu=\Leb$, in which the quasimodes are equidistributed at every fixed scale.

\begin{cor}[Quantum ergodic quasimodes]\label{cor:QEquasimodes}
There are $C=C(M)>0$ and, for each $N$, a quasimode $\psi_N\in\Hc_N$ of order $C(\log N)^{-1}$ such that $\{\psi_N\}$ is equidistributed at every fixed scale but not at the scales $(\log N)^{-\frac14}$.
\end{cor}
\begin{proof}
Apply Theorem \ref{thm:quasimodes} with $\mu=\Leb$ and $r_N=(\log N)^{-\frac14}$.
\end{proof}

Theorem \ref{thm:quasimodes} does not hold for eigenstates, whose semiclassical measures satisfy additional restrictions. The weight of their pure point part is at most $\frac12$, by Faure--Nonnenmacher \cite{FN}, improving Bonechi--De Bi\`evre \cite{BonDB}. In particular, the orbit measure of a periodic orbit is not a semiclassical measure of eigenstates. Brooks \cite{Br1} and Rivi\`ere \cite{R} proved restrictions on their entropy. Moreover, they have full support: see Schwartz \cite{Schw2} under the parity condition in Theorem \ref{thm:density}, and Dyatlov--J\'ez\'equel \cite{DyJe} for every $M$. The estimate \cite[Theorem 2]{DyJe} also applies to quasimodes of order $o((\log N)^{-1})$. Hence, for measures $\mu$ without full support, the order $C(\log N)^{-1}$ in Theorem \ref{thm:quasimodes} is sharp up to the constant $C$. So Theorem \ref{thm:quasimodes} shows a sharp difference between eigenstates and quasimodes of order $O((\log N)^{-1})$. For cat maps on higher-dimensional tori, see also Kim \cite{K}. Finally, Theorem \ref{thm:quasimodes} gives one quasimode for each $N$, which has zero density in every orthonormal basis.

On negatively curved manifolds, the analogues of quasimodes of order $O((\log N)^{-1})$ are the logarithmic quasimodes, of order $O(h/|\log h|)$ for the semiclassical Laplacian $-h^2\Delta$. The semiclassical measures of eigenfunctions satisfy the entropy bounds of Anantharaman \cite{Ana} and Anantharaman--Nonnenmacher \cite{AN}. On hyperbolic surfaces, they have full support by Dyatlov--Jin \cite{DyJi}, and so do the semiclassical measures of quasimodes of order $o(h/|\log h|)$. On Anosov manifolds, Eswarathasan \cite{E} proved lower bounds on the entropy of the semiclassical measures of quasimodes of order $ch/|\log h|$ for small $c>0$. In contrast, Brooks \cite{Br2}, Eswarathasan--Nonnenmacher \cite{EN}, and Eswara\-thasan--Silberman \cite{ES} constructed logarithmic quasimodes that concentrate near a closed geodesic, in the spirit of Faure--Nonnenmacher--De Bi\`evre \cite{FNDB}. Theorem \ref{thm:quasimodes} is the analogue for cat maps, with every invariant measure as a semiclassical measure, not only the measures on periodic orbits, and with the failure of equidistribution at small scales.

\subsection{Full density quasimodes that are not equidistributed at logarithmic scales}\label{sec:introneg}
The third group of results concerns full density subsets of orthonormal bases. The question is whether a full density subset of an orthonormal basis of quasimodes, or of eigenstates, can fail to be equidistributed at logarithmic scales. The next theorem shows that this happens for quasimodes of order $\pi(\log N)^{-1}$ for every $N$.

\begin{thm}[Concentrated orthonormal quasimodes]\label{thm:ONquasimodes}
For each $N$, there is an orthonormal basis $\{\psi_j\}_{j=1}^N$ of $\Hc_N$ consisting of quasimodes of order $\pi(\log N)^{-1}$ with the following property. Let $\alpha>\frac12$ and $s_N=(\log N)^{-\alpha}$. There are subsets $S(N)\subset\{1,\ldots,N\}$ of full density and $\delta_N>0$ with $\delta_N\to0$ as $N\to\infty$ such that the following hold.
\begin{enumerate}[(i)]
\item Equidistribution at every fixed scale: \eqref{eq:QEcat} holds.
\item No equidistribution at any scale $r\le s_N$:
\[\sup_{x\in\T^2}\frac{\Hr_{\psi_j}\left(B_2(x,r)\right)}{\pi r^2}\ge\frac1{\delta_N}\quad\text{for all }j\in S(N)\text{ and }r\in(0,s_N].\]
\end{enumerate}
The same holds in position space with $\alpha>1$, with the balls $B_1(q,r)$ and the measures $|\psi_j|^2$ in place of the balls $B_2(x,r)$ and the Husimi functions.
\end{thm}

The key point of Theorem \ref{thm:ONquasimodes} is the full density: the same full density subset is equidistributed at every fixed scale, but not at any scale $r\le s_N$. (i) of Theorem \ref{thm:ONquasimodes} holds for every orthonormal basis of quasimodes whose order tends to $0$ (Proposition \ref{prop:ONQE}).

For eigenstates, the order of the quasimodes is $0$, and we need additional information on $N$. We use the integers $N$ for which the period $P(N)$ is short, that is, comparable to $\log N$. For $C>0$, let\label{idx:shortperiod}
\[\Pc_C=\left\{N\ge2:P(N)\le C\log N\right\}.\]
The set $\Pc_C$ is infinite for each $C>2\lambda^{-1}$, where $\lambda$ is the Lyapunov exponent of $M$ (Proposition \ref{prop:periods}).

\begin{thm}[Eigenbases for short periods]\label{thm:shortperiod}
Let $C>2\lambda^{-1}$. For each $N\in\Pc_C$, there is an eigenbasis $\{\psi_j\}_{j=1}^N$ of $\hat M$ in $\Hc_N$ with the following property. Let $\alpha_1\in(0,\frac12)$ and $\alpha_2>\frac12$. Write 
\[r_N=(\log N)^{-\alpha_1}\quad\text{and}\quad s_N=(\log N)^{-\alpha_2}.\] 
There are subsets $S(N)\subset\{1,\ldots,N\}$ of full density and $\delta_N>0$ with $\delta_N\to0$, both as $N\to\infty$ in $\Pc_C$, such that the following hold.
\begin{enumerate}[(i)]
\item Equidistribution at every fixed scale: \eqref{eq:QEcat} holds as $N\to\infty$ in $\Pc_C$.
\item Equidistribution at the scales $r_N$:
\[\sup_{x\in\T^2}\left|\frac{\Hr_{\psi_j}\left(B_2\left(x,r_N\right)\right)}{\pi r_N^2}-1\right|\le\delta_N\quad\text{for all }j\in S(N).\]
\item No equidistribution at any scale $r\le s_N$:
\[\sup_{x\in\T^2}\frac{\Hr_{\psi_j}\left(B_2(x,r)\right)}{\pi r^2}\ge\frac1{\delta_N}\quad\text{for all }j\in S(N)\text{ and }r\in(0,s_N].\]
\end{enumerate}
The same holds in position space with $\alpha_1\in(0,1)$ and $\alpha_2>1$, with the balls $B_1(q,r)$ and the measures $|\psi_j|^2$ in place of the balls $B_2(x,r)$ and the Husimi functions.
\end{thm}

In fact, (i) and (ii) of Theorem \ref{thm:shortperiod} hold for every eigenbasis, by the quantum ergodicity theorem and Theorem \ref{thm:QElog}. The point is that (iii) holds for the same full density subset: each of these eigenstates is equidistributed at every fixed scale and at the scales $r_N$, but, for each $r\le s_N$, puts much more mass than the area on some ball of radius $r$. So even for a full density subset of eigenstates, equidistribution at every fixed scale does not extend beyond the critical logarithmic scales.

This failure is different from the failure of QUE found by Faure--Nonnenmacher--De Bi\`evre \cite{FNDB}. For the integers $N$ whose periods are close to the lower bound in (i) of Proposition \ref{prop:periods}, they constructed eigenstates with semiclassical measure
\[\frac12\mu_\gamma+\frac12\Leb,\]
where $\mu_\gamma$ is the orbit measure of a periodic orbit $\gamma$ (Section \ref{sec:periodic}). Since their semiclassical measure is not $\Leb$, they form a zero density subset of an eigenbasis, by the quantum ergodicity theorem. In contrast, the eigenstates in Theorem \ref{thm:shortperiod} form a full density subset and fail to be equidistributed only at small scales.

Schubert \cite{Schu} showed that, for cat maps, the logarithmic rate of decay of the quantum variance of a fixed observable is sharp for suitable eigenbases. Theorem \ref{thm:shortperiod} concerns instead shrinking scales.

\subsection*{Organization of the paper}
Sections \ref{sec:classical} and \ref{sec:quantum} review the classical dynamics and the quantization. Section \ref{sec:positive} proves Theorems \ref{thm:QElog}, \ref{thm:density}, and \ref{thm:Hecke}. Section \ref{sec:negativequasi} proves Theorem \ref{thm:quasimodes}. Section \ref{sec:negativeeigen} proves Theorems \ref{thm:ONquasimodes} and \ref{thm:shortperiod}. The appendices prove the lemmas on translation operators and on coherent states. A list of symbols is at the end of the paper.

\section{Hyperbolic toral automorphisms}\label{sec:classical}
In this section, we review the classical dynamics of cat maps. We collect the properties of the stable and unstable directions, of the periodic orbits, and of the periods modulo $N$ that are used in the later sections. See Katok--Hasselblatt \cite{KH} for the background in hyperbolic dynamics.

\subsection{Stable and unstable directions}
Recall that $M\in\SL(2,\Z)$ satisfies $|\Tr M|>2$. We use $M$ both for the matrix and for the map $x\mapsto Mx\pmod{\Z^2}$ on $\T^2$, which we call a cat map\label{idx:catmap}. Write points of $\R^2$ and $\T^2$ as $x=(q,p)$, where $q$ is the position variable and $p$ is the momentum variable. Denote by $\Leb$\label{idx:Leb} the Lebesgue measure on $\T^2$, which is invariant under $M$.

Denote by $\sigma=\sign(\Tr M)$ the sign of the trace. The eigenvalues of $M$ are $\sigma e^{\lambda}$ and $\sigma e^{-\lambda}$, where $\lambda>0$ is determined by
\[|\Tr M|=e^\lambda+e^{-\lambda}.\]
We call $\lambda$ the Lyapunov exponent\label{idx:Lyapunov} of $M$. Since $|\Tr M|\ge3$,
\[e^\lambda\ge\frac{3+\sqrt5}2\quad\text{and}\quad\lambda>\frac12.\]
Let $v_+$ and $v_-$ be unit eigenvectors of $M$ for the eigenvalues $\sigma e^{\lambda}$ and $\sigma e^{-\lambda}$. They span the unstable and the stable direction\label{idx:stable} of $M$. For $x\in\R^2$, write
\[x=x_+v_++x_-v_-\quad\text{with }x_\pm\in\R.\]
Then
\begin{equation}\label{eq:Mt}
M^tx=\sigma^t\left(e^{\lambda t}x_+v_++e^{-\lambda t}x_-v_-\right)\quad\text{for }t\in\Z.
\end{equation}
Denote by $C_M$\label{idx:CM} the maximum of $|x_+|+|x_-|$ over the unit vectors $x\in\R^2$:
\[C_M=\sup_{x\in\R^2,|x|=1}\left\{\left|x_+\right|+\left|x_-\right|\right\}.\]
Since $v_+$ and $v_-$ are unit vectors, $C_M\ge1$. By \eqref{eq:Mt}, the operator norm of $M^t$ satisfies
\begin{equation}\label{eq:normMt}
e^{\lambda|t|}\le\left\|M^t\right\|\le C_Me^{\lambda|t|}\quad\text{for }t\in\Z.
\end{equation}
Here, the lower bound holds since $\sigma^te^{\lambda|t|}$ is an eigenvalue of $M^t$.

The symplectic product\label{idx:symplectic} of $u=(u_1,u_2)$ and $v=(v_1,v_2)$ in $\R^2$ is
\[u\wedge v=u_2v_1-u_1v_2.\]
It is the determinant of the matrix with columns $v$ and $u$. Since $\det M=1$,
\[Mu\wedge Mv=u\wedge v\quad\text{for }u,v\in\R^2.\]

\begin{ex}[Arnold cat map]
The Arnold cat map is given by
\[M=\begin{pmatrix}2&1\\1&1\end{pmatrix}.\]
Its Lyapunov exponent is
\[\lambda=\log\frac{3+\sqrt5}2.\]
Since $M$ is symmetric, $v_+$ and $v_-$ are orthogonal. Their slopes are
\[\frac{\sqrt5-1}2\quad\text{and}\quad-\frac{\sqrt5+1}2.\]
See Figure \ref{fig:catmap}.
\end{ex}

\begin{figure}[!ht]
\centering
\begin{tikzpicture}
\draw[step=1.9,thin,gray!50] (-2.28,-2.28) grid (6.46,4.56);
\fill[gray!25] (0,0) -- (3.8,1.9) -- (5.7,3.8) -- (1.9,1.9) -- cycle;
\draw[thick] (0,0) -- (3.8,1.9) -- (5.7,3.8) -- (1.9,1.9) -- cycle;
\draw[thick] (0,0) rectangle (1.9,1.9);
\draw[dashed] (-2.28,-1.4098) -- (6.46,3.9919) node[above left] {$\R v_+$};
\draw[dashed] (1.4098,-2.28) -- (-1.4098,2.28) node[below left] {$\R v_-$};
\draw[->,thick] (0,0) -- (1.6169,0.9994);
\draw (1.6169,0.9994) node[above left] {$v_+$};
\draw[->,thick] (0,0) -- (0.9994,-1.6169) node[right] {$v_-$};
\fill (0,0) circle (0.057) node[below left] {$0$};
\draw (0.57,1.558) node {$[0,1]^2$};
\draw (4.75,1.33) node {$M[0,1]^2$};
\draw[->] (1.14,2.185) to[bend left=25] (3.04,3.325);
\draw (1.9,3.04) node {$M$};
\end{tikzpicture}
\caption{The Arnold cat map. The unit square $[0,1]^2$ is mapped to the shaded parallelogram $M[0,1]^2$, which is then reduced modulo $\Z^2$. The dashed lines are the unstable and stable lines $\R v_+$ and $\R v_-$ through $0$.}\label{fig:catmap}
\end{figure}

The eigenvalues of $M$ are irrational, since $(\Tr M)^2-4$ is not a perfect square for $|\Tr M|>2$. So the stable and unstable directions contain no non-zero integer vectors. The next lemma is a quantitative version of this fact: integer vectors stay away from the stable and unstable directions in a hyperbolic sense. It was proved by Faure--Nonnenmacher--De Bi\`evre \cite[Lemma 2]{FNDB}; we give a short proof with an explicit constant. This is the only Diophantine property of $M$ that we use.

\begin{lemma}[Bad approximability]\label{lemma:badapprox}
Let
\[c_M=\frac{1}{\left(e^\lambda-e^{-\lambda}\right)\left|v_+\wedge v_-\right|}.\]
Then
\[\left|n_+n_-\right|\ge c_M\quad\text{for all }n\in\Z^2\setminus\{0\}.\]
\end{lemma}
\begin{proof}
Let $n\in\Z^2\setminus\{0\}$. Since $Mn\in\Z^2$, the symplectic product $n\wedge Mn$ is an integer. It is non-zero, since otherwise $Mn$ is parallel to $n$, and $n$ is an integer eigenvector of $M$. By \eqref{eq:Mt} with $t=1$,
\[1\le|n\wedge Mn|=\left|n_+n_-\right|\left(e^\lambda-e^{-\lambda}\right)\left|v_+\wedge v_-\right|.\qedhere\]
\end{proof}

For the Arnold cat map, $c_M=\frac1{\sqrt5}$, and equality holds for $n=(1,0)$. Figure \ref{fig:badapprox} shows the region
\[\left\{x\in\R^2:\left|x_+x_-\right|<c_M\right\}.\]
By Lemma \ref{lemma:badapprox}, it contains no non-zero integer points.

\begin{figure}[!ht]
\centering
\includegraphics{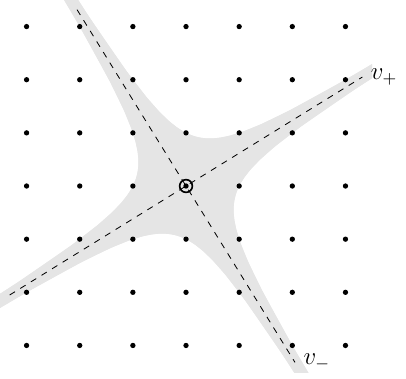}
\caption{For the Arnold cat map, the shaded region around the stable and unstable lines, where $|x_+x_-|<c_M$, contains no integer points other than $0$ (Lemma \ref{lemma:badapprox}). Some integer points lie on its boundary.}\label{fig:badapprox}
\end{figure}

\subsection{Periodic orbits}\label{sec:periodic}
We say that a point $x\in\T^2$ is periodic\label{idx:periodic} if
\[M^Tx=x\quad\text{for some }T\in\N.\]
The smallest such $T$ is the prime period of $x$, and
\[\gamma=\left\{x,Mx,\ldots,M^{T-1}x\right\}\]
is a periodic orbit of prime period $T$. Its orbit measure\label{idx:orbitmeasure} is the $M$-invariant probability measure
\[\mu_\gamma=\frac1T\sum_{t=0}^{T-1}\delta_{M^tx}.\]
Here and below, $M^T$ is the $T$-th power of $M$. Denote by $\dist(x,y)$ the distance of $x,y\in\T^2$.

\begin{lemma}[Fixed points]\label{lemma:fixed}
Let $T\in\N$. Then
\[\dist(x,y)\ge\frac{e^{-\lambda T}}{2C_M}\quad\text{for all distinct fixed points }x,y\in\T^2\text{ of }M^T.\]
In particular, $M^T$ has finitely many fixed points.
\end{lemma}
\begin{proof}
Let $x,y\in\T^2$ be distinct fixed points of $M^T$. Choose lifts of $x$ and $y$ to $\R^2$ whose difference $z$ satisfies $|z|=\dist(x,y)$. Then $z\notin\Z^2$ and $(M^T-\Id)z\in\Z^2$. Since the eigenvalues $\sigma^Te^{\pm\lambda T}$ of $M^T$ are not $1$, the matrix $M^T-\Id$ is invertible. So $(M^T-\Id)z$ is a non-zero integer vector. By \eqref{eq:normMt},
\[1\le\left|\left(M^T-\Id\right)z\right|\le\left(C_Me^{\lambda T}+1\right)|z|\le2C_Me^{\lambda T}|z|.\]
\end{proof}

We use the following result of Sigmund \cite{Si}, which holds for transitive Anosov diffeomorphisms, such as the cat maps.

\begin{thm}[Sigmund]\label{thm:Sigmund}
The orbit measures of periodic orbits are dense in the set of $M$-invariant probability measures on $\T^2$ with the weak topology.
\end{thm}

Let $\mu$ be an $M$-invariant probability measure. By Theorem \ref{thm:Sigmund}, there are periodic orbits $\gamma_j$ with prime periods $T_j$ such that $\mu_{\gamma_j}\to\mu$ weakly. If $\mu$ is not the orbit measure of a periodic orbit, then $T_j\to\infty$. Indeed, there are only finitely many periodic orbits with prime period at most a given number, by Lemma \ref{lemma:fixed}. Figure \ref{fig:orbits} shows a short periodic orbit and an equidistributed long one.

\begin{figure}[!ht]
\centering
\includegraphics{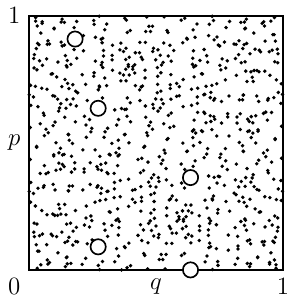}
\caption{Two periodic orbits of the Arnold cat map in $[0,1)^2$: one of prime period $5$ (circles) and one of prime period $788$ (dots). The long orbit is spread uniformly over the torus, while the short one consists of five points.}\label{fig:orbits}
\end{figure}

Even a long periodic orbit is not equidistributed at small scales. Each point of a periodic orbit $\gamma$ of prime period $T$ carries the mass $1/T$ of $\mu_\gamma$. By the next lemma, some ball of radius comparable to $T^{-\frac12}$ contains no point of $\gamma$; see Figure \ref{fig:quasimode}. So $\mu_\gamma$ is not equidistributed at scales below $T^{-\frac12}$. In the position variable, the same holds at scales below $T^{-1}$.

\begin{lemma}[Periodic orbits at small scales]\label{lemma:orbitscale}
Let $\gamma$ be a periodic orbit of prime period $T$. Let $r>0$.
\begin{enumerate}[(i)]
\item If $4\pi Tr^2<1$, then there is $z\in\T^2$ such that the ball $B_2(z,2r)$ contains no point of $\gamma$.
\item If $4Tr<1$, then there is $z\in\T$ such that the strip $B_1(z,2r)\times\T$ contains no point of $\gamma$.
\end{enumerate}
\end{lemma}
\begin{proof}
(i) The union of the balls $B_2(x,2r)$ with $x\in\gamma$ has area at most $4\pi Tr^2<1$. Choose $z\in\T^2$ outside this union. Then $\dist(z,x)\ge2r$ for all $x\in\gamma$.

(ii) The union of the balls $B_1(q,2r)$ around the position coordinates $q$ of the points of $\gamma$ has length at most $4Tr<1$. Choose $z\in\T$ outside this union.
\end{proof}

\subsection{Periods modulo \texorpdfstring{$N$}{N}}\label{sec:periodsN}
The matrix $M$ is invertible modulo $N$. Recall that $P(N)$ denotes its period modulo $N$ (Section \ref{sec:introQE}). The periods control the quantum dynamics through the quantum periods in Section \ref{sec:qperiods}. Kurlberg--Rudnick \cite[Theorem 17]{KR2} proved that the period is large for almost all $N$. The next proposition shows that it is always at least logarithmic, and logarithmic for infinitely many $N$.

\begin{prop}[Periods]\label{prop:periods}\leavevmode
\begin{enumerate}[(i)]
\item For all $N$,
\[P(N)\ge\frac{2\log N-\log4}{\lambda}.\]
\item There are $\nu_0=\nu_0(M)>0$, $c=c(M)>0$, and a set $\Gc_0\subset\N$ of full density such that
\[P(N)\ge cN^{\frac12}\exp\left((\log N)^{\nu_0}\right)\quad\text{for all }N\in\Gc_0.\]
\item Let $a_k$ be the integers given by
\[a_0=0,\quad a_1=1,\quad\text{and}\quad a_{k+1}=(\Tr M)a_k-a_{k-1}\quad\text{for }k\in\N.\]
Then $N_k=|a_k|$ satisfies $N_k\to\infty$ and
\[P\left(N_k\right)\le2k\le\frac{2\log N_k+2\log\left(2e^\lambda\right)}{\lambda}\quad\text{for }k\ge2.\]
\end{enumerate}
\end{prop}
\begin{proof}
(i) Let $P=P(N)$. Since all entries of $M^P-\Id$ are divisible by $N$, its determinant $2-\Tr M^P$ is divisible by $N^2$. It is non-zero, since $|\Tr M^P|=e^{\lambda P}+e^{-\lambda P}>2$. So 
\[N^2\le\left|\Tr M^P\right|+2\le4e^{\lambda P}.\]

(ii) This is \cite[Theorem 17]{KR2}.

(iii) By the Cayley--Hamilton theorem,
\[M^2=(\Tr M)M-\Id.\]
By induction on $k$,
\[M^k=a_kM-a_{k-1}\Id\quad\text{for }k\in\N.\]
Since $M^{-1}$ has the same trace and determinant as $M$, the same holds with $M^{-1}$ in place of $M$. Subtracting, we obtain
\[M^k-M^{-k}=a_k\left(M-M^{-1}\right).\]
Hence,
\[M^{2k}-\Id=M^k\left(M^k-M^{-k}\right)=a_kM^k\left(M-M^{-1}\right)=0\pmod{a_k}.\]
So $P(N_k)\le2k$. The recursion gives
\[a_k=\frac{\left(\sigma e^\lambda\right)^k-\left(\sigma e^{-\lambda}\right)^k}{\sigma e^\lambda-\sigma e^{-\lambda}}.\]
Since $e^{2\lambda}>2$, this gives
\[N_k=\frac{e^{\lambda k}-e^{-\lambda k}}{e^\lambda-e^{-\lambda}}\ge\frac{e^{\lambda k}}{2e^\lambda}\quad\text{for }k\ge1.\]
So $N_k\to\infty$. Taking logarithms gives the bound on $2k$. This proves (iii).
\end{proof}

For the Arnold cat map, $P(N_k)=2k$ for $k\ge2$; see Figure \ref{fig:periods}.

\begin{figure}[!ht]
\centering
\includegraphics{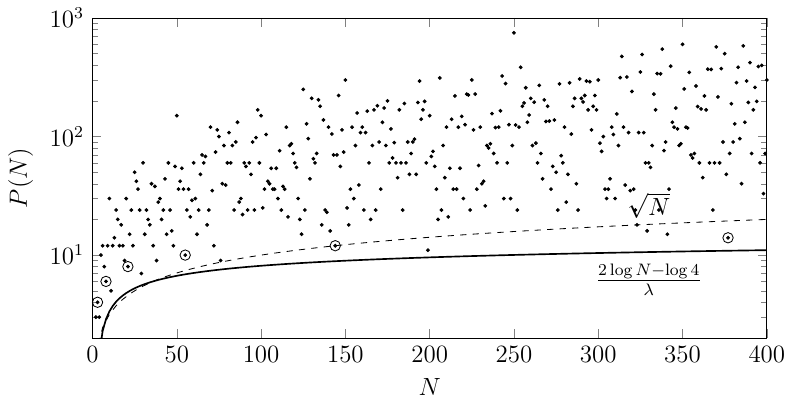}
\caption{The periods $P(N)$ of the Arnold cat map for $2\le N\le400$ (dots). The solid curve is the lower bound in (i) of Proposition \ref{prop:periods}. The circles are the short periods $P(N_k)=2k$ in (iii). The dashed curve $\sqrt N$ indicates the size of the periods in (ii).}\label{fig:periods}
\end{figure}

\section{Quantum cat maps}\label{sec:quantum}
In this section, we review the quantization of cat maps, following Bouzouina--De Bi\`evre \cite[Section 6]{BouDB} and Faure--Nonnenmacher--De Bi\`evre \cite[Sections 2 and 3]{FNDB}. We also state the properties of translation operators, coherent states, and Husimi functions on the torus that we use later. Lemmas \ref{lemma:periodization} and \ref{lemma:trace} are proved in Appendix \ref{sec:apptranslation}. Lemmas \ref{lemma:coherent}, \ref{lemma:antiWick}, \ref{lemma:positionHusimi}, and \ref{lemma:overlap} are proved in Appendix \ref{sec:appcoherent}.

\subsection{Quantization on the plane}\label{sec:plane}
Define the position and momentum operators on the Schwartz space $\Sc(\R)$ by
\[\hat q\psi(q)=q\psi(q)\quad\text{and}\quad\hat p\psi(q)=-\frac i{2\pi N}\psi'(q).\]
For $v=(v_1,v_2)\in\R^2$, the translation operator\label{idx:translation} is
\[T_v=\exp\left(-2\pi iN\left(v_1\hat p-v_2\hat q\right)\right).\]
It is given by
\begin{equation}\label{eq:Tv}
T_v\psi(q)=e^{-\pi iNv_1v_2+2\pi iNv_2q}\psi\left(q-v_1\right).
\end{equation}
It is unitary on $L^2(\R)$ and extends to the space $\Sc'(\R)$ of tempered distributions. For $u,v\in\R^2$,
\begin{equation}\label{eq:TuTv}
T_v^*=T_{-v}\quad\text{and}\quad T_uT_v=e^{\pi iNu\wedge v}T_{u+v}=e^{2\pi iNu\wedge v}T_vT_u.
\end{equation}
The inner product $\langle\cdot,\cdot\rangle$ is linear in the first argument.

By the Stone--von Neumann theorem, there is a unitary operator $\hat M$ on $L^2(\R)$, unique up to a phase factor, such that
\begin{equation}\label{eq:metaplectic}
\hat MT_v\hat M^{-1}=T_{Mv}\quad\text{for all }v\in\R^2.
\end{equation}
It is a metaplectic operator\label{idx:metaplectic} of $M$, and it extends to $\Sc'(\R)$; see Folland \cite[Chapter 4]{F}.

Denote by
\[\phi_0(q)=(2N)^{\frac14}e^{-\pi Nq^2}\]
the ground state of the harmonic oscillator. The coherent state\label{idx:coherent} at $x\in\R^2$ is $\phi_x=T_x\phi_0$. A direct computation from \eqref{eq:Tv} gives
\begin{equation}\label{eq:overlap}
\left|\left\langle\phi_x,\phi_y\right\rangle\right|=e^{-\frac{\pi N}2|x-y|^2}\quad\text{for }x,y\in\R^2.
\end{equation}

The evolution of coherent states under $\hat M$ is explicit, through their Wigner functions. For $\psi\in L^2(\R)$, the Wigner function\label{idx:Wigner} of $\psi$ is
\[W_\psi(q,p)=N\int_\R\psi\left(q+\frac s2\right)\ol{\psi\left(q-\frac s2\right)}e^{-2\pi iNps}\,ds.\]
We use three standard properties of Wigner functions; see Folland \cite[Chapters 1 and 4]{F}. First, Moyal's identity states that
\[\left|\langle\psi,\vp\rangle\right|^2=\frac1N\int_{\R^2}W_\psi W_\vp\,dx\quad\text{for }\psi,\vp\in L^2(\R).\]
Second,
\[W_{T_v\psi}(x)=W_\psi(x-v)\quad\text{and}\quad W_{\hat M\psi}(x)=W_\psi\left(M^{-1}x\right).\]
Third,
\[W_{\phi_0}(x)=2Ne^{-2\pi N|x|^2}.\]

\begin{lemma}[Evolved coherent states]\label{lemma:Gaussian}
Let $x,x'\in\R^2$ and $s,s'\in\Z$. Then
\[\left|\left\langle\hat M^s\phi_x,\hat M^{s'}\phi_{x'}\right\rangle\right|\le\sqrt2e^{-\frac\lambda2|s-s'|}\exp\left(-\frac{\pi N|\Delta|^2}{2C_M^2e^{2\lambda\max\{|s|,|s'|\}}}\right),\quad\text{where }\Delta=M^sx-M^{s'}x'.\]
\end{lemma}
\begin{proof}
For $t\in\Z$, the matrix
\[\Sigma_t=M^t\left(M^t\right)^\top\]
is positive definite with determinant $1$. Here, $A^\top$ is the transpose of a matrix $A$. By \eqref{eq:metaplectic},
\[\hat M^s\phi_x=\hat M^sT_x\phi_0=T_{M^sx}\hat M^s\phi_0.\]
So the Wigner function of $\hat M^s\phi_x$ is
\[W_{\hat M^s\phi_x}(z)=2N\exp\left(-2\pi N\left(z-M^sx\right)^\top\Sigma_s^{-1}\left(z-M^sx\right)\right).\]
This is the density of the normal distribution with mean $M^sx$ and covariance matrix $(4\pi N)^{-1}\Sigma_s$. By Moyal's identity, the square of the overlap is $N^{-1}$ times the density of the normal distribution with mean $0$ and covariance matrix $(4\pi N)^{-1}(\Sigma_s+\Sigma_{s'})$, evaluated at $\Delta$. Since
\[\Sigma_s+\Sigma_{s'}=M^s\left(\Id+\Sigma_{s'-s}\right)\left(M^s\right)^\top,\]
this gives
\[\left|\left\langle\hat M^s\phi_x,\hat M^{s'}\phi_{x'}\right\rangle\right|^2=\frac2{\sqrt{\det\left(\Id+\Sigma_{s'-s}\right)}}\exp\left(-2\pi N\Delta^\top\left(\Sigma_s+\Sigma_{s'}\right)^{-1}\Delta\right).\]

Now, $\|M^t\|^2$ is the largest eigenvalue of $\Sigma_t$, so it is at most $\Tr\Sigma_t$. Since $\det\Sigma_t=1$, \eqref{eq:normMt} gives
\[\det\left(\Id+\Sigma_t\right)=2+\Tr\Sigma_t\ge\left\|M^t\right\|^2\ge e^{2\lambda|t|}.\]
Moreover, the largest eigenvalue of $\Sigma_s+\Sigma_{s'}$ is at most
\[\left\|M^s\right\|^2+\left\|M^{s'}\right\|^2\le2C_M^2e^{2\lambda\max\{|s|,|s'|\}}.\]
Taking square roots, we obtain the inequality.
\end{proof}

So $\hat M^s\phi_x$ is a Gaussian wave packet centered at $M^sx$. It is stretched by a factor comparable to $e^{\lambda|s|}$, asymptotically along $v_+$ if $s>0$ and along $v_-$ if $s<0$, and compressed in the other direction; see Figure \ref{fig:ellipses}. By \eqref{eq:normMt}, its width is comparable to $N^{-\frac12}e^{\lambda|s|}$, with constants depending only on $M$. So it tends to $0$ uniformly for $|s|\le cT_E$, for each fixed $c<\frac12$, where $T_E=\lambda^{-1}\log N$ is the Ehrenfest time\label{idx:Ehrenfest}. Two such states with the same center and different times $s\ne s'$ have overlap at most $\sqrt2e^{-\frac\lambda2|s-s'|}$ by Lemma \ref{lemma:Gaussian}, because they are stretched in different ways.

\begin{figure}[!ht]
\centering
\begin{tikzpicture}
\draw[->] (-5.06,0) -- (5.29,0) node[right] {$v_+$};
\draw[->] (0,-2.99) -- (0,3.105) node[above] {$v_-$};
\foreach \sx/\sy in {1/1,-1/1,1/-1,-1/-1}{
\draw[->] plot[domain=0.62:4.0,samples=40,variable=\t] ({1.15*\sx*\t},{1.725*\sy/\t});}
\draw[thick] (0,0) circle (0.575);
\draw[thick] (0,0) ellipse (1.15 and 0.2875);
\draw[thick] (0,0) ellipse (2.3 and 0.14375);
\draw[thick] (0,0) ellipse (4.14 and 0.0805);
\draw[thick,dashed] (0,0) ellipse (0.2875 and 1.15);
\draw[thick,dashed] (0,0) ellipse (0.14375 and 2.3);
\draw (0.483,0.483) node[above right] {\small $s=0$};
\draw (1.15,0) node[below right] {\small $1$};
\draw (2.3,0) node[below right] {\small $2$};
\draw (4.14,0) node[below right] {\small $3$};
\draw (0,1.15) node[left] {\small $-1$};
\draw (0,2.3) node[left] {\small $-2$};
\end{tikzpicture}
\caption{The ellipses on which the Wigner functions of $\hat M^s\phi_0$ are concentrated, in the coordinates $(x_+,x_-)$, for the fixed point $0$ of $M$ (not to scale). Solid: $s=0,1,2,3$, stretched along $v_+$. Dashed: $s=-1,-2$, stretched along $v_-$. The thin arrows indicate the motion under $M^t$, expanding along $v_+$ and contracting along $v_-$. All ellipses have the same area. The drawing is schematic: for a non-symmetric $M$, the axes of the ellipses approach the directions $v_\pm$ only as $|s|\to\infty$.}\label{fig:ellipses}
\end{figure}

\subsection{Quantum states on the torus}\label{sec:states}
By \eqref{eq:TuTv}, the operators $T_{(1,0)}$ and $T_{(0,1)}$ commute, since $e^{2\pi iN}=1$. Let $\theta=(\theta_1,\theta_2)\in[0,2\pi)^2$. Denote by $\Hc_{N,\theta}$ the space of $\psi\in\Sc'(\R)$ with
\[T_{(1,0)}\psi=e^{i\theta_1}\psi\quad\text{and}\quad T_{(0,1)}\psi=e^{i\theta_2}\psi.\]
Here, $\theta$ is the Floquet angle\label{idx:Floquet}. By \eqref{eq:Tv}, $T_{(1,0)}$ is the translation by $1$, and $T_{(0,1)}$ is the multiplication by $e^{2\pi iNq}$. So $\psi\in\Hc_{N,\theta}$ if and only if
\begin{equation}\label{eq:states}
\psi=\sum_{k\in\Z}\psi_k\delta_{q_k},\quad\text{where }q_k=\frac1N\left(k+\frac{\theta_2}{2\pi}\right)\text{ and }\psi_{k+N}=e^{-i\theta_1}\psi_k\text{ for }k\in\Z.
\end{equation}
In particular, $\Hc_{N,\theta}$ is $N$-dimensional. We equip it with the inner product
\[\langle\psi,\vp\rangle=\sum_{k=0}^{N-1}\psi_k\ol{\vp_k}.\]
This normalization differs from that of Kurlberg--Rudnick \cite{KR1,KR2} by a factor of $N$, which plays no role for normalized states. It makes the coherent states below asymptotically normalized.

The position density\label{idx:position} of $\psi\in\Hc_{N,\theta}$ is the measure
\[|\psi|^2=\sum_{k=0}^{N-1}\left|\psi_k\right|^2\delta_{q_k}\quad\text{on }\T=\R/\Z,\] 
where $q_k$ is taken modulo $1$. By \eqref{eq:states}, the sequence $|\psi_k|$ is $N$-periodic, so the sum can be taken over any $N$ consecutive indices. The mass of $\psi$ on a Borel set $E\subset\T$ is
\[|\psi|^2(E)=\sum_{0\le k<N,q_k\in E}\left|\psi_k\right|^2.\]
In particular, the total mass of $|\psi|^2$ is $\|\psi\|^2$.

States on the torus are obtained from states on the plane by periodization\label{idx:periodization}. For $\vp\in\Sc(\R)$, define
\[\Pi_N\vp=\sum_{k\in\Z}\left(\frac1{\sqrt N}\sum_{l\in\Z}e^{-il\theta_1}\vp\left(q_k-l\right)\right)\delta_{q_k}.\]
In the notation of \cite[Section 3.2]{FNDB},
\[\Pi_N=\sqrt N\hat P_\theta,\]
where $\hat P_\theta$ is the projector onto $\Hc_{N,\theta}$. The next lemma shows that inner products on the torus are lattice sums of inner products on the plane.

\begin{restatable}[Periodization]{lemma}{periodizationlemma}\label{lemma:periodization}
Let $\chi,\vp\in\Sc(\R)$.
\begin{enumerate}[(i)]
\item 
\[\Pi_N\vp\in\Hc_{N,\theta}.\]
\item There are $c_n\in\C$ with $|c_n|=1$ and $c_0=1$, depending only on $n\in\Z^2$, $N$, and $\theta$, such that
\[\left\langle\Pi_N\chi,\Pi_N\vp\right\rangle=\sum_{n\in\Z^2}c_n\left\langle\chi,T_n\vp\right\rangle_{L^2(\R)}.\]
\item For $n\in\Z^2$, there is $c'_n\in\C$ with $|c'_n|=1$ such that
\[\Pi_NT_{n/N}\vp=T_{n/N}\Pi_N\vp\quad\text{and}\quad\Pi_NT_n\vp=c'_n\Pi_N\vp.\]
\end{enumerate}
\end{restatable}

\subsection{Observables and the quantum cat map}\label{sec:observables}
For $n,u\in\Z^2$,
\[2\pi N\left(\frac nN\wedge u\right)=2\pi(n\wedge u)\in2\pi\Z.\]
So by \eqref{eq:TuTv}, $T_{n/N}$ commutes with $T_{(1,0)}$ and $T_{(0,1)}$, and it preserves $\Hc_{N,\theta}$. Denote by $T_N(n)$ the restriction of $T_{n/N}$ to $\Hc_{N,\theta}$. By \eqref{eq:Tv},
\begin{equation}\label{eq:TNn}
\left(T_N(n)\psi\right)_k=e^{-\frac{\pi i}Nn_1n_2+2\pi in_2q_k}\psi_{k-n_1}\quad\text{for }k\in\Z.
\end{equation}
So $T_N(n)$ is unitary. By \eqref{eq:TuTv}, for $m,n\in\Z^2$,
\begin{equation}\label{eq:TNcomm}
T_N(n)^*=T_N(-n)\quad\text{and}\quad T_N(m)T_N(n)=e^{\frac{\pi i}Nm\wedge n}T_N(m+n)=e^{\frac{2\pi i}Nm\wedge n}T_N(n)T_N(m).
\end{equation}
For $m\in\Z^2$, the operator $T_N(Nm)$ is the restriction of $T_m$, which is a unimodular multiple of $T_{(1,0)}^{m_1}T_{(0,1)}^{m_2}$. So $T_N(Nm)$ is a unimodular multiple of the identity.

\begin{restatable}[Trace and irreducibility]{lemma}{tracelemma}\label{lemma:trace}\leavevmode
\begin{enumerate}[(i)]
\item Let $n\in\Z^2$. Then
\[\left|\Tr T_N(n)\right|=\begin{cases}
N & \text{if }n\in N\Z^2,\\
0 & \text{otherwise}.\end{cases}\]
\item If an operator $A$ on $\Hc_{N,\theta}$ commutes with $T_N(n)$ for all $n\in\Z^2$, then $A$ is a multiple of the identity.
\end{enumerate}
\end{restatable}

For $n\in\Z^2$, the Fourier mode\label{idx:Fourier} $\eb_n$ on $\T^2$ is
\[\eb_n(x)=e^{2\pi i(n\wedge x)}.\]
For $a\in C^\infty(\T^2)$, denote its Fourier coefficients by
\[\tilde a(n)=\int_{\T^2}a(x)\ol{\eb_n(x)}\,dx,\quad\text{so that}\quad a=\sum_{n\in\Z^2}\tilde a(n)\eb_n.\]
The Weyl quantization\label{idx:Weyl} of $a$ is
\begin{equation}\label{eq:OpN}
\Op_N(a)=\sum_{n\in\Z^2}\tilde a(n)T_N(n).
\end{equation}
The series converges in the operator norm, and
\[\left\|\Op_N(a)\right\|\le\sum_{n\in\Z^2}\left|\tilde a(n)\right|.\]
If $a$ is real-valued, then
\[\tilde a(-n)=\ol{\tilde a(n)}\quad\text{for }n\in\Z^2.\]
So $\Op_N(a)$ is self-adjoint by \eqref{eq:TNcomm}.

Suppose that $a$ depends only on the position variable $q$. Then $\tilde a(n)=0$ unless $n=(0,m)$ with $m\in\Z$, and
\[\eb_{(0,m)}(x)=e^{2\pi imq}.\]
By \eqref{eq:TNn}, $T_N((0,m))$ is the multiplication by $e^{2\pi imq_k}$. Hence, $\Op_N(a)$ is the multiplication by $a(q_k)$, and
\begin{equation}\label{eq:OpNposition}
\left\langle\Op_N(a)\psi,\psi\right\rangle=\int_\T a\,d|\psi|^2\quad\text{for }a\in C^\infty(\T).
\end{equation}

By \eqref{eq:metaplectic}, $\hat M$ maps $\Hc_{N,\theta}$ onto $\Hc_{N,\theta'}$, where $\theta'$ depends affinely on $\theta$. For each $N$, there is $\theta$ with $\theta'=\theta$; see Faure--Nonnenmacher--De Bi\`evre \cite[Section 3.2]{FNDB} and Bouzouina--De Bi\`evre \cite[Section 6]{BouDB}. We fix such a $\theta=\theta(M,N)$ and write $\Hc_N=\Hc_{N,\theta}$\label{idx:HN}.

If $M$ is congruent modulo $2$ to $\Id$ or to the matrix that exchanges the two coordinates, then $\theta=0$ is such a choice for all $N$ \cite[Section 3.2]{FNDB}, and we take it. In this case, $\Hc_N$ is the space $L^2(\Z/N\Z)$ of Kurlberg--Rudnick \cite{KR1,KR2}.

The restriction of $\hat M$ to $\Hc_N$ is unitary \cite[Section 6]{BouDB}. We still denote it by $\hat M$ and call it the quantum cat map\label{idx:quantumcat}. Since the projector $\hat P_\theta$ commutes with $\hat M$ for this $\theta$,
\begin{equation}\label{eq:MPc}
\hat M\Pi_N\vp=\Pi_N\hat M\vp\quad\text{for }\vp\in\Sc(\R).
\end{equation}
The eigenvalues of $\hat M$ have the form $e^{i\w}$ with $\w\in\R$. By (ii) of Lemma \ref{lemma:trace} and the next proposition, $\hat M$ is determined up to a phase factor, which does not change its eigenvectors.

\begin{prop}[Exact Egorov theorem]\label{prop:Egorov}
For $n\in\Z^2$, $t\in\Z$, and $a\in C^\infty(\T^2)$,
\[\hat MT_N(n)\hat M^{-1}=T_N(Mn)\quad\text{and}\quad\hat M^{-t}\Op_N(a)\hat M^t=\Op_N\left(a\circ M^t\right).\]
\end{prop}
\begin{proof}
The first identity is the restriction of \eqref{eq:metaplectic} to $\Hc_N$. For the second, it suffices to take $t=1$. Since $M$ preserves the symplectic product,
\[\eb_n(Mx)=\eb_{M^{-1}n}(x),\quad\text{so that}\quad\widetilde{a\circ M}(n)=\tilde a(Mn).\]
Hence,
\[\Op_N(a\circ M)=\sum_{n\in\Z^2}\tilde a(Mn)T_N(n)=\sum_{n\in\Z^2}\tilde a(n)T_N\left(M^{-1}n\right)=\hat M^{-1}\Op_N(a)\hat M.\qedhere\]
\end{proof}

\subsection{Quantum periods}\label{sec:qperiods}
By Proposition \ref{prop:Egorov},
\[\hat M^tT_N(n)\hat M^{-t}=T_N\left(M^tn\right)\quad\text{for }t\in\Z\text{ and }n\in\Z^2.\]
For $t=P(N)$, we have $M^tn=n\pmod N$. So the operator $\hat M^{P(N)}$ commutes with all $T_N(n)$ up to phase factors. The next lemma shows that a bounded power of it is a multiple of the identity. So $\hat M$ has few distinct eigenvalues if $P(N)$ is short. Faure--Nonnenmacher--De Bi\`evre \cite{FNDB} used a version of this fact to construct scarred eigenstates.

We call the smallest $Q\in\N$ such that $\hat M^Q$ is a multiple of the identity the quantum period\label{idx:qperiod} of $\hat M$.

\begin{lemma}[Few distinct eigenvalues]\label{lemma:fewvalues}
The quantum period $Q$ of $\hat M$ satisfies 
\[Q\le|\Tr M-2|P(N).\] 
In particular, $\hat M$ has at most $|\Tr M-2|P(N)$ distinct eigenvalues.
\end{lemma}
\begin{proof}
Write $P=P(N)$. By Proposition \ref{prop:Egorov},
\[\hat M^PT_N(n)\hat M^{-P}=T_N\left(M^Pn\right)\quad\text{for }n\in\Z^2.\]
Since $M^Pn-n\in N\Z^2$ and $T_N(Nm)$ is a multiple of the identity for $m\in\Z^2$, \eqref{eq:TNcomm} gives unimodular numbers $\chi(n)$ with
\[\hat M^PT_N(n)\hat M^{-P}=\chi(n)T_N(n)\quad\text{for }n\in\Z^2.\]
Conjugation by $\hat M^P$ preserves the products in \eqref{eq:TNcomm}. So $\chi(n+m)=\chi(n)\chi(m)$ for $n,m\in\Z^2$. Moreover, $\chi(Nn)=1$, since $T_N(Nn)$ is a multiple of the identity. Thus, $\chi$ is a character of $(\Z/N\Z)^2$, and there is $k\in\Z^2$ with
\[\chi(n)=e^{\frac{2\pi i}Nk\wedge n}\quad\text{for }n\in\Z^2.\]
By \eqref{eq:TNcomm}, conjugation by $T_N(k)$ acts on every $T_N(n)$ in the same way. So $T_N(k)^{-1}\hat M^P$ commutes with all $T_N(n)$, and (ii) of Lemma \ref{lemma:trace} gives
\[\hat M^P=cT_N(k)\quad\text{with }|c|=1.\]
Since $\hat M$ commutes with $\hat M^P$, Proposition \ref{prop:Egorov} gives $T_N(Mk)=T_N(k)$. By \eqref{eq:TNcomm}, $T_N(k)^*T_N(Mk)$ is a unimodular multiple of $T_N(Mk-k)$. So the trace of $T_N(Mk-k)$ has absolute value $N$, and (i) of Lemma \ref{lemma:trace} gives $(M-\Id)k\in N\Z^2$. Multiplying by the adjugate of $M-\Id$ gives
\[(2-\Tr M)k\in N\Z^2.\]
Write $\ell=|\Tr M-2|$. Since $k\wedge k=0$, \eqref{eq:TNcomm} gives $T_N(k)^\ell=T_N(\ell k)$. So
\[\hat M^{\ell P}=c^\ell T_N(\ell k)\]
is a multiple of the identity. So $Q\le\ell P$.

Finally, $\hat M^Q=c'\Id$ for some $c'\in\C$. So every eigenvalue $e^{i\w}$ satisfies $e^{iQ\w}=c'$, and there are at most $Q$ of them.
\end{proof}

\subsection{Coherent states and Husimi functions on the torus}\label{sec:coherent}
For $x\in\R^2$, the coherent state\label{idx:coherenttorus} at $x$ on the torus is $\Phi_x=\Pi_N\phi_x\in\Hc_N$. By (iii) of Lemma \ref{lemma:periodization} and \eqref{eq:TuTv}, $\Phi_{x+n}$ is a unimodular multiple of $\Phi_x$ for $n\in\Z^2$. So for $x\in\T^2$, the state $\Phi_x$ is defined up to a phase factor, which plays no role below.

\begin{restatable}[Coherent states on the torus]{lemma}{coherentlemma}\label{lemma:coherent}
Let $\psi\in\Hc_N$ and $x,y\in\T^2$.
\begin{enumerate}[(i)]
\item (Resolution of the identity)
\[N\int_{\T^2}\left|\left\langle\psi,\Phi_x\right\rangle\right|^2\,dx=\|\psi\|^2.\]
\item (Position marginal) For $q\in\T$,
\[N\int_\T\left|\left\langle\psi,\Phi_{(q,p)}\right\rangle\right|^2\,dp=\sum_{k=0}^{N-1}\left|\psi_k\right|^2g_N\left(q-q_k\right),\quad\text{where }g_N(s)=\sqrt{2N}\sum_{l\in\Z}e^{-2\pi N(s-l)^2}.\]
\item (Lattice sums) For $a\ge2$ and $z\in\R^2$,
\[\sum_{n\in\Z^2}e^{-a|z-n|^2}\le9e^{-\frac a2\dist\left(z,\Z^2\right)^2}.\]
\item (Gaussian decay)
\[\left|\left\langle\Phi_x,\Phi_y\right\rangle\right|\le9e^{-\frac\pi4N\dist(x,y)^2}\quad\text{and}\quad\left|\left\|\Phi_x\right\|^2-1\right|\le5e^{-\frac\pi2N}.\]
\end{enumerate}
\end{restatable}

(i) of Lemma \ref{lemma:coherent} is the resolution of the identity in \cite[Section 4.1]{FNDB}, where it is quoted from \cite{BouDB}.

Recall the Husimi function $\Hr_\psi$ in \eqref{eq:Husimi}. By (i) of Lemma \ref{lemma:coherent}, $\Hr_\psi\,dx$ is a probability measure if $\|\psi\|=1$.

The anti-Wick quantization\label{idx:antiWick} of $a\in L^\infty(\T^2)$ is the operator on $\Hc_N$ given by
\[\Op_N^\aw(a)\psi=N\int_{\T^2}a(x)\left\langle\psi,\Phi_x\right\rangle\Phi_x\,dx.\]
So
\begin{equation}\label{eq:antiWick}
\left\langle\Op_N^\aw(a)\psi,\psi\right\rangle=\int_{\T^2}a\Hr_\psi\,dx.
\end{equation}
By (i) of Lemma \ref{lemma:coherent}, the anti-Wick quantization of the constant function $1$ is the identity. If $a\ge0$, then $\Op_N^\aw(a)\ge0$. Hence, for real-valued $a$ and $b$ with $a\le b$,
\[\Op_N^\aw(a)\le\Op_N^\aw(b).\]
By the Cauchy--Schwarz inequality and (i) of Lemma \ref{lemma:coherent},
\[\left\|\Op_N^\aw(a)\right\|\le\|a\|_{L^\infty(\T^2)}.\]
The positivity is the reason to use the anti-Wick quantization for the indicator functions of balls. The next lemma relates it to the Weyl quantization.

\begin{restatable}[Anti-Wick quantization of Fourier modes]{lemma}{antiwicklemma}\label{lemma:antiWick}
For $n\in\Z^2$,
\[\Op_N^\aw\left(\eb_n\right)=e^{-\frac{\pi|n|^2}{2N}}T_N(n).\]
In particular, for $a\in C^\infty(\T^2)$,
\[\Op_N^\aw(a)=\sum_{n\in\Z^2}e^{-\frac{\pi|n|^2}{2N}}\tilde a(n)T_N(n)\quad\text{and}\quad\left\|\Op_N^\aw(a)-\Op_N(a)\right\|\le\frac\pi{2N}\sum_{n\in\Z^2}|n|^2\left|\tilde a(n)\right|.\]
\end{restatable}

The Husimi function also controls the position density, up to Gaussian errors at the scale $N^{-\frac12}$.

\begin{restatable}[Position density and Husimi function]{lemma}{positionlemma}\label{lemma:positionHusimi}
Let $\psi\in\Hc_N$ with $\|\psi\|=1$. Let $q\in\T$ and $0<\rho<r\le\frac12$. Then
\[|\psi|^2\left(B_1(q,r)\right)\ge\Hr_\psi\left(B_1(q,r-\rho)\times\T\right)-2e^{-2\pi N\rho^2}\]
and
\[\left(1-2e^{-2\pi N\rho^2}\right)|\psi|^2\left(B_1(q,r)\right)\le\Hr_\psi\left(B_1(q,r+\rho)\times\T\right).\]
\end{restatable}

A large overlap with a coherent state forces a large mass in a ball around its center. Indeed, the Husimi function of $\psi$ at $x$ is $N$ times the squared overlap of $\psi$ and $\Phi_x$, and a ball of radius $N^{-\frac12}$ has area comparable to $1/N$. The next lemma makes this precise.

\begin{restatable}[Overlap and concentration]{lemma}{overlaplemma}\label{lemma:overlap}
Let $\psi\in\Hc_N$ with $\|\psi\|=1$. Suppose that $x=(q,p)\in\T^2$. Then for $r\in(0,\frac12]$,
\[\Hr_\psi\left(B_2(x,r)\right)\ge\left|\left\langle\psi,\Phi_x\right\rangle\right|^2-49e^{-\frac\pi4Nr^2}\quad\text{and}\quad|\psi|^2\left(B_1(q,r)\right)\ge\left|\left\langle\psi,\Phi_x\right\rangle\right|^2-51e^{-\frac\pi{16}Nr^2}.\]
\end{restatable}

\subsection{Quasimodes and semiclassical measures}\label{sec:qmeasures}
Recall the quasimodes and the semiclassical measures defined in Section \ref{sec:introquasi}. By Lemma \ref{lemma:antiWick}, the semiclassical measure of a sequence $\{\psi_N\}$ is $\mu$ if and only if
\[\lim_{N\to\infty}\left\langle\Op_N(a)\psi_N,\psi_N\right\rangle=\int_{\T^2}a\,d\mu\quad\text{for all }a\in C^\infty(\T^2).\]
So the Weyl and the anti-Wick quantizations give the same semiclassical measures.

\begin{lemma}[Invariance of semiclassical measures]\label{lemma:invariance}
Let $\Gc\subset\N$ be infinite. For $N\in\Gc$, let $\psi_N\in\Hc_N$ be a quasimode of order $\ve_N$. Suppose that $\ve_N\to0$ as $N\to\infty$ in $\Gc$ and that $\Hr_{\psi_N}\,dx$ converges weakly to a measure $\mu$. Then $\mu$ is $M$-invariant.
\end{lemma}
\begin{proof}
Let $a\in C^\infty(\T^2)$. Choose $\w_N\in\R$ with
\[\left\|\hat M\psi_N-e^{i\w_N}\psi_N\right\|\le\ve_N.\]
By Proposition \ref{prop:Egorov} and since $\|\Op_N(a)\|$ is at most the sum of $|\tilde a(n)|$,
\[\left\langle\Op_N(a\circ M)\psi_N,\psi_N\right\rangle=\left\langle\Op_N(a)\hat M\psi_N,\hat M\psi_N\right\rangle=\left\langle\Op_N(a)\psi_N,\psi_N\right\rangle+O\left(\ve_N\sum_{n\in\Z^2}\left|\tilde a(n)\right|\right).\]
Letting $N\to\infty$ in $\Gc$ gives
\[\int_{\T^2}a\circ M\,d\mu=\int_{\T^2}a\,d\mu.\]
Since $C^\infty(\T^2)$ is dense in $C(\T^2)$, $\mu$ is $M$-invariant.
\end{proof}

\section{Equidistribution at small scales}\label{sec:positive}
In this section, we prove Theorems \ref{thm:QElog}, \ref{thm:density}, and \ref{thm:Hecke}. The proof of Theorem \ref{thm:QElog} is based on the moment estimate in Proposition \ref{prop:moments}, which gives the sharp exponents. The proofs of Theorems \ref{thm:density} and \ref{thm:Hecke} use the fourth moment bounds of Kurlberg--Rudnick \cite{KR1,KR2} and, for (iii) of Theorem \ref{thm:density}, the bound of Bourgain \cite{Bo}.

\subsection{Beurling--Selberg polynomials}
The Weyl and anti-Wick quantizations of a trigonometric polynomial are finite sums of translation operators. So we approximate the indicator functions of balls by trigonometric polynomials. We regard functions on $\T$ as functions on $\T^2$ that depend only on the position variable $q$. We use the following version of the Beurling--Selberg polynomials, due to Harman \cite[Lemma 4]{Har}; see also Lester--Rudnick \cite[Lemma 2.5]{LR}.

\begin{lemma}[Beurling--Selberg polynomials]\label{lemma:BS}
Let $d\in\{1,2\}$. There is $c_d>0$ with the following property. Let $r\in(0,\frac14]$ and $D\ge1/r$. Then there are real-valued trigonometric polynomials $a_r^+$ and $a_r^-$ on $\T^d$ such that the following hold.
\begin{enumerate}[(i)]
\item They bound the indicator function of the ball:
\[a_r^-\le\indicator_{B_d(0,r)}\le a_r^+.\]
\item They have degree less than $D$, that is,
\[\tilde a_r^\pm(n)=0\quad\text{for }|n|\ge D.\]
\item Their means satisfy
\[\left|\tilde a_r^\pm(0)-\Vol\left(B_d(0,r)\right)\right|\le\frac{c_dr^{d-1}}D.\]
\item Their Fourier coefficients satisfy
\[\left|\tilde a_r^\pm(n)\right|\le c_dr^d\quad\text{for all }n.\]
\end{enumerate}
\end{lemma}

Here, $\Vol(B_d(0,r))$ is the measure of $B_d(0,r)$, which is $2r$ for $d=1$ and $\pi r^2$ for $d=2$. The statements in \cite{Har,LR} use the Fourier transform with respect to the dot product, while $\tilde a(n)$ uses the symplectic product. The two differ by a rotation of $n$ by $\pi/2$, which does not change (ii)--(iv) of Lemma \ref{lemma:BS}.

For $y\in\T^d$, denote by $a_{y,r}^\pm$ the translates of $a_r^\pm$ by $y$. That is,
\[a_{y,r}^\pm=a_r^\pm(\cdot-y),\quad\text{so that}\quad\left|\tilde a_{y,r}^\pm(n)\right|=\left|\tilde a_r^\pm(n)\right|.\]
So (i)--(iv) of Lemma \ref{lemma:BS} hold for $a_{y,r}^\pm$ and the ball $B_d(y,r)$, uniformly in $y$. Figure \ref{fig:majorant} shows an example.

\begin{figure}[!ht]
\centering
\includegraphics{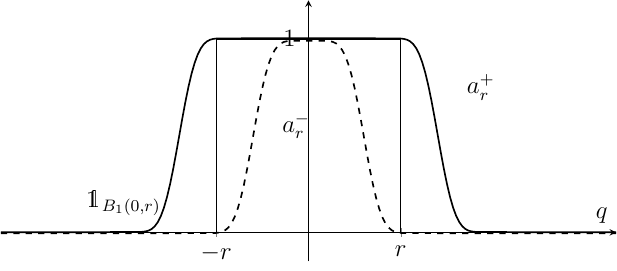}
\caption{A majorant $a_r^+$ (solid) and a minorant $a_r^-$ (dashed) of $\indicator_{B_1(0,r)}$ by trigonometric polynomials of degree $22$, with $r=0.15$, as in Lemma \ref{lemma:BS}.}\label{fig:majorant}
\end{figure}

For $d\in\{1,2\}$ and $\psi\in\Hc_N$ with $\|\psi\|=1$, denote by $\mu_\psi$ the position density $|\psi|^2$ on $\T$ if $d=1$, and the measure $\Hr_\psi\,dx$ on $\T^2$ if $d=2$. The next lemma expresses the integral of a trigonometric polynomial with respect to $\mu_\psi$ by matrix elements of translation operators.

\begin{lemma}[Trigonometric polynomials]\label{lemma:trigpoly}
Let $d\in\{1,2\}$. Suppose that $a$ is a real-valued trigonometric polynomial on $\T^d$ of degree less than $D\ge1$. Then there is a real-valued trigonometric polynomial $b$ on $\T^2$ with mean $0$, at most $9D^d$ non-zero Fourier coefficients, and $|\tilde b(n)|\le|\tilde a(n)|$ for all $n$, such that
\[\int_{\T^d}a\,d\mu_\psi=\tilde a(0)+\left\langle\Op_N(b)\psi,\psi\right\rangle\quad\text{for all }\psi\in\Hc_N\text{ with }\|\psi\|=1.\]
\end{lemma}
\begin{proof}
For $d=1$, let $b=a-\tilde a(0)$, regarded as a function on $\T^2$ that depends only on $q$, and use \eqref{eq:OpNposition}. For $d=2$, let $b$ be given by
\[\tilde b(0)=0\quad\text{and}\quad\tilde b(n)=e^{-\frac{\pi|n|^2}{2N}}\tilde a(n)\quad\text{for }n\ne0,\]
and use \eqref{eq:antiWick} and Lemma \ref{lemma:antiWick}.
\end{proof}

The next lemma reduces the equidistribution of a single state to bounds on the matrix elements of the translation operators. It is used for the polynomial scales.

\begin{lemma}[Reduction to translation operators]\label{lemma:modes}
Let $\psi\in\Hc_N$ with $\|\psi\|=1$. Suppose that $r\in(0,\frac14]$ and $D\ge1/r$. Then
\begin{align*}
\sup_{q\in\T}\left|\frac{|\psi|^2\left(B_1(q,r)\right)}{2r}-1\right|&\le\frac{c_1}{2Dr}+\frac{c_1}2\sum_{1\le|l|<D}\left|\left\langle T_N\left((0,l)\right)\psi,\psi\right\rangle\right|,\\
\sup_{x\in\T^2}\left|\frac{\Hr_\psi\left(B_2(x,r)\right)}{\pi r^2}-1\right|&\le\frac{c_2}{\pi Dr}+\frac{c_2}\pi\sum_{1\le|n|<D}\left|\left\langle T_N(n)\psi,\psi\right\rangle\right|.
\end{align*}
\end{lemma}
\begin{proof}
Let $d\in\{1,2\}$ and $x\in\T^d$. By (i) of Lemma \ref{lemma:BS} for the translates $a_{x,r}^\pm$,
\[\int_{\T^d}a_{x,r}^-\,d\mu_\psi\le\mu_\psi\left(B_d(x,r)\right)\le\int_{\T^d}a_{x,r}^+\,d\mu_\psi.\]
By Lemma \ref{lemma:trigpoly}, \eqref{eq:OpN}, and (ii)--(iv) of Lemma \ref{lemma:BS},
\[\left|\int_{\T^d}a_{x,r}^\pm\,d\mu_\psi-\Vol\left(B_d(0,r)\right)\right|\le\frac{c_dr^{d-1}}D+c_dr^d\sum_{1\le|n|<D}\left|\left\langle T_N(n)\psi,\psi\right\rangle\right|,\]
where the sum is over the vectors $n=(0,l)$ if $d=1$. Dividing by $\Vol(B_d(0,r))$ gives the two inequalities.
\end{proof}

\subsection{Moments of observables}\label{sec:moments}
In this subsection, we estimate the moments of $\langle\Op_N(b)\psi_j,\psi_j\rangle$ over an orthonormal basis of quasimodes, for a trigonometric polynomial $b$ with mean $0$. The key input is the bad approximability of $M$ in the following form: a vanishing sum of vectors $M^{t_k}\xi_k$ cannot be split by a long time gap into two parts with non-zero sums.

\begin{lemma}[Clusters]\label{lemma:cluster}
Let $\xi_1,\ldots,\xi_{2m}\in\Z^2$ with $|\xi_k|\le D$. Suppose that the integers $t_1\le\cdots\le t_{2m}$ satisfy
\[\sum_{k=1}^{2m}M^{t_k}\xi_k=0.\]
Assume that
\[e^{\lambda(t_{l+1}-t_l)}>\frac{m^2C_M^2D^2}{c_M}\]
for some $l\in\{1,\ldots,2m-1\}$, where $c_M$ is the constant in Lemma \ref{lemma:badapprox}. Then
\[\sum_{k=1}^lM^{t_k}\xi_k=0.\]
\end{lemma}
\begin{proof}
Denote by $u$ the integer vector
\[u=\sum_{k=1}^lM^{t_k}\xi_k=-\sum_{k=l+1}^{2m}M^{t_k}\xi_k.\]
By the definition of $C_M$, the coordinates of each $\xi_k$ satisfy $|(\xi_k)_\pm|\le C_MD$. Since $t_k\le t_l$ for $k\le l$ and $t_k\ge t_{l+1}$ for $k>l$, \eqref{eq:Mt} and the two expressions of $u$ give
\[\left|u_+\right|\le lC_MDe^{\lambda t_l}\quad\text{and}\quad\left|u_-\right|\le(2m-l)C_MDe^{-\lambda t_{l+1}}.\]
Since $l(2m-l)\le m^2$, the assumption gives
\[\left|u_+u_-\right|\le m^2C_M^2D^2e^{-\lambda(t_{l+1}-t_l)}<c_M.\]
By Lemma \ref{lemma:badapprox}, $u=0$.
\end{proof}

\begin{prop}[Moments over orthonormal bases]\label{prop:moments}
Let $b$ be a real-valued trigonometric polynomial with mean $0$ and degree less than $D\ge1$. Denote by $\Nc$ the set of $n\in\Z^2$ with $\tilde b(n)\ne0$. Suppose that $m,T\in\N$ satisfy
\begin{equation}\label{eq:nointerference}
2mC_Me^{\lambda T}D<N\quad\text{and}\quad\Lambda=\frac{|\Nc|\log(2+D)}T\le1.
\end{equation}
Then for every orthonormal basis $\{\psi_j\}_{j=1}^N$ of $\Hc_N$ consisting of quasimodes of order $\ve\ge0$,
\[\frac1N\sum_{j=1}^N\left\langle\Op_N(b)\psi_j,\psi_j\right\rangle^{2m}\le C\left\|\tilde b\right\|_{\ell^\infty}^{2m}\left(\Lambda^m+\left(\ve T|\Nc|\right)^{2m}\right),\]
where $C=C(M,m)>0$.
\end{prop}

Proposition \ref{prop:moments} is the main estimate of this section. It bounds the $2m$-th moment of the matrix elements of $\Op_N(b)$ over an orthonormal basis. For $m=1$, this is the quantum variance. Up to a constant, the bound $\|\tilde b\|_{\ell^\infty}^{2m}\Lambda^m$ is the $2m$-th moment of a Gaussian random variable with variance $\|\tilde b\|_{\ell^\infty}^2\Lambda$. So the matrix elements are small for most $j$ when $\Lambda$ is small, that is, when the averaging time $T$ is much larger than the number $|\Nc|$ of Fourier modes of $b$, up to the factor $\log(2+D)$. The first condition in \eqref{eq:nointerference} ensures that the evolved Fourier modes $M^t\xi$ with $0\le t<T$ do not interfere modulo $N$. It holds if $T\le\frac12T_E$ and $D\le N^{\frac13}$, for $N$ large in terms of $M$ and $m$, since then $e^{\lambda T}\le N^{\frac12}$. The term $(\ve T|\Nc|)^{2m}$ accounts for the quasimode error up to the time $T$. A quasimode of order $0$ is an eigenstate, so the case $\ve=0$ concerns the eigenbases of $\hat M$. The proposition is used with $m$ large in the proofs of Theorem \ref{thm:QElog} and of (ii) of Proposition \ref{prop:ONQE}, and with $m=1$ in the proof of (i) of Proposition \ref{prop:ONQE}.

\begin{proof}
\textit{Step 1: Averaging along the dynamics.} Write
\[X_j=\left\langle\Op_N(b)\psi_j,\psi_j\right\rangle,\]
which is real, since $\Op_N(b)$ is self-adjoint. Choose $\w_j\in\R$ with
\[\left\|\left(\hat M-e^{i\w_j}\right)\psi_j\right\|\le\ve.\]
Since $\hat M$ is unitary, induction on $t$ gives
\[\left\|\hat M^{-t}\psi_j-e^{-i\w_jt}\psi_j\right\|\le t\ve\quad\text{for }t\in\N.\]
For an operator $B$ and unit vectors $u$ and $v$, we have 
\[\left|\langle Bu,u\rangle-\langle Bv,v\rangle\right|\le2\|B\|\|u-v\|.\] 
By \eqref{eq:OpN}, the norm of $\Op_N(b)$ is at most $\|\tilde b\|_{\ell^\infty}|\Nc|$. So
\[\left|\left\langle\hat M^t\Op_N(b)\hat M^{-t}\psi_j,\psi_j\right\rangle-X_j\right|\le2t\ve\left\|\tilde b\right\|_{\ell^\infty}|\Nc|\quad\text{for }t\in\N.\]
By \eqref{eq:OpN} and Proposition \ref{prop:Egorov},
\[\hat M^t\Op_N(b)\hat M^{-t}=\sum_{\xi\in\Nc}\tilde b(\xi)T_N\left(M^t\xi\right).\]
Denote by $A$ the average of these operators over $0\le t<T$, that is,
\[A=\frac1T\sum_{t=0}^{T-1}\sum_{\xi\in\Nc}\tilde b(\xi)T_N\left(M^t\xi\right).\]
Then
\[\left|X_j-\left\langle A\psi_j,\psi_j\right\rangle\right|\le\ve T\left\|\tilde b\right\|_{\ell^\infty}|\Nc|.\]
The operator $A$ is self-adjoint. By the spectral theorem and Jensen's inequality,
\[\left\langle A\psi_j,\psi_j\right\rangle^{2m}\le\left\langle A^{2m}\psi_j,\psi_j\right\rangle.\]
Since $\{\psi_j\}_{j=1}^N$ is an orthonormal basis, the sum of the right-hand side over $j$ is $\Tr(A^{2m})$. Since 
\[(x+y)^{2m}\le2^{2m-1}\left(x^{2m}+y^{2m}\right)\quad\text{for }x,y\in\R,\] 
we obtain
\[\frac1N\sum_{j=1}^NX_j^{2m}\le\frac{2^{2m-1}}N\Tr\left(A^{2m}\right)+2^{2m-1}\left(\ve T\left\|\tilde b\right\|_{\ell^\infty}|\Nc|\right)^{2m}.\]
So it remains to show that
\[\frac1N\Tr\left(A^{2m}\right)\le C\left\|\tilde b\right\|_{\ell^\infty}^{2m}\Lambda^m.\]

\textit{Step 2: The trace.} Expanding $A^{2m}$ gives $T^{-2m}$ times a sum, over $(t_k,\xi_k)_{k=1}^{2m}$ with $0\le t_k<T$ and $\xi_k\in\Nc$, of the operators
\[\tilde b\left(\xi_1\right)\cdots\tilde b\left(\xi_{2m}\right)T_N\left(M^{t_1}\xi_1\right)\cdots T_N\left(M^{t_{2m}}\xi_{2m}\right).\]
By \eqref{eq:TNcomm}, the product of translation operators is a unimodular multiple of $T_N(u)$, where
\[u=\sum_{k=1}^{2m}M^{t_k}\xi_k.\]
By \eqref{eq:normMt},
\[|u|\le2mC_Me^{\lambda T}D.\]
By the first condition in \eqref{eq:nointerference}, this is less than $N$. So $u\in N\Z^2$ only if $u=0$. By (i) of Lemma \ref{lemma:trace}, the trace of each operator above has absolute value at most $N\|\tilde b\|_{\ell^\infty}^{2m}$ if $u=0$, and it vanishes otherwise. Hence,
\[\frac1N\Tr\left(A^{2m}\right)\le\frac{\left\|\tilde b\right\|_{\ell^\infty}^{2m}|\Xc|}{T^{2m}},\quad\text{where }\Xc=\left\{\left(t_k,\xi_k\right)_{k=1}^{2m}:0\le t_k<T,\xi_k\in\Nc,\sum_{k=1}^{2m}M^{t_k}\xi_k=0\right\}.\]
So it remains to show that $|\Xc|\le CT^{2m}\Lambda^m$.

\textit{Step 3: Clusters.} Denote by $G$ the smallest positive integer with
\[e^{\lambda G}>\frac{m^2C_M^2D^2}{c_M}.\]
Then $G\le C\log(2+D)$. Let $(t_k,\xi_k)_{k=1}^{2m}\in\Xc$. Reorder the indices so that $t_1\le\cdots\le t_{2m}$. Cut the list between $l$ and $l+1$ whenever $t_{l+1}-t_l\ge G$. By Lemma \ref{lemma:cluster}, the sum of $M^{t_k}\xi_k$ up to each cut vanishes. So the sum over each block between consecutive cuts vanishes; see Figure \ref{fig:clusters}. Since $M^{t_k}\xi_k\ne0$, each block has at least two elements. So there are at most $m$ blocks. Within a block, consecutive times differ by less than $G$, so all times lie within $2mG$ of the first one.

\begin{figure}[!ht]
\centering
\begin{tikzpicture}
\draw[->] (0,0) -- (13,0) node[right] {$t$};
\draw (0,0.12) -- (0,-0.12) node[below] {$0$};
\draw (12.5,0.12) -- (12.5,-0.12) node[below] {$T$};
\foreach \x in {1.2,1.6,2.1}{\fill (\x,0) circle (0.07);}
\foreach \x in {5.1,5.5}{\fill (\x,0) circle (0.07);}
\foreach \x in {9.0,9.3,9.9}{\fill (\x,0) circle (0.07);}
\draw[thick] (1.0,0.35) -- (2.3,0.35);
\draw[thick] (4.9,0.35) -- (5.7,0.35);
\draw[thick] (8.8,0.35) -- (10.1,0.35);
\draw[<->] (2.1,-0.55) -- (5.1,-0.55);
\draw (3.6,-0.55) node[below] {$\ge G$};
\draw[<->] (5.5,-0.55) -- (9.0,-0.55);
\draw (7.25,-0.55) node[below] {$\ge G$};
\draw (1.65,0.35) node[above] {$\sum M^{t_k}\xi_k=0$};
\draw (5.3,0.35) node[above] {$\sum M^{t_k}\xi_k=0$};
\draw (9.45,0.35) node[above] {$\sum M^{t_k}\xi_k=0$};
\end{tikzpicture}
\caption{The times $t_k$ of an element of the set $\Xc$ in the proof of Proposition \ref{prop:moments}, with $2m=8$: three blocks, separated by gaps of length at least $G$.}\label{fig:clusters}
\end{figure}

\textit{Step 4: Counting.} The reordering and the cuts can be chosen in at most $(2m)!2^{2m}$ ways. Fix them, and denote by $g\le m$ the number of blocks. In each block, the first time has at most $T$ choices, and each other time has at most $2mG$ choices. Each $\xi_k$ has at most $|\Nc|$ choices, except the last one in each block, which is determined since the sum over the block vanishes. So a block with $\ell$ elements contributes a factor at most $T(2mG|\Nc|)^{\ell-1}$, and
\[|\Xc|\le(2m)!2^{2m}\sum_{g=1}^mT^g\left(2mG|\Nc|\right)^{2m-g}=(2m)!2^{2m}T^{2m}\sum_{g=1}^m\left(\frac{2mG|\Nc|}T\right)^{2m-g}.\]
By the bound on $G$, the ratio in the last sum is at most $C\Lambda$. Since $\Lambda\le1$ and $2m-g\ge m$, each term of the sum is at most $C\Lambda^m$. Hence, 
\[|\Xc|\le CT^{2m}\Lambda^m.\]
\end{proof}

\subsection{Proof of Theorem \ref{thm:QElog}}
To pass from finitely many balls to all balls, we use majorants of slightly larger balls and minorants of slightly smaller balls, centered at the points of a grid. The next lemma holds for every positive measure.

\begin{lemma}[Finitely many balls]\label{lemma:grid}
Let $d\in\{1,2\}$. There is $C=C(d)>0$ with the following property. Let $r\in(0,\frac18]$ and $\eta\in(0,\frac14]$. Then there are a finite set $Y\subset\T^d$ and real-valued trigonometric polynomials $a_y^+$ and $a_y^-$ on $\T^d$ for $y\in Y$ such that the following hold.
\begin{enumerate}[(i)]
\item The set $Y$ has at most $2^d(\eta r)^{-d}$ elements.
\item The polynomials $a_y^\pm$ have degree less than $2(\eta r)^{-1}$, and their Fourier coefficients satisfy
\[\left|\tilde a_y^\pm(n)\right|\le2c_dr^d\quad\text{for all }n.\]
\item If $\mu$ is a positive measure on $\T^d$ with
\[\left|\int_{\T^d}a_y^\pm\,d\mu-\tilde a_y^\pm(0)\right|\le\eta r^d\quad\text{for all }y\in Y,\]
then
\[\sup_{x\in\T^d}\left|\frac{\mu\left(B_d(x,r)\right)}{\Vol\left(B_d(0,r)\right)}-1\right|\le C\eta.\]
\end{enumerate}
\end{lemma}
\begin{proof}
Denote by $K$ the smallest integer with $K\ge(\eta r)^{-1}$. Then $K\le2(\eta r)^{-1}$. Denote by $Y\subset\T^d$ the grid
\[Y=\left\{\frac kK:k\in\Z^d\right\}.\]
Then $Y$ has $K^d\le2^d(\eta r)^{-d}$ elements. Write $D=2(\eta r)^{-1}$. The radii $(1\pm\eta)r$ lie in $[r/2,\frac54r]\subset(0,\frac14]$, and $D\ge2r^{-1}$. For $y\in Y$, denote by $a_y^+$ the translate by $y$ of the majorant in Lemma \ref{lemma:BS} for the radius $(1+\eta)r$ and this $D$. Denote by $a_y^-$ the translate by $y$ of the minorant for the radius $(1-\eta)r$. By (ii) and (iv) of Lemma \ref{lemma:BS}, they have degree less than $D$, and their Fourier coefficients are at most 
\[c_d\left(\frac54r\right)^d\le2c_dr^d.\] 
This proves (i) and (ii).

For (iii), let $x\in\T^d$. Choosing the nearest grid point in each coordinate gives $y\in Y$ with
\[\dist(x,y)\le\frac{\sqrt d}{2K}<\eta r,\]
since $d\le2$ and $K^{-1}\le\eta r$. Then
\[B_d\left(y,(1-\eta)r\right)\subset B_d(x,r)\subset B_d\left(y,(1+\eta)r\right).\]
Since $\mu$ is positive, (i) of Lemma \ref{lemma:BS} and the hypothesis give
\[\mu\left(B_d(x,r)\right)\le\int_{\T^d}a_y^+\,d\mu\le\tilde a_y^+(0)+\eta r^d.\]
Since $(1+\eta)^d\le1+3\eta$ and $D=2(\eta r)^{-1}$, (iii) of Lemma \ref{lemma:BS} gives
\[\tilde a_y^+(0)\le(1+3\eta)\Vol\left(B_d(0,r)\right)+c_d\eta r^d.\]
Since $\Vol(B_d(0,r))\ge2r^d$, we obtain
\[\frac{\mu\left(B_d(x,r)\right)}{\Vol\left(B_d(0,r)\right)}\le1+\left(3+\frac{c_d+1}2\right)\eta.\]
The lower bound follows in the same way from $(1-\eta)^d\ge1-2\eta$.
\end{proof}

\begin{proof}[Proof of Theorem \ref{thm:QElog}]
The statements (i) and (ii) are the cases $d=1$ and $d=2$ of the following argument. Let $d\in\{1,2\}$ and $\alpha\in(0,1/d)$. In this proof, $C$ denotes a positive constant that depends only on $M$ and $\alpha$, whose value may change from line to line.

\textit{Step 1: Parameters.} Write
\[\kappa=\frac{1-d\alpha}{4d}.\]
Choose $m\in\N$ with 
\[(3d-2)m\kappa\ge d(\alpha+\kappa)+\kappa.\] 
Let
\[r=(\log N)^{-\alpha},\quad\eta=(\log N)^{-\kappa},\quad T=\left\lfloor\frac{\log N}{2\lambda}\right\rfloor.\]
Let $N$ be large, so that $r\le\frac18$, $\eta\le\frac14$, and $T\ge1$. Let $Y$ and $a_y^\pm$ be as in Lemma \ref{lemma:grid}. Their degrees are less than
\[D=2(\eta r)^{-1}=2(\log N)^{\alpha+\kappa}.\]

\textit{Step 2: Chebyshev's inequality.} Let $\{\psi_j\}_{j=1}^N$ be an eigenbasis. For $y\in Y$, let $b_y^\pm$ be the polynomials in Lemma \ref{lemma:trigpoly} for $a_y^\pm$. By (ii) of Lemma \ref{lemma:grid}, their Fourier coefficients are at most $2c_dr^d$. By Step 1 and since $d\alpha=1-4d\kappa$, the quantity $\Lambda$ in Proposition \ref{prop:moments} for $b_y^\pm$ satisfies
\[\Lambda\le\frac{9D^d\log(2+D)}T\le C(\log N)^{d(\alpha+\kappa)-1}\log\log N=C(\log N)^{-3d\kappa}\log\log N.\]
For $N$ large, $\Lambda\le1$, and \eqref{eq:nointerference} holds, since $e^{\lambda T}\le N^{\frac12}$. So Proposition \ref{prop:moments} with $\ve=0$ gives
\[\frac1N\sum_{j=1}^N\left\langle\Op_N\left(b_y^\pm\right)\psi_j,\psi_j\right\rangle^{2m}\le Cr^{2dm}\Lambda^m.\]
By Chebyshev's inequality, the number of $j$ with
\[\left|\left\langle\Op_N\left(b_y^\pm\right)\psi_j,\psi_j\right\rangle\right|>\eta r^d\]
is at most $C\eta^{-2m}\Lambda^mN$. Denote by $S$ the set of the other $j$, for all $y\in Y$ and both signs. Since $|Y|\le C(\log N)^{d(\alpha+\kappa)}$, we obtain
\[N-|S|\le C(\log N)^{d(\alpha+\kappa)}\eta^{-2m}\Lambda^mN\le C(\log N)^{d(\alpha+\kappa)-(3d-2)m\kappa}(\log\log N)^mN.\]
By the choice of $m$, the exponent of $\log N$ is at most $-\kappa$.

\textit{Step 3: Conclusion.} Let $j\in S$. By Lemma \ref{lemma:trigpoly}, the measure $\mu_{\psi_j}$ satisfies the hypothesis in (iii) of Lemma \ref{lemma:grid}. So
\[\sup_{x\in\T^d}\left|\frac{\mu_{\psi_j}\left(B_d(x,r)\right)}{\Vol\left(B_d(0,r)\right)}-1\right|\le C\eta=C(\log N)^{-\kappa}.\]
Write
\[\delta_N=C(\log N)^{-\kappa}(\log\log N)^m,\]
where $C$ is the larger of the constants in Steps 2 and 3. Then $|S|\ge(1-\delta_N)N$, and $\delta_N\to0$ depends only on $M$, $\alpha$, and $N$.
\end{proof}

\begin{rmk}[The number of moments]
The proof of Theorem \ref{thm:QElog} needs moments of high order. Up to powers of $\eta$ and $\log\log N$, the set $Y$ has $r^{-d}$ elements, and $\Lambda$ is of order $r^{-d}(\log N)^{-1}$. So the proportion of exceptional $j$ in Step 2 is of order
\[r^{-d}\Lambda^m=r^{-d(m+1)}(\log N)^{-m}=(\log N)^{d\alpha(m+1)-m},\]
which tends to $0$ if $d\alpha<\frac m{m+1}$. The variance, $m=1$, gives only $\alpha<\frac12$ in position space and $\alpha<\frac14$ in phase space, as in \cite{Han3}. The fourth moment, $m=2$, gives $\alpha<\frac23$ and $\alpha<\frac13$. The exponents $1$ and $\frac12$ in Theorem \ref{thm:QElog} need $m\to\infty$ as $\alpha\to\frac1d$.
\end{rmk}

\begin{prop}[Orthonormal bases of quasimodes]\label{prop:ONQE}\leavevmode
\begin{enumerate}[(i)]
\item Let $\ve_N>0$ with $\ve_N\to0$ as $N\to\infty$. For each $N$, let $\{\psi_j\}_{j=1}^N$ be an orthonormal basis of $\Hc_N$ consisting of quasimodes of order $\ve_N$. Then there are subsets $S(N)\subset\{1,\ldots,N\}$ of full density such that \eqref{eq:QEcat} holds.
\item Theorem \ref{thm:QElog} holds with orthonormal bases of $\Hc_N$ consisting of quasimodes of order $(\log N)^{-2}$ in place of eigenbases.
\end{enumerate}
\end{prop}
\begin{proof}
(i) Let $b$ be a real-valued trigonometric polynomial with mean $0$ and degree less than $D\ge1$. Write
\[T=\min\left\{\left\lceil\ve_N^{-\frac12}\right\rceil,\left\lfloor\frac{\log N}{4\lambda}\right\rfloor\right\}.\]
Then $T\to\infty$ and $\ve_NT\to0$ as $N\to\infty$. For $N$ large, $T\ge1$ and \eqref{eq:nointerference} holds with $m=1$ and this $T$. So Proposition \ref{prop:moments} with $m=1$ and $\ve=\ve_N$ gives
\[V_N(b)=\frac1N\sum_{j=1}^N\left\langle\Op_N(b)\psi_j,\psi_j\right\rangle^2\to0\quad\text{as }N\to\infty.\]

Let $\{b_k\}_{k=1}^\infty$ be an enumeration of the real and imaginary parts of the Fourier modes $\eb_n$ with $n\in\Z^2\setminus\{0\}$. For $1\le j\le N$, write
\[Z_j=\sum_{k=1}^\infty2^{-k}\min\left\{1,\left\langle\Op_N\left(b_k\right)\psi_j,\psi_j\right\rangle^2\right\}.\]
The average $\eta_N$ of the $Z_j$ over $j$ is at most the sum of $2^{-k}\min\{1,V_N(b_k)\}$ over $k$. So $\eta_N\to0$ by dominated convergence. Let $S(N)$ be the set of $j$ with $Z_j\le\eta_N^{\frac12}$. By Chebyshev's inequality, 
\[N-|S(N)|\le\eta_N^{\frac12}N.\] 
So $S(N)$ has full density. For $j\in S(N)$ and each $k$, the minimum in $Z_j$ is at most $2^k\eta_N^{\frac12}$. So
\[\lim_{N\to\infty}\max_{j\in S(N)}\left|\left\langle\Op_N\left(b_k\right)\psi_j,\psi_j\right\rangle\right|=0\quad\text{for each }k.\]
By \eqref{eq:antiWick} and Lemma \ref{lemma:antiWick}, the same holds with the integral of $b_k\Hr_{\psi_j}$ over $\T^2$ in place of $\langle\Op_N(b_k)\psi_j,\psi_j\rangle$. So \eqref{eq:QEcat} holds for all trigonometric polynomials $a$. Since the trigonometric polynomials are dense in $C(\T^2)$ and $\Hr_{\psi_j}\,dx$ is a probability measure, \eqref{eq:QEcat} holds for all $a\in C(\T^2)$.

(ii) Let $\{\psi_j\}_{j=1}^N$ be an orthonormal basis of $\Hc_N$ consisting of quasimodes of order $(\log N)^{-2}$. We follow the proof of Theorem \ref{thm:QElog}. In Step 2, Proposition \ref{prop:moments} with $\ve=(\log N)^{-2}$ adds the term
\[Cr^{2dm}\left((\log N)^{-2}T|\Nc|\right)^{2m}\]
to the bound on the moments, where $|\Nc|\le9D^d$. By Step 1,
\[(\log N)^{-2}T|\Nc|\le C(\log N)^{d(\alpha+\kappa)-1}=C(\log N)^{-3d\kappa}.\]
So the added term is at most a constant times the bound on $r^{2dm}\Lambda^m$ in Step 2, and the proof of Theorem \ref{thm:QElog} applies, with larger constants.
\end{proof}

The order $(\log N)^{-2}$ in (ii) of Proposition \ref{prop:ONQE} cannot be replaced by $(\log N)^{-\alpha}$ with $0<\alpha<1$. Indeed, the proof of Theorem \ref{thm:ONquasimodes} with $K=\lceil\pi(\log N)^\alpha\rceil$ in place of $\lceil\log N\rceil$ gives orthonormal bases of quasimodes of order $(\log N)^{-\alpha}$. In these bases, a full density subset is not equidistributed at the scales $(\log N)^{-\tilde\alpha}$, for $\tilde\alpha\in(\alpha/2,\frac12)$ in phase space and for $\tilde\alpha\in(\alpha,1)$ in position space.

\subsection{Polynomial scales}\label{sec:poly}
For the polynomial scales, we use the bounds of Kurlberg--Rudnick \cite{KR1,KR2} on the fourth moments of the matrix elements $\langle T_N(n)\psi_j,\psi_j\rangle$. In this subsection, suppose that $M$ satisfies the parity condition in Theorem \ref{thm:density}. Then $\theta=0$, and $\Hc_N$ is the space $L^2(\Z/N\Z)$ of \cite{KR1,KR2}; see Section \ref{sec:observables}.

\begin{rmk}
The conventions of \cite{KR1,KR2} differ from ours. By \eqref{eq:TNn}, their translation operator for $n=(n_1,n_2)\in\Z^2$ is $T_N((-n_1,n_2))$. By Proposition \ref{prop:Egorov}, $\hat M$ is a multiple of their quantization of the matrix obtained from $M$ by exchanging the two diagonal entries and exchanging the two off-diagonal entries. This matrix is conjugate to $M$, has the same period modulo $N$, and satisfies the same congruence conditions as $M$. So the two operators have the same eigenbases. The bounds below only depend on $|n|$, so they transfer to our setting.
\end{rmk}

\begin{thm}[Kurlberg--Rudnick]\label{thm:KR}
There is a set $\Gc_1\subset\N$ of full density, and for each $\vk>0$ there is $C=C(M,\vk)>0$, such that the following hold for every $n\in\Z^2\setminus\{0\}$.
\begin{enumerate}[(i)]
\item For $N\in\Gc_1$ and every eigenbasis $\{\psi_j\}_{j=1}^N$ of $\hat M$,
\[\sum_{j=1}^N\left|\left\langle T_N(n)\psi_j,\psi_j\right\rangle\right|^4\le C|n|^{8+\vk}\frac{N(\log N)^{14}}{P(N)^2}.\]
\item If $M=\Id\pmod4$, then for every $N$ and every Hecke eigenbasis $\{\psi_j\}_{j=1}^N$,
\[\sum_{j=1}^N\left|\left\langle T_N(n)\psi_j,\psi_j\right\rangle\right|^4\le C|n|^{16}N^{-1+\vk}.\]
\end{enumerate}
\end{thm}

We remark that (i) of Theorem \ref{thm:KR} is \cite[Proposition 8]{KR2}, whose set of full density does not depend on $\vk$, and (ii) is \cite[Theorem 10]{KR1}. The Hecke operators are defined in \cite[Section 5]{KR1} for $M=\Id\pmod4$. Since they are commuting unitary operators that commute with $\hat M$, the space $\Hc_N$ has an orthonormal basis of joint eigenvectors of $\hat M$ and the Hecke operators, which is a Hecke eigenbasis. 

In the proofs of Theorems \ref{thm:density} and \ref{thm:Hecke}, the auxiliary exponents $\tilde\beta$ and $\vk$ are fixed in terms of $\beta$, so that the bounds depend only on $M$, $\beta$, and $N$.

For almost all eigenstates, in (i) and (ii) of Theorems \ref{thm:density} and \ref{thm:Hecke}, we use the fourth moments in Theorem \ref{thm:KR} and Chebyshev's inequality. For all eigenstates, in (iii), we need a bound on $\langle T_N(n)\psi,\psi\rangle$ for every eigenstate $\psi$, and this is where the exponents get smaller. For Hecke eigenbases, we bound a single term by the sum in (ii) of Theorem \ref{thm:KR}. This gives $N^{-\frac14+\vk}$ in place of the typical size $N^{-\frac12+\vk}$ of the matrix elements, and halves the exponents: $\frac1{20}$ and $\frac1{24}$ in (iii) of Theorem \ref{thm:Hecke}, in place of $\frac1{10}$ and $\frac1{12}$. For $N$ in a set of full density, the same argument with (i) of Theorem \ref{thm:KR} gives only a bound that decays like $e^{-\frac12(\log N)^{\nu_0}}$, slower than every power of $N$. So we use instead the bound of Bourgain \cite[Theorem 3]{Bo}, which gives the exponent $\beta_0$.

The proofs of (i) and (ii) use the following consequence of Lemma \ref{lemma:modes} and Chebyshev's inequality.

\begin{lemma}[From fourth moments to small scales]\label{lemma:fourth}
Let $\{\psi_j\}_{j=1}^N$ be an orthonormal basis of $\Hc_N$. Let $r\in(0,\frac14]$ and $D\ge1/r$. Suppose that
\[\sum_{j=1}^N\sum_{1\le|n|<D}\left|\left\langle T_N(n)\psi_j,\psi_j\right\rangle\right|^4\le A.\]
Then for each $\eta>0$, all but at most $729D^6A\eta^{-4}$ indices $j$ satisfy
\[\sup_{x\in\T^2}\left|\frac{\Hr_{\psi_j}\left(B_2(x,r)\right)}{\pi r^2}-1\right|\le\frac{c_2}\pi\left(\frac1{Dr}+\eta\right).\]
The same holds in position space, with the sums over the vectors $(0,l)$ with $1\le|l|<D$, with $8D^3$ in place of $729D^6$, and with $c_1/2$ in place of $c_2/\pi$.
\end{lemma}
\begin{proof}
Write
\[Y_j=\sum_{1\le|n|<D}\left|\left\langle T_N(n)\psi_j,\psi_j\right\rangle\right|.\]
The sum has at most $9D^2$ terms. So H\"older's inequality gives
\[Y_j^4\le729D^6\sum_{1\le|n|<D}\left|\left\langle T_N(n)\psi_j,\psi_j\right\rangle\right|^4.\]
Summing over $j$ and using Chebyshev's inequality, we see that $Y_j>\eta$ for at most $729D^6A\eta^{-4}$ indices $j$. For the other $j$, Lemma \ref{lemma:modes} gives the estimate. In position space, the sum has at most $2D$ terms.
\end{proof}

\begin{proof}[Proof of Theorem \ref{thm:density}]
Denote by $\Gc_0$ the set in (ii) of Proposition \ref{prop:periods}. Let $\Gc$ be the intersection of $\Gc_0$, the set $\Gc_1$ in (i) of Theorem \ref{thm:KR}, and the set of full density in Ostafe--Shparlinski--Voloch \cite[Section 8.2]{OSV}. Then $\Gc$ has full density. Let $\vk\in(0,\frac12)$. By (i) of Theorem \ref{thm:KR} and (ii) of Proposition \ref{prop:periods}, for $N\in\Gc$, every eigenbasis, and $n\ne0$,
\begin{equation}\label{eq:KRdensity}
\sum_{j=1}^N\left|\left\langle T_N(n)\psi_j,\psi_j\right\rangle\right|^4\le C|n|^{8+\vk}(\log N)^{14}e^{-2(\log N)^{\nu_0}},
\end{equation}
where $C$ depends only on $M$ and $\vk$.

(i) and (ii). We first prove (ii). Let $\beta<\tilde\beta<\frac1{16}$. Choose $\vk$ with $(16+\vk)\tilde\beta<1$. Let $r=N^{-\beta}$ and $D=\lceil N^{\tilde\beta}\rceil$. By \eqref{eq:KRdensity},
\[\sum_{j=1}^N\sum_{1\le|n|<D}\left|\left\langle T_N(n)\psi_j,\psi_j\right\rangle\right|^4\le CD^{10+\vk}(\log N)^{14}.\]
Let
\[\eta_N=N^{-\frac{1-(16+\vk)\tilde\beta}8}.\]
By Lemma \ref{lemma:fourth}, all but at most $C\eta_N^4(\log N)^{14}N$ indices $j$ satisfy
\[\sup_{x\in\T^2}\left|\frac{\Hr_{\psi_j}\left(B_2(x,r)\right)}{\pi r^2}-1\right|\le\frac{c_2}\pi\left(N^{\beta-\tilde\beta}+\eta_N\right).\]
Both bounds tend to $0$, since $\beta<\tilde\beta$ and $\eta_N\to0$. This proves (ii). For (i), let $\beta<\tilde\beta<\frac1{12}$, and choose $\vk$ with $(12+\vk)\tilde\beta<1$. The sum over the vectors $(0,l)$ with $1\le|l|<D$ is at most $CD^{9+\vk}(\log N)^{14}$, and the position space case of Lemma \ref{lemma:fourth} gives (i) in the same way.

(iii) By Ostafe--Shparlinski--Voloch \cite[Section 8.2, (8.4)]{OSV}, which makes explicit a bound of Bourgain \cite[Theorem 3]{Bo}, there are $B=B(M)\ge0$ and $C=C(M)>0$ such that
\[\left|\left\langle T_N(n)\psi,\psi\right\rangle\right|\le C|n|^BN^{-\frac1{120}}\quad\text{for }n\in\Z^2\setminus\{0\}\]
for $N\in\Gc$ large and every eigenstate $\psi$. Here, we replaced the exponent $\frac1{60}-o(1)$ of \cite{OSV} by $\frac1{120}$, and the conventions are related as in the remark on conventions in this subsection. Choose $\beta_0$ and $\tilde\beta$ with
\[0<\beta_0<\tilde\beta<\frac1{120(B+2)}.\]
Let $r=N^{-\beta_0}$ and $D=\lceil N^{\tilde\beta}\rceil$. By Lemma \ref{lemma:modes},
\[\sup_{x\in\T^2}\left|\frac{\Hr_\psi\left(B_2(x,r)\right)}{\pi r^2}-1\right|\le C\left(N^{\beta_0-\tilde\beta}+D^{B+2}N^{-\frac1{120}}\right).\]
The right-hand side does not depend on $\psi$ and tends to $0$ as $N\to\infty$. The position space estimate is analogous, with $D^{B+1}$ in place of $D^{B+2}$.
\end{proof}

\begin{proof}[Proof of Theorem \ref{thm:Hecke}]
Let $\{\psi_j\}_{j=1}^N$ be a Hecke eigenbasis.

(i) and (ii). We first prove (ii). Let $\beta<\tilde\beta<\frac1{12}$. Choose $\vk$ with $24\tilde\beta+\vk<2$. Let $r=N^{-\beta}$ and $D=\lceil N^{\tilde\beta}\rceil$. By (ii) of Theorem \ref{thm:KR},
\[\sum_{j=1}^N\sum_{1\le|n|<D}\left|\left\langle T_N(n)\psi_j,\psi_j\right\rangle\right|^4\le CD^{18}N^{-1+\vk}.\]
Let
\[\eta_N=N^{-\frac{2-24\tilde\beta-\vk}8}.\]
By Lemma \ref{lemma:fourth}, all but at most $C\eta_N^4N$ indices $j$ satisfy
\[\sup_{x\in\T^2}\left|\frac{\Hr_{\psi_j}\left(B_2(x,r)\right)}{\pi r^2}-1\right|\le\frac{c_2}\pi\left(N^{\beta-\tilde\beta}+\eta_N\right),\]
which tends to $0$. This proves (ii). For (i), the sum over the vectors $(0,l)$ is at most $CD^{17}N^{-1+\vk}$, and the same argument applies if $\beta<\tilde\beta<\frac1{10}$ and $20\tilde\beta+\vk<2$.

(iii) By (ii) of Theorem \ref{thm:KR},
\[\left|\left\langle T_N(n)\psi_j,\psi_j\right\rangle\right|\le C|n|^4N^{-\frac14+\vk}\quad\text{for all }j\text{ and }n\ne0.\]
In phase space, let $\beta<\tilde\beta<\frac1{24}$, $r=N^{-\beta}$, and $D=\lceil N^{\tilde\beta}\rceil$. By Lemma \ref{lemma:modes},
\[\sup_{x\in\T^2}\left|\frac{\Hr_{\psi_j}\left(B_2(x,r)\right)}{\pi r^2}-1\right|\le\frac{c_2}\pi N^{\beta-\tilde\beta}+CD^6N^{-\frac14+\vk}.\]
The right-hand side tends to $0$ if $\vk<\frac14-6\tilde\beta$. In position space, let $\beta<\tilde\beta<\frac1{20}$, and define $r$ and $D$ in the same way. Then the sum over $1\le|l|<D$ gives $CD^5N^{-\frac14+\vk}$, which tends to $0$ if $\vk<\frac14-5\tilde\beta$.
\end{proof}

\section{Quasimodes that are not equidistributed at small scales}\label{sec:negativequasi}
In this section, we prove Theorem \ref{thm:quasimodes}. We superpose evolved coherent states along a periodic orbit, with weights that make the superposition a quasimode of order $O((\log N)^{-1})$. Here, $C$ denotes a positive constant that depends only on $M$, whose value may change from line to line.

\subsection{The construction}\label{sec:construction}
Let $\gamma$ be a periodic orbit of prime period $T$. Let $x_0\in\gamma$. For $s\in\Z$, write $x_s=M^sx_0$. Then
\[x_s=x_{s'}\quad\text{if and only if}\quad s=s'\pmod T.\]
Suppose that
\begin{equation}\label{eq:Tsmall}
T\le\frac{\log N}{8\lambda}-\frac{\log\left(8C_M\right)}\lambda,\quad\text{that is,}\quad e^{\lambda T}\le\frac1{8C_M}N^{\frac18}.
\end{equation}
Let $L$ be the largest multiple of $T$ with
\begin{equation}\label{eq:L}
L\le\frac{\log N}{4\lambda}-\frac{\log\left(2C_M\right)}\lambda,\quad\text{so that}\quad4C_M^2e^{2\lambda L}\le\sqrt N.
\end{equation}
Suppose that $L\ge T$. Since $L$ is the largest such multiple of $T$, the right-hand side of \eqref{eq:L} is less than $L+T$. So by \eqref{eq:Tsmall},
\begin{equation}\label{eq:Llower}
L\ge\frac{\log N}{8\lambda}-C.
\end{equation}
Hence, $L$ lies between about $\frac18T_E$ and $\frac14T_E$.

For $s\in\Z$ with $|s|\le L$, let
\[w_s=\cos\left(\frac{\pi s}{2L}\right),\]
so that $w_{\pm L}=0$. For $\w\in\R$, define
\begin{equation}\label{eq:Psi}
\Psi_\w=\sum_{|s|<L}w_se^{-i\w s}u_s,\quad\text{where }u_s=\hat M^s\Phi_{x_0}.
\end{equation}

\begin{lemma}[Cosine weights]\label{lemma:weights}
For every $t\in\Z$,
\[\sum_{|s|<L,s=t\pmod T}w_s^2=\frac LT.\]
Moreover,
\[\left|w_{s-1}-w_s\right|\le\frac\pi{2L}\quad\text{for }-L<s\le L.\]
\end{lemma}
\begin{proof}
Let $k=L/T\in\N$. Let $t'\in\{0,\ldots,T-1\}$ with $t'=t\pmod T$. Since $w_{\pm L}=0$, the sum equals
\[\sum_{l=-k}^{k-1}\cos^2\left(\frac{\pi\left(t'+lT\right)}{2L}\right)=k+\frac12\Rer\left(e^{\frac{\pi it'}L}\sum_{l=-k}^{k-1}e^{\frac{\pi il}k}\right)=k.\]
Here, we used the identity $\cos^2\vartheta=\frac12(1+\cos2\vartheta)$, and the last sum vanishes, since it is the sum of all $2k$-th roots of unity. The second statement follows from the mean value theorem.
\end{proof}

So the weights put the same total mass $L/T$ on each point of $\gamma$. They also change slowly, which gives a quasimode of order $C/L$; see Figure \ref{fig:weights}.

\begin{figure}[!ht]
\centering
\begin{tikzpicture}
\draw[->] (-5.376,0) -- (5.46,0) node[right] {$s$};
\draw[->] (0,0) -- (0,2.88) node[above] {$w_s$};
\foreach \s in {-11,...,11}{
  \pgfmathsetmacro{\wt}{2.4*cos(90*\s/12)}
  \draw (0.42*\s,0) -- (0.42*\s,\wt);
  \pgfmathtruncatemacro{\rr}{mod(\s+120,4)}
  \ifnum\rr=0 \fill (0.42*\s,\wt) circle (0.072); \else \draw[fill=white] (0.42*\s,\wt) circle (0.072); \fi}
\draw (-5.04,0) node[below] {$-L$};
\draw (5.04,0) node[below] {$L$};
\draw (0,0) node[below] {$0$};
\draw[dashed] (-5.04,0) -- (-5.04,0.72);
\draw[dashed] (5.04,0) -- (5.04,0.72);
\end{tikzpicture}
\caption{The weights $w_s=\cos\frac{\pi s}{2L}$ for $L=12$ and $T=4$. The filled dots form the residue class $s=0\pmod T$.}\label{fig:weights}
\end{figure}

\subsection{Concentration}
The next lemma shows that the states $u_s$ in \eqref{eq:Psi} are concentrated near $x_s$, and that their overlaps decay exponentially with the time separation.

\begin{lemma}[Concentration of evolved coherent states]\label{lemma:localization}
For $s,s'\in\Z$ with $|s|,|s'|\le L$ and for $y\in\T^2$,
\[\left|\left\langle u_s,u_{s'}\right\rangle\right|\le13e^{-\frac\lambda2|s-s'|}e^{-\pi\sqrt N\dist\left(x_s,x_{s'}\right)^2}\quad\text{and}\quad\left|\left\langle u_s,\Phi_y\right\rangle\right|\le13e^{-\pi\sqrt N\dist\left(x_s,y\right)^2}.\]
\end{lemma}
\begin{proof}
Let $\tilde x_0\in\R^2$ be a lift of $x_0$. By \eqref{eq:MPc} and the definition of $\Phi_{x_0}$,
\[u_s=\Pi_N\hat M^s\phi_{\tilde x_0}.\]
By (ii) of Lemma \ref{lemma:periodization},
\[\left|\left\langle u_s,u_{s'}\right\rangle\right|\le\sum_{n\in\Z^2}\left|\left\langle\hat M^s\phi_{\tilde x_0},T_n\hat M^{s'}\phi_{\tilde x_0}\right\rangle\right|.\]
By \eqref{eq:metaplectic}, $T_n\hat M^{s'}=\hat M^{s'}T_{M^{-s'}n}$. So by \eqref{eq:TuTv}, $T_n\hat M^{s'}\phi_{\tilde x_0}$ is a unimodular multiple of $\hat M^{s'}\phi_{\tilde x_0+M^{-s'}n}$. Hence, Lemma \ref{lemma:Gaussian} applies with
\[\Delta=M^s\tilde x_0-M^{s'}\tilde x_0-n.\]
Since $|s|,|s'|\le L$, by \eqref{eq:L}, the coefficient of $-|\Delta|^2$ in the exponent in Lemma \ref{lemma:Gaussian} is at least
\[\frac{\pi N}{2C_M^2e^{2\lambda L}}\ge2\pi\sqrt N.\]
So each term of the sum is at most
\[\sqrt2e^{-\frac\lambda2|s-s'|}e^{-2\pi\sqrt N|\Delta|^2}.\]

Apply (iii) of Lemma \ref{lemma:coherent} with $a=2\pi\sqrt N$ and $z=\Delta+n$, which does not depend on $n$ and whose distance to $\Z^2$ is $\dist(x_s,x_{s'})$. Since $9\sqrt2\le13$, this gives the first inequality. The second follows in the same way from Lemma \ref{lemma:Gaussian} with $s'=0$ and with a lift of $y$ in place of $\tilde x_0$ in the second state.
\end{proof}

Let $\rho=N^{-\frac18}$. For $t\in\{0,\ldots,T-1\}$, denote by $B_t=B_2(x_t,\rho)$ the ball of radius $\rho$ around $x_t$. Denote by
\[U=\bigcup_{t=0}^{T-1}B_t\]
the union of these balls. By Lemma \ref{lemma:fixed} and \eqref{eq:Tsmall}, distinct points $x,y\in\gamma$ satisfy
\[\dist(x,y)\ge\frac{e^{-\lambda T}}{2C_M}\ge4\rho.\]
So the balls $B_2(x_t,2\rho)$ are disjoint.

\begin{prop}[The quasimodes]\label{prop:quasimodes}
Suppose that \eqref{eq:Tsmall} holds and $L\ge T$. There is $\w\in\R$ such that $\psi=\Psi_\w/\|\Psi_\w\|$ satisfies the following.
\begin{enumerate}[(i)]
\item $\psi$ is a quasimode of order $C(\log N)^{-1}$.
\item 
\[\Hr_\psi\left(\T^2\setminus U\right)\le Ce^{-N^{\frac14}}.\]
\item 
\[\sum_{t=0}^{T-1}\left|\Hr_\psi\left(B_t\right)-\frac1T\right|\le C\left(e^{-\frac\lambda2T}+e^{-N^{\frac14}}\right).\]
\end{enumerate}
\end{prop}
\begin{proof}
\textit{Step 1: Almost orthogonality.} Let $s,s'\in\Z$ with $|s|,|s'|\le L$. If $s'=s\pmod T$, then $x_{s'}=x_s$, and Lemma \ref{lemma:localization} gives
\[\left|\left\langle u_s,u_{s'}\right\rangle\right|\le13e^{-\frac\lambda2|s-s'|}.\]
Otherwise, $\dist(x_s,x_{s'})\ge4\rho$, and Lemma \ref{lemma:localization} gives
\[\left|\left\langle u_s,u_{s'}\right\rangle\right|\le13e^{-16\pi N^{\frac14}}.\]
By \eqref{eq:L} and since $\lambda>\frac12$, we have $2L\le\log N$. Hence,
\[\sum_{|s'|\le L}\left|\left\langle u_s,u_{s'}\right\rangle\right|\le13\sum_{l\in\Z}e^{-\frac\lambda2|l|}+13(2L+1)e^{-16\pi N^{\frac14}}\le C.\]
By Schur's test, for all complex numbers $c_s$ with $|s|\le L$,
\begin{equation}\label{eq:Schur}
\left\|\sum_{|s|\le L}c_su_s\right\|^2\le C\sum_{|s|\le L}\left|c_s\right|^2.
\end{equation}

\textit{Step 2: Choice of $\w$.} Since $\hat M$ is unitary, each $u_s$ has the same norm as $\Phi_{x_0}$. By the orthogonality of the functions $e^{-i\w s}$ on $[0,2\pi)$ and Lemma \ref{lemma:weights},
\[\frac1{2\pi}\int_0^{2\pi}\left\|\Psi_\w\right\|^2\,d\w=\sum_{|s|<L}w_s^2\left\|u_s\right\|^2=L\left\|\Phi_{x_0}\right\|^2.\]
So we can choose $\w$ such that
\[\left\|\Psi_\w\right\|^2\ge L\left\|\Phi_{x_0}\right\|^2\ge\frac L2,\]
where the second inequality follows from (iv) of Lemma \ref{lemma:coherent}.

\textit{Step 3: The quasimode estimate.} Since $\hat M$ maps $u_s$ to $u_{s+1}$ and $w_{\pm L}=0$,
\[\left(\hat M-e^{i\w}\right)\Psi_\w=e^{i\w}\sum_{s=-L+1}^L\left(w_{s-1}-w_s\right)e^{-i\w s}u_s.\]
By \eqref{eq:Schur} and Lemma \ref{lemma:weights}, the norm squared is at most
\[2LC\left(\frac\pi{2L}\right)^2\le\frac CL.\]
Dividing by the lower bound for $\|\Psi_\w\|^2$ in Step 2, we see that $\psi$ is a quasimode of order $C/L$. Since $L\ge1$, the bound \eqref{eq:Llower} shows that this order is at most $C(\log N)^{-1}$. This proves (i).

\textit{Step 4: Concentration of the Husimi function.} For $t\in\{0,\ldots,T-1\}$, write
\[\Psi_t=\sum_{|s|<L,s=t\pmod T}w_se^{-i\w s}u_s,\quad\text{so that}\quad\Psi_\w=\sum_{t=0}^{T-1}\Psi_t.\]
Let $y\in\T^2\setminus U$. Then $\dist(y,x_s)\ge\rho$ for all $s$, and Lemma \ref{lemma:localization} gives
\[\left|\left\langle\Psi_\w,\Phi_y\right\rangle\right|\le26Le^{-\pi N^{\frac14}}.\]
So
\[\Hr_{\Psi_\w}\left(\T^2\setminus U\right)\le N(26L)^2e^{-2\pi N^{\frac14}}.\]
Dividing by the lower bound for $\|\Psi_\w\|^2$ in Step 2 proves (ii), since $L\le\log N$.

Now let $y\in B_t$. Then $\dist(y,x_{t'})\ge3\rho$ for $t'\ne t$, and similarly
\[\left|\left\langle\Psi_\w-\Psi_t,\Phi_y\right\rangle\right|\le26Le^{-9\pi N^{\frac14}}.\]
In the same way,
\[\Hr_{\Psi_t}\left(\T^2\setminus B_t\right)\le N(26L)^2e^{-2\pi N^{\frac14}}.\]
By \eqref{eq:Schur},
\[\Hr_{\Psi_t}\left(\T^2\right)=\left\|\Psi_t\right\|^2\le CL.\]
For $y\in B_t$, write
\[\Hr_{\Psi_\w}(y)=N\left|\left\langle\Psi_t,\Phi_y\right\rangle+\left\langle\Psi_\w-\Psi_t,\Phi_y\right\rangle\right|^2.\]
Expand the square and integrate over $B_t$. The Cauchy--Schwarz inequality for the cross term and the three previous bounds give
\begin{equation}\label{eq:massBt}
\left|\Hr_{\Psi_\w}\left(B_t\right)-\left\|\Psi_t\right\|^2\right|\le CNL^2e^{-2\pi N^{\frac14}}.
\end{equation}

\textit{Step 5: The masses.} By Lemma \ref{lemma:weights}, the diagonal terms of $\|\Psi_t\|^2$ sum to $L\|\Phi_{x_0}\|^2/T$. The off-diagonal terms have $s=s'\pmod T$ and $s\ne s'$, so $s'=s+lT$ with $l\ne0$. Each residue class modulo $T$ contains at most $2L/T$ values of $s$ with $|s|<L$, since $L$ is a multiple of $T$. So by Step 1,
\[\left|\left\|\Psi_t\right\|^2-\frac LT\left\|\Phi_{x_0}\right\|^2\right|\le\frac{2L}T\sum_{l\ne0}13e^{-\frac\lambda2|l|T}\le\frac{CL}Te^{-\frac\lambda2T}.\]
By Step 1,
\[\sum_{t\ne t'}\left|\left\langle\Psi_t,\Psi_{t'}\right\rangle\right|\le13(2L)^2e^{-16\pi N^{\frac14}}.\]
Summing the previous estimate over $t$, we obtain
\[\left|\left\|\Psi_\w\right\|^2-L\left\|\Phi_{x_0}\right\|^2\right|\le CLe^{-\frac\lambda2T}+CL^2e^{-16\pi N^{\frac14}}.\]
Since $\psi$ is the normalization of $\Psi_\w$,
\[\Hr_\psi\left(B_t\right)=\frac{\Hr_{\Psi_\w}\left(B_t\right)}{\left\|\Psi_\w\right\|^2}.\]
By \eqref{eq:massBt} and the two previous estimates,
\begin{align*}
\left|\Hr_\psi\left(B_t\right)-\frac1T\right|&\le\frac{\left|T\left\|\Psi_t\right\|^2-\left\|\Psi_\w\right\|^2\right|}{T\left\|\Psi_\w\right\|^2}+\frac{CNL^2e^{-2\pi N^{\frac14}}}{\left\|\Psi_\w\right\|^2}\\
&\le\frac CTe^{-\frac\lambda2T}+\frac CTe^{-N^{\frac14}}.
\end{align*}
Here, we used the lower bound for $\|\Psi_\w\|^2$ in Step 2, the inequality $T\le L$, and the bound
\[NL^2e^{-2\pi N^{\frac14}}\le Ce^{-N^{\frac14}},\]
which holds since $L\le\log N$. Summing over $t$ proves (iii).
\end{proof}

\begin{rmk}[The cosine weights]
By Steps 2 and 3 of the proof of Proposition \ref{prop:quasimodes}, the square of the order of $\psi$ is at most a constant times
\[\frac{\sum_{s=-L+1}^L\left(w_{s-1}-w_s\right)^2}{\sum_{|s|<L}w_s^2}.\]
For equal weights $w_s=1$ with $|s|<L$, only the two boundary terms remain in the numerator, and the order is about $L^{-\frac12}$. The cosine weights vanish at $s=\pm L$ and change slowly, which gives the order $1/L$. Among the weights with $w_{\pm L}=0$, they minimize this ratio. Moreover, $L$ is comparable to $T_E$ for every period $T$: the orbit is traversed about $2L/T$ times, so every short orbit is allowed.
\end{rmk}

Figure \ref{fig:quasimode} shows how Proposition \ref{prop:quasimodes} leads to the failure of equidistribution: the Husimi function is concentrated near the orbit, so a ball centered at a point of the orbit has much more mass than its area, and a ball that misses the orbit has much less.

\subsection{Proof of Theorem \ref{thm:quasimodes}}
\begin{proof}[Proof of Theorem \ref{thm:quasimodes}]
\textit{Step 1: Choice of the periodic orbits.} Let
\[\tau_N=\min\left\{r_N^{-\frac12},\frac{\log N}{8\lambda}-\frac{\log\left(8C_M\right)}\lambda\right\},\]
so that $\tau_N\to\infty$.

If $\mu$ is the orbit measure of a periodic orbit $\gamma_*$, let $\gamma_N=\gamma_*$ for all $N$. Otherwise, take the periodic orbits $\gamma_j$ with prime periods $T_j\to\infty$ in Section \ref{sec:periodic}, whose orbit measures converge weakly to $\mu$. For each $N$ such that some $j$ satisfies $T_j\le\tau_N$, let $\gamma_N=\gamma_j$ for the largest such $j$. Since $\tau_N\to\infty$, the orbit $\gamma_N$ is defined for $N$ large, and this $j$ tends to infinity.

In both cases, denote by $T(N)$ the prime period of $\gamma_N$. Let $N$ be large, so that the orbit $\gamma_N$ is defined, $T(N)\le\tau_N$, and
\[r_N\le\frac14,\quad L\ge T(N),\quad4\pi T(N)r_N^2<1,\quad4T(N)r_N<1,\quad\text{and}\quad Ce^{-N^{\frac14}}\le\frac12.\]
Here, $L$ is given by \eqref{eq:L} with $T=T(N)$, and $C$ is the constant in Proposition \ref{prop:quasimodes} and in Step 3 below. Then $T(N)\le r_N^{-\frac12}$, and \eqref{eq:Tsmall} holds for $T=T(N)$.

Let $\psi_N$ be the state $\psi$ in Proposition \ref{prop:quasimodes} for $\gamma=\gamma_N$. By (i) of Proposition \ref{prop:quasimodes}, it is a quasimode of order $C(\log N)^{-1}$.

\textit{Step 2: The semiclassical measure.} Suppose first that $\mu$ is not an orbit measure. Let $f$ be a Lipschitz function on $\T^2$ with Lipschitz constant $\Lip(f)$. Denote by $x_t$ the points of $\gamma_N$. Since $f$ varies by at most $\Lip(f)\rho$ on each ball $B_t$, (ii) of Proposition \ref{prop:quasimodes} gives
\[\left|\int_{\T^2}f\Hr_{\psi_N}\,dx-\sum_{t=0}^{T(N)-1}\Hr_{\psi_N}\left(B_t\right)f\left(x_t\right)\right|\le\Lip(f)N^{-\frac18}+C\|f\|_{L^\infty(\T^2)}e^{-N^{\frac14}}.\]
By (iii) of Proposition \ref{prop:quasimodes},
\[\left|\sum_{t=0}^{T(N)-1}\Hr_{\psi_N}\left(B_t\right)f\left(x_t\right)-\int_{\T^2}f\,d\mu_{\gamma_N}\right|\le C\|f\|_{L^\infty(\T^2)}\left(e^{-\frac\lambda2T(N)}+e^{-N^{\frac14}}\right).\]
Since $T(N)\to\infty$ and $\mu_{\gamma_N}\to\mu$, the integral of $f\Hr_{\psi_N}$ tends to the integral of $f$ with respect to $\mu$. Since the Lipschitz functions are dense in $C(\T^2)$, this proves (i).

Suppose now that $\mu=\mu_{\gamma_*}$. By (ii) of Proposition \ref{prop:quasimodes}, every weak limit of $\Hr_{\psi_N}\,dx$ along a subsequence is supported on $\gamma_*$. By (i) of Proposition \ref{prop:quasimodes} and Lemma \ref{lemma:invariance}, it is $M$-invariant. So it gives the same mass to each point of $\gamma_*$, that is, it equals $\mu_{\gamma_*}$. Since the probability measures on $\T^2$ form a compact set in the weak topology, $\Hr_{\psi_N}\,dx\to\mu_{\gamma_*}$. This proves (i).

\textit{Step 3: Failure of equidistribution.} Write $T=T(N)$. Let $r\in[4N^{-\frac18},r_N]$. Then $r\le\frac14$ and $\rho\le r/4$. By (ii) of Proposition \ref{prop:quasimodes}, the balls $B_t$ have total mass at least $1-Ce^{-N^{\frac14}}$. So
\[\Hr_{\psi_N}\left(B_t\right)\ge\frac{1-Ce^{-N^{\frac14}}}T\quad\text{for some }t.\]
Since $B_t\subset B_2(x_t,r)$,
\[\sup_{x\in\T^2}\frac{\Hr_{\psi_N}\left(B_2(x,r)\right)}{\pi r^2}\ge\frac{1-Ce^{-N^{\frac14}}}{\pi Tr^2}\ge\frac{1-Ce^{-N^{\frac14}}}{\pi Tr_N^2}.\]
Write $x_t=(q_t,p_t)$. Since $B_t\subset B_1(q_t,r-\rho)\times\T$, Lemma \ref{lemma:positionHusimi} gives
\[\left|\psi_N\right|^2\left(B_1\left(q_t,r\right)\right)\ge\Hr_{\psi_N}\left(B_t\right)-2e^{-2\pi N^{\frac34}}.\]
Since $T\le\log N$, after enlarging $C$, we obtain
\[\sup_{q\in\T}\frac{\left|\psi_N\right|^2\left(B_1(q,r)\right)}{2r}\ge\frac{1-Ce^{-N^{\frac14}}}{2Tr_N}.\]

For the infimum in phase space, Step 1 gives $4\pi Tr^2<1$. So by (i) of Lemma \ref{lemma:orbitscale}, there is $z\in\T^2$ such that $B_2(z,2r)$ contains no point of $\gamma_N$. Since $\rho<r$, the ball $B_2(z,r)$ is disjoint from $U$. By (ii) of Proposition \ref{prop:quasimodes} and since $r\ge4N^{-\frac18}$,
\[\inf_{x\in\T^2}\frac{\Hr_{\psi_N}\left(B_2(x,r)\right)}{\pi r^2}\le\frac{Ce^{-N^{\frac14}}}{\pi r^2}\le CN^{\frac14}e^{-N^{\frac14}}.\]

For the infimum in position space, Step 1 gives $4Tr<1$. So by (ii) of Lemma \ref{lemma:orbitscale}, there is $z\in\T$ such that $B_1(z,2r)\times\T$ contains no point of $\gamma_N$. Since $\rho<r/2$, the strip $B_1(z,r+\rho)\times\T$ is disjoint from $U$. By Lemma \ref{lemma:positionHusimi} and (ii) of Proposition \ref{prop:quasimodes},
\[\left(1-2e^{-2\pi N^{\frac34}}\right)\left|\psi_N\right|^2\left(B_1(z,r)\right)\le\Hr_{\psi_N}\left(B_1(z,r+\rho)\times\T\right)\le Ce^{-N^{\frac14}}.\]
Since $2e^{-2\pi N^{\frac34}}\le\frac12$ and $r\ge4N^{-\frac18}$, we obtain
\[\inf_{q\in\T}\frac{\left|\psi_N\right|^2\left(B_1(q,r)\right)}{2r}\le CN^{\frac18}e^{-N^{\frac14}}.\]

Let
\[\delta_N=\max\left\{2\pi T(N)r_N^2,4T(N)r_N,CN^{\frac14}e^{-N^{\frac14}}\right\}.\]
Since $T(N)\le r_N^{-\frac12}$, we have $\delta_N\to0$. By Step 1, the numerators $1-Ce^{-N^{\frac14}}$ above are at least $\frac12$. So the four estimates above prove (ii).
\end{proof}

\begin{rmk}[Smaller scales]
Let $\beta\in(0,\frac14]$. Replace the condition \eqref{eq:L} by
\[4C_M^2e^{2\lambda L}\le N^\beta.\]
Similarly, replace \eqref{eq:Tsmall} by
\[T\le\frac{\beta\log N}{4\lambda}-\frac{\log\left(8C_M\right)}\lambda,\]
and replace the radius $\rho=N^{-\frac18}$ by $\rho=N^{-\frac12+\beta}$. Then the coefficient of $-|\Delta|^2$ in the exponent in Lemma \ref{lemma:Gaussian} is at least $2\pi N^{1-\beta}$. So the Gaussian factors at distance $\rho$ become $e^{-\pi N^\beta}$. The distinct points of the orbit still have distance at least $4\rho$, and
\[L\ge\frac{\beta\log N}{4\lambda}-C.\]
The proofs of Proposition \ref{prop:quasimodes} and Theorem \ref{thm:quasimodes} apply with $e^{-N^{\frac14}}$ replaced by $e^{-N^\beta}$ and with the polynomial factors in $N$ adjusted. They give quasimodes of order $C(\beta\log N)^{-1}$ that satisfy (ii) of Theorem \ref{thm:quasimodes} for all $r\in[4N^{-\frac12+\beta},r_N]$.
\end{rmk}

\section{Full density quasimodes that are not equidistributed at logarithmic scales}\label{sec:negativeeigen}
In this section, we prove Theorems \ref{thm:ONquasimodes} and \ref{thm:shortperiod}. We construct the bases in spectral subspaces of $\hat M$ from projections of coherent states.

\subsection{Concentrated orthonormal bases}
For a subspace $F\subset\Hc_N$, denote by $\Pi_F$ the orthogonal projection onto $F$. By (i) of Lemma \ref{lemma:coherent}, the squared overlap of a unit vector and $\Phi_x$ is $1/N$ on average over $x\in\T^2$. The next proposition gives an orthonormal basis of a subspace $V$ whose elements have much larger overlaps. We project the coherent states onto $V$, and then onto the subspace that remains after the vectors already chosen are removed. Some projection is large as long as the remaining subspace has large dimension.

\begin{prop}[Orthonormal bases from coherent states]\label{prop:coherentbasis}
Let $V\subset\Hc_N$ be a subspace of dimension $d\ge1$. There are points $x_1,\ldots,x_d\in\T^2$ and an orthonormal basis $\{v_k\}_{k=1}^d$ of $V$ such that
\[\left|\left\langle v_k,\Phi_{x_k}\right\rangle\right|^2\ge\frac{d-k+1}N\quad\text{for }1\le k\le d.\]
\end{prop}
\begin{proof}
Let $F\subset\Hc_N$ be a subspace. Applying (i) of Lemma \ref{lemma:coherent} to the elements of an orthonormal basis of $F$ gives
\[N\int_{\T^2}\left\|\Pi_F\Phi_x\right\|^2\,dx=\dim F.\]
We choose $x_k$ and $v_k$ inductively. Suppose that $v_1,\ldots,v_{k-1}$ are chosen. Denote by $F_k$ the orthogonal complement of $\Span\{v_1,\ldots,v_{k-1}\}$ in $V$. Since $\dim F_k=d-k+1$ and $\T^2$ has measure $1$, there is $x_k\in\T^2$ with
\[\left\|\Pi_{F_k}\Phi_{x_k}\right\|^2\ge\frac{d-k+1}N>0.\]
Let
\[v_k=\frac{\Pi_{F_k}\Phi_{x_k}}{\left\|\Pi_{F_k}\Phi_{x_k}\right\|}.\]
Then $v_k\in F_k$ is orthogonal to $v_1,\ldots,v_{k-1}$, and
\[\left\langle v_k,\Phi_{x_k}\right\rangle=\frac{\left\langle\Pi_{F_k}\Phi_{x_k},\Phi_{x_k}\right\rangle}{\left\|\Pi_{F_k}\Phi_{x_k}\right\|}=\left\|\Pi_{F_k}\Phi_{x_k}\right\|.\]
After $d$ steps, $v_1,\ldots,v_d$ form an orthonormal basis of $V$.
\end{proof}

In this proof, $v_k$ is the normalized projection of $\Phi_{x_k}$ onto $F_k$. So $\{v_k\}_{k=1}^d$ is the basis that the Gram--Schmidt process gives from $\Pi_V\Phi_{x_1},\ldots,\Pi_V\Phi_{x_d}$. The points are not fixed in advance: each $x_k$ is chosen after $v_1,\ldots,v_{k-1}$, so that the projection of $\Phi_{x_k}$ onto the remaining subspace $F_k$ is large. The vectors $\Pi_V\Phi_{x_k}$ need not be almost orthogonal. The points chosen in this process are all distinct, since $\Phi_{x_j}$ is orthogonal to $F_k$ for $j<k$. For each $\eta\in(0,1)$, the first $(1-\eta)d$ vectors have squared overlaps at least $\eta d/N$.

The next proposition combines Proposition \ref{prop:coherentbasis} and Lemma \ref{lemma:overlap} over an orthogonal decomposition of $\Hc_N$.

\begin{prop}[Concentrated orthonormal bases]\label{prop:localizedON}
Let $\Hc_N=V_1\oplus\cdots\oplus V_K$ be an orthogonal decomposition, where some $V_l$ may be zero. There are an orthonormal basis $\{\psi_j\}_{j=1}^N$ of $\Hc_N$, with each $\psi_j$ in one of the $V_l$, and points $x_j=(q_j,p_j)\in\T^2$ with the following property. For each $\eta\in(0,1)$, there is $S\subset\{1,\ldots,N\}$ with $|S|\ge(1-\eta)N$ such that
\[\Hr_{\psi_j}\left(B_2\left(x_j,r\right)\right)\ge\frac\eta K-49e^{-\frac\pi4Nr^2}\quad\text{and}\quad\left|\psi_j\right|^2\left(B_1\left(q_j,r\right)\right)\ge\frac\eta K-51e^{-\frac\pi{16}Nr^2}\]
for all $j\in S$ and $r\in(0,\frac12]$.
\end{prop}
\begin{proof}
Write $d_l=\dim V_l$. For each $l$ with $d_l\ge1$, take the basis of $V_l$ and the points in Proposition \ref{prop:coherentbasis}. Together, these bases form an orthonormal basis $\{\psi_j\}_{j=1}^N$ of $\Hc_N$. Let $\eta\in(0,1)$. In each $V_l$, retain the basis vectors $v_k$ with
\[d_l-k+1\ge\frac{\eta N}K.\]
By Proposition \ref{prop:coherentbasis}, they satisfy
\[\left|\left\langle v_k,\Phi_{x_k}\right\rangle\right|^2\ge\frac{d_l-k+1}N\ge\frac\eta K.\]
The other basis vectors of $V_l$ have $d_l-k+1<\eta N/K$, so there are fewer than $\eta N/K$ of them; see Figure \ref{fig:count}. So fewer than $\eta N$ vectors are excluded in total. Let $S$ be the set of $j$ such that $\psi_j$ is one of the retained vectors. Then $|S|\ge(1-\eta)N$.
For $j\in S$, Lemma \ref{lemma:overlap} gives the two inequalities.
\end{proof}

\begin{figure}[!ht]
\centering
\begin{tikzpicture}
\draw[->] (0,0) -- (9.6,0) node[right] {$k$};
\draw[->] (0,0) -- (0,3.9);
\draw[dashed] (0,1.2) -- (9.3,1.2);
\draw (0,1.2) node[left] {$\eta/K$};
\draw (0,3.4) node[left] {$d_l/N$};
\draw[thick] (0.3,3.4) -- (5.8,1.2);
\draw (5.8,1.2) -- (8.675,0.05);
\draw (0.3,0.9) -- (2.425,0.05);
\draw[dashed] (5.8,1.2) -- (5.8,0);
\draw (0.3,0.1) -- (0.3,-0.1) node[below] {$1$};
\draw (8.675,0.1) -- (8.675,-0.1) node[below] {$d_l$};
\draw[<->] (5.8,-0.9) -- (8.675,-0.9);
\draw (7.24,-0.9) node[below] {\small fewer than $\eta N/K$};
\draw (3.3,2.3) node[above right] {\small retained, $d_l\ge\eta N/K$};
\draw (1.5,0.5) node[above right] {\small $d_l<\eta N/K$};
\end{tikzpicture}
\caption{The lower bound $(d_l-k+1)/N$ of the squared overlap of the $k$-th basis vector of $V_l$ in the proof of Proposition \ref{prop:localizedON}, for a subspace with $d_l\ge\eta N/K$ and for one with $d_l<\eta N/K$. Thick: the retained vectors. Thin: the excluded vectors, fewer than $\eta N/K$ in each subspace.}\label{fig:count}
\end{figure}

Proposition \ref{prop:localizedON} holds for every orthogonal decomposition of $\Hc_N$; it uses neither the dynamics of $M$ nor the operator $\hat M$. Its content is that a subspace of dimension $d$ contains many orthonormal states, each of which puts mass about $d/N$ near a point that may depend on the state. At the scale $r$, this mass is much larger than $\pi r^2$ if $d/N$ is much larger than $r^2$. We choose $\eta\to0$ slowly enough that the mass lower bound in Proposition \ref{prop:localizedON}, divided by the volume of the ball, tends to infinity. This is possible when $r^2K\to0$ in phase space or $rK\to0$ in position space, provided that the Gaussian error is negligible compared with the volume of the ball.

\subsection{Orthonormal quasimodes}
Split $[0,2\pi)$ into the $K$ intervals
\[J_l=\left[\frac{2\pi(l-1)}K,\frac{2\pi l}K\right)\quad\text{for }1\le l\le K.\]
Denote by $V_l$ the spectral subspace\label{idx:spectral} of $\hat M$ for $J_l$, that is, the span of the eigenvectors of $\hat M$ with eigenvalues $e^{i\w}$, $\w\in J_l$. Since $\hat M$ is unitary, these subspaces give the orthogonal decomposition
\[\Hc_N=V_1\oplus\cdots\oplus V_K.\]
Denote by $\w_l$ the midpoint of $J_l$. For $\psi\in V_l$ with $\|\psi\|=1$,
\[\left\|\left(\hat M-e^{i\w_l}\right)\psi\right\|\le\sup_{\w\in J_l}\left|e^{i\w}-e^{i\w_l}\right|\le2\sin\frac\pi{2K}\le\frac\pi K.\]
So every normalized state in $V_l$ is a quasimode of order $\pi/K$.

The next lemma passes from one radius to all smaller radii.

\begin{lemma}[Smaller radii]\label{lemma:smaller}
Let $d\in\{1,2\}$ and $\mu$ be a positive measure on $\T^d$. Then for $x\in\T^d$ and $0<r\le s\le\frac18$,
\[\sup_{y\in\T^d}\frac{\mu\left(B_d(y,r)\right)}{\Vol\left(B_d(0,r)\right)}\ge\frac1{2^d}\frac{\mu\left(B_d(x,s)\right)}{\Vol\left(B_d(0,s)\right)}.\]
\end{lemma}
\begin{proof}
If $z\in B_d(x,s)$, then $B_d(z,r)\subset B_d(x,s+r)$. So by Fubini's theorem,
\[\int_{B_d(x,s+r)}\mu\left(B_d(y,r)\right)\,dy=\int_{\T^d}\Vol\left(B_d(z,r)\cap B_d(x,s+r)\right)\,d\mu(z)\ge\Vol\left(B_d(0,r)\right)\mu\left(B_d(x,s)\right).\]
The left-hand side is at most $\Vol(B_d(0,s+r))$ times the supremum of $\mu(B_d(y,r))$ over $y\in\T^d$. Since $s+r\le2s$, we have 
\[\Vol\left(B_d(0,s+r)\right)\le2^d\Vol\left(B_d(0,s)\right).\]
\end{proof}

\begin{proof}[Proof of Theorem \ref{thm:ONquasimodes}]
For each $N$, let $K=\lceil\log N\rceil$. Let $\{\psi_j\}_{j=1}^N$ and $x_j=(q_j,p_j)$ be the orthonormal basis and the points in Proposition \ref{prop:localizedON} for the spectral subspaces $V_1,\ldots,V_K$. Each $\psi_j$ is a quasimode of order $\pi/K$, hence of order $\pi(\log N)^{-1}$.

\textit{(i).} Since $\pi(\log N)^{-1}\to0$, (i) of Proposition \ref{prop:ONQE} gives subsets $S_1(N)\subset\{1,\ldots,N\}$ of full density along which \eqref{eq:QEcat} holds.

\textit{(ii).} Let $\alpha>\frac12$. Let $N$ be large, so that $s_N\le\frac18$. Denote by $S_2(N)$ the set in Proposition \ref{prop:localizedON} for
\[\eta=\eta_N=(\log N)^{-\frac{2\alpha-1}2}.\]
Since $K\le2\log N$, Proposition \ref{prop:localizedON} gives
\[\frac{\Hr_{\psi_j}\left(B_2\left(x_j,s_N\right)\right)}{\pi s_N^2}\ge R_N\quad\text{for }j\in S_2(N),\]
where, by the choice of $\eta_N$,
\[R_N=\frac{(\log N)^{\frac{2\alpha-1}2}}{2\pi}-\frac{49(\log N)^{2\alpha}}\pi e^{-\frac\pi4N(\log N)^{-2\alpha}}.\]
So $R_N\to\infty$. By Lemma \ref{lemma:smaller} with $\mu=\Hr_{\psi_j}\,dx$ and $s=s_N$,
\[\sup_{x\in\T^2}\frac{\Hr_{\psi_j}\left(B_2(x,r)\right)}{\pi r^2}\ge\frac{R_N}4\quad\text{for }j\in S_2(N)\text{ and }r\in(0,s_N].\]
Proposition \ref{prop:localizedON} also gives $|S_2(N)|\ge(1-\eta_N)N$. So the subsets $S_2(N)$ have full density. Let
\[S(N)=S_1(N)\cap S_2(N)\quad\text{and}\quad\delta_N=4R_N^{-1}.\]
Then the subsets $S(N)$ have full density, $\delta_N\to0$, and (i) and (ii) hold.

\textit{Position space.} Let $\alpha>1$. Denote by $S_2(N)$ the set in Proposition \ref{prop:localizedON} for
\[\eta=\eta_N=(\log N)^{-\frac{\alpha-1}2}.\]
The second inequality in Proposition \ref{prop:localizedON} gives
\[\frac{\left|\psi_j\right|^2\left(B_1\left(q_j,s_N\right)\right)}{2s_N}\ge\frac14(\log N)^{\frac{\alpha-1}2}-\frac{26}{s_N}e^{-\frac\pi{16}Ns_N^2}\quad\text{for }j\in S_2(N).\]
The right-hand side tends to infinity, and the rest of the argument is the same as for (ii). For (i), let $a\in C(\T)$. By (ii) of Lemma \ref{lemma:coherent}, the integral of $a\Hr_{\psi_j}$ over $\T^2$ is the integral of the convolution $a*g_N$ with respect to $|\psi_j|^2$, and $a*g_N\to a$ uniformly. So \eqref{eq:QEcat} gives the same limit for the measures $|\psi_j|^2$ on $\T$.
\end{proof}

\subsection{Short periods}\label{sec:shortperiods}
For $N\in\Pc_C$, the quantum cat map has few distinct eigenvalues, and Proposition \ref{prop:localizedON} applies to its eigenspaces.

\begin{proof}[Proof of Theorem \ref{thm:shortperiod}]
We prove the phase space statement; the position space statement is the same. Let $N\in\Pc_C$. Denote by $K$ the number of distinct eigenvalues of $\hat M$. By Lemma \ref{lemma:fewvalues} and the definition of $\Pc_C$,
\begin{equation}\label{eq:Kshort}
K\le|\Tr M-2|P(N)\le|\Tr M-2|C\log N.
\end{equation}
Let $\{\psi_j\}_{j=1}^N$ be the eigenbasis in Proposition \ref{prop:localizedON} for the eigenspaces of $\hat M$. The proof of (ii) of Theorem \ref{thm:ONquasimodes} with $\alpha=\alpha_2$ applies, with the bound $K\le2\log N$ replaced by \eqref{eq:Kshort}. This changes only the constants in the lower bounds. It gives subsets $S_3(N)$ of full density and $\delta_{3,N}\to0$ such that the estimate in (iii) holds with $\delta_{3,N}$ in place of $\delta_N$ for $j\in S_3(N)$.

Let $S_1(N)$ be the subsets of full density in the quantum ergodicity theorem \eqref{eq:QEcat} for this eigenbasis. By (ii) of Theorem \ref{thm:QElog} with $\alpha=\alpha_1$, there are subsets $S_2(N)$ of full density and $\delta_{2,N}\to0$ such that the estimate in (ii) holds with $\delta_{2,N}$ in place of $\delta_N$ for $j\in S_2(N)$. Let
\[S(N)=S_1(N)\cap S_2(N)\cap S_3(N)\quad\text{and}\quad\delta_N=\max\left\{\delta_{2,N},\delta_{3,N}\right\}.\]
Then the subsets $S(N)$ have full density as $N\to\infty$ in $\Pc_C$, $\delta_N\to0$, and (i), (ii), and (iii) hold.
\end{proof}

\begin{rmk}[Other periods]
The proof of (iii) of Theorem \ref{thm:shortperiod} uses the period only through the number $K$ of distinct eigenvalues of $\hat M$. Let $\Gc\subset\N$ be infinite and $s_N>0$ for $N\in\Gc$. Suppose that
\[s_N^2K\to0\quad\text{and}\quad\frac{Ns_N^2}{\log N}\to\infty\quad\text{as }N\to\infty\text{ in }\Gc.\]
Then the same proof, with $\eta_N=(s_N^2K)^{\frac12}$, gives for each $N\in\Gc$ an eigenbasis in which a full density subset satisfies (iii) of Theorem \ref{thm:shortperiod}, with $\Gc$ in place of $\Pc_C$. By Lemma \ref{lemma:fewvalues}, the first condition holds if $s_N^2P(N)\to0$. So, under the second condition, the equidistribution of almost all eigenstates in phase space at the scales $s_N$ along $\Gc$ requires that $s_N^2P(N)$ is bounded below by a positive constant. In position space, the same holds with $s_NK$ and $s_NP(N)$ in place of $s_N^2K$ and $s_N^2P(N)$.
\end{rmk}

\begin{figure}[!ht]
\centering
\includegraphics{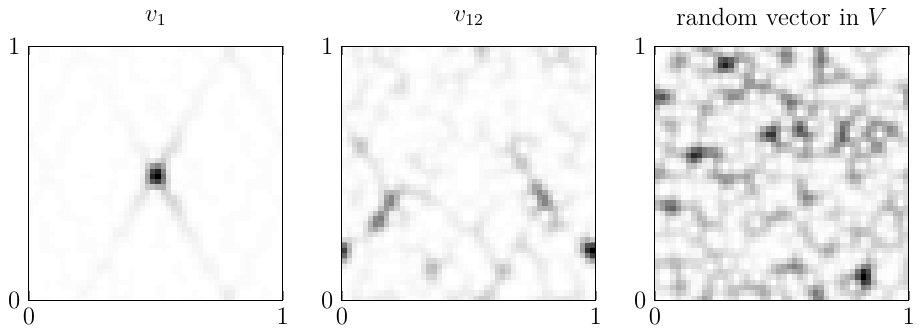}
\caption{Husimi functions of three eigenstates of the quantum cat map in \eqref{eq:example}, in an eigenspace $V$ of dimension $23$. Darker means larger, and each picture is scaled to its own maximum. Left: the first vector $v_1$ of the basis in Proposition \ref{prop:coherentbasis}, concentrated near the fixed point $(\frac12,\frac12)$ and its stable and unstable lines; the maximum of its Husimi function is about $72$. Middle: the vector $v_{12}$, with maximum about $23$. Right: a random unit vector in $V$, with maximum about $7$. For comparison, the uniform density on $\T^2$ is $1$.}\label{fig:numerics}
\end{figure}

Figure \ref{fig:numerics} illustrates Theorem \ref{thm:shortperiod} numerically. Let
\begin{equation}\label{eq:example}
M=\begin{pmatrix}2&1\\3&2\end{pmatrix}\quad\text{and}\quad N=209,
\end{equation}
so that $M$ satisfies the parity condition in Theorem \ref{thm:density}, $\theta=0$, and $N=a_5$ in (iii) of Proposition \ref{prop:periods}. Then $P(N)=10$, and $\hat M$ has $10$ distinct eigenvalues, with multiplicities between $19$ and $23$. In an eigenspace $V$ of dimension $23$, we compute the basis $\{v_k\}_{k=1}^{23}$ of Proposition \ref{prop:coherentbasis}, choosing each $x_k$ to maximize the norm of the remaining projection over a finite grid.

\appendix
\section{Translation operators on the torus}\label{sec:apptranslation}
In this appendix, we prove Lemmas \ref{lemma:periodization} and \ref{lemma:trace}. 

\periodizationlemma*
\begin{proof}
(i) Since $q_{k+N}=q_k+1$, the coefficients of $\Pi_N\vp$ satisfy \eqref{eq:states}.

(ii) By definition,
\[\left\langle\Pi_N\chi,\Pi_N\vp\right\rangle=\frac1N\sum_{k=0}^{N-1}\sum_{l,l'\in\Z}e^{-i(l-l')\theta_1}\chi\left(q_k-l\right)\ol{\vp\left(q_k-l'\right)}.\]
Substitute $j=k-Nl$ and $m=l'-l$. Since $q_k-l=q_j$, the right-hand side equals
\[\sum_{m\in\Z}e^{im\theta_1}\frac1N\sum_{j\in\Z}\chi\left(q_j\right)\ol{\vp\left(q_j-m\right)}.\]
By the Poisson summation formula for $f(q)=\chi(q)\ol{\vp(q-m)}$,
\[\frac1N\sum_{j\in\Z}f\left(q_j\right)=\sum_{m'\in\Z}e^{im'\theta_2}\int_\R f(q)e^{-2\pi iNm'q}\,dq.\]
By \eqref{eq:Tv},
\[\int_\R f(q)e^{-2\pi iNm'q}\,dq=\left\langle\chi,T_{(0,m')}T_{(m,0)}\vp\right\rangle_{L^2(\R)}.\]
By \eqref{eq:TuTv}, $T_{(0,m')}T_{(m,0)}$ is a unimodular multiple of $T_{(m,m')}$. This proves (ii), with $c_0=1$.

(iii) The first identity follows from \eqref{eq:Tv} and the definition of $\Pi_N$, since
\[q_k-\frac{n_1}N=q_{k-n_1}\quad\text{and}\quad e^{2\pi in_2(q_k-l)}=e^{2\pi in_2q_k}\quad\text{for }l\in\Z.\]
Take $n=(N,0)$ and $n=(0,N)$ in the first identity. Since $\Pi_N\vp\in\Hc_{N,\theta}$ by (i), we obtain
\[\Pi_NT_{(1,0)}\vp=e^{i\theta_1}\Pi_N\vp\quad\text{and}\quad\Pi_NT_{(0,1)}\vp=e^{i\theta_2}\Pi_N\vp.\]
By \eqref{eq:TuTv}, $T_n$ is a unimodular multiple of $T_{(1,0)}^{n_1}T_{(0,1)}^{n_2}$. This gives the second identity.
\end{proof}

\tracelemma*
\begin{proof}
(i) If $n\in N\Z^2$, then $T_N(n)$ is a unimodular multiple of the identity. Suppose that $n\notin N\Z^2$. Then there is $u\in\{(1,0),(0,1)\}$ with $u\wedge n\notin N\Z$. By \eqref{eq:TNcomm},
\[T_N(u)T_N(n)T_N(u)^*=e^{\frac{2\pi i}Nu\wedge n}T_N(n).\]
The two sides have the same trace, and $e^{\frac{2\pi i}Nu\wedge n}\ne1$. So $\Tr T_N(n)=0$.

(ii) By (i) and \eqref{eq:TNcomm}, the $N^2$ operators $T_N(n)$ with $n\in\{0,\ldots,N-1\}^2$ are orthogonal with respect to the Hilbert--Schmidt inner product. So they form a basis of the space of operators on $\Hc_{N,\theta}$, and we can write
\[A=\sum_{n\in\{0,\ldots,N-1\}^2}a_nT_N(n).\]
Since $A$ commutes with $T_N(u)$, \eqref{eq:TNcomm} gives
\[A=T_N(u)AT_N(u)^*=\sum_{n\in\{0,\ldots,N-1\}^2}a_ne^{\frac{2\pi i}Nu\wedge n}T_N(n)\quad\text{for }u\in\Z^2.\]
So $a_n=0$ unless $u\wedge n\in N\Z$ for all $u\in\Z^2$, that is, unless $n=0$.
\end{proof}

\section{Coherent states and Husimi functions}\label{sec:appcoherent}
In this appendix, we prove Lemmas \ref{lemma:coherent}, \ref{lemma:antiWick}, \ref{lemma:positionHusimi}, and \ref{lemma:overlap}.

\coherentlemma*
\begin{proof}
(i) and (ii). Let $x=(q,p)\in[0,1)^2$. By \eqref{eq:Tv},
\[\phi_x(s)=e^{-\pi iNqp+2\pi iNps}\phi_0(s-q).\]
So for $0\le k,k'<N$,
\[\left(\Phi_x\right)_k\ol{\left(\Phi_x\right)_{k'}}=\frac1N\sum_{l,l'\in\Z}e^{-i(l-l')\theta_1}e^{2\pi ip(k-k'-N(l-l'))}\phi_0\left(q_k-l-q\right)\phi_0\left(q_{k'}-l'-q\right).\]
Since $|k-k'|<N$, integration over $p\in[0,1)$ leaves only the terms with $k=k'$ and $l=l'$:
\[\int_0^1\left(\Phi_x\right)_k\ol{\left(\Phi_x\right)_{k'}}\,dp=\frac{\delta_{kk'}}N\sum_{l\in\Z}\phi_0\left(q_k-l-q\right)^2=\frac{\delta_{kk'}}Ng_N\left(q-q_k\right).\]
Since
\[\left\langle\psi,\Phi_x\right\rangle=\sum_{k=0}^{N-1}\psi_k\ol{\left(\Phi_x\right)_k},\]
this proves (ii). Since $g_N$ is the periodization of $\phi_0^2$ and $\|\phi_0\|_{L^2(\R)}=1$, its integral over $\T$ is $1$. Integrating (ii) over $q$ proves (i).

(iii) Let $a\ge2$ and $z\in\R^2$. Then
\[|z-n|^2\ge\frac12\left(\dist\left(z,\Z^2\right)^2+|z-n|^2\right)\quad\text{for }n\in\Z^2.\]
Moreover, comparing a sum with an integral gives
\[\sum_{j\in\Z}e^{-b(j-u)^2}\le1+\sqrt{\frac\pi b}\le3\quad\text{for }b\ge1\text{ and }u\in\R.\]
Hence,
\[\sum_{n\in\Z^2}e^{-a|z-n|^2}\le e^{-\frac a2\dist(z,\Z^2)^2}\sum_{n\in\Z^2}e^{-\frac a2|z-n|^2}\le9e^{-\frac a2\dist(z,\Z^2)^2}.\]

(iv) By (ii) of Lemma \ref{lemma:periodization} and \eqref{eq:overlap}, for lifts of $x$ and $y$ to $\R^2$,
\[\left|\left\langle\Phi_x,\Phi_y\right\rangle\right|\le\sum_{n\in\Z^2}\left|\left\langle\phi_x,T_n\phi_y\right\rangle\right|=\sum_{n\in\Z^2}e^{-\frac\pi2N|x-y-n|^2}.\]
Since $\pi N/2\ge2$, (iii) with $a=\pi N/2$ and $z=x-y$ gives the first inequality. For the second, (ii) of Lemma \ref{lemma:periodization} gives
\[\left|\left\|\Phi_x\right\|^2-1\right|\le\sum_{n\ne0}e^{-\frac\pi2N|n|^2}\le5e^{-\frac\pi2N}.\]
\end{proof}

\antiwicklemma*
\begin{proof}
Let $m\in\Z^2$. By (iii) of Lemma \ref{lemma:periodization} and \eqref{eq:TuTv}, $T_N(m)^*\Phi_x$ is a unimodular multiple of $\Phi_{x-m/N}$. So
\[T_N(m)^*\Op_N^\aw(a)T_N(m)=\Op_N^\aw\left(a\left(\cdot+\frac mN\right)\right).\]
By the definition of $\eb_n$ and by \eqref{eq:TNcomm},
\[\eb_n\left(x+\frac mN\right)=e^{\frac{2\pi i}Nn\wedge m}\eb_n(x)\quad\text{and}\quad T_N(m)^*T_N(n)T_N(m)=e^{\frac{2\pi i}Nn\wedge m}T_N(n).\]
So the operator $T_N(n)^*\Op_N^\aw(\eb_n)$ commutes with all $T_N(m)$. By (ii) of Lemma \ref{lemma:trace}, it is $g\Id$ for some $g\in\C$. Then
\[g=\frac1N\Tr\left(T_N(n)^*\Op_N^\aw\left(\eb_n\right)\right)=\int_{[0,1)^2}\eb_n(x)\left\langle T_N(n)^*\Phi_x,\Phi_x\right\rangle\,dx.\]
By Lemma \ref{lemma:periodization}, with the constants $c_l$ in (ii) of that lemma,
\[\left\langle T_N(n)^*\Phi_x,\Phi_x\right\rangle=\sum_{l\in\Z^2}c_l\left\langle T_{-n/N}\phi_x,T_l\phi_x\right\rangle_{L^2(\R)}.\]
By \eqref{eq:TuTv},
\[T_{-x}T_vT_x=e^{2\pi iNv\wedge x}T_v.\]
By \eqref{eq:Tv} and a Gaussian integral,
\[\left\langle T_v\phi_0,\phi_0\right\rangle=e^{-\frac{\pi N}2|v|^2}.\]
So the term with $l\in\Z^2$ equals a constant independent of $x$ times $e^{-2\pi iN(l+n/N)\wedge x}$. Its product with $\eb_n(x)$ has integral $0$ over $[0,1)^2$ if $l\ne0$. The term with $l=0$ equals $\ol{\eb_n(x)}e^{-\frac{\pi|n|^2}{2N}}$. Hence,
\[g=e^{-\frac{\pi|n|^2}{2N}}.\]
This proves the first identity. The formula for $a\in C^\infty(\T^2)$ follows, since the Fourier series of $a$ converges uniformly and $\|\Op_N^\aw(b)\|\le\|b\|_{L^\infty(\T^2)}$ for $b\in L^\infty(\T^2)$. The last statement follows from
\[0\le1-e^{-u}\le u\quad\text{for }u\ge0.\]
\end{proof}

\positionlemma*
\begin{proof}
The function $g_N$ in (ii) of Lemma \ref{lemma:coherent} is the density on $\T$ of $Z$ modulo $1$, where $Z$ is a normal random variable with mean $0$ and variance $(4\pi N)^{-1}$. So the $g_N$-measure of the set of $s\in\T$ with $\dist(s,0)\ge\rho$ is at most $2e^{-2\pi N\rho^2}$. By (ii) of Lemma \ref{lemma:coherent},
\[\Hr_\psi(I\times\T)=\sum_{k=0}^{N-1}\left|\psi_k\right|^2\int_Ig_N\left(q'-q_k\right)\,dq'\quad\text{for a Borel set }I\subset\T.\]
For $I=B_1(q,r-\rho)$, the integral is at most $1$ if $q_k\in B_1(q,r)$ and at most $2e^{-2\pi N\rho^2}$ otherwise. This gives the first inequality. For $I=B_1(q,r+\rho)$, the integral is at least $1-2e^{-2\pi N\rho^2}$ if $q_k\in B_1(q,r)$. This gives the second inequality.
\end{proof}

\overlaplemma*
\begin{proof}
Write $\kappa=|\langle\psi,\Phi_x\rangle|^2$, $u=e^{-\frac\pi4Nr^2}$, and $A=\Op_N^\aw(\indicator_{B_2(x,r)})$. Then $0\le A\le\Id$, and $\langle A\psi,\psi\rangle=\Hr_\psi(B_2(x,r))$ by \eqref{eq:antiWick}. By the Cauchy--Schwarz inequality for the non-negative forms $\langle A\cdot,\cdot\rangle$ and $\langle(\Id-A)\cdot,\cdot\rangle$,
\[\sqrt\kappa\le\left|\left\langle A\psi,\Phi_x\right\rangle\right|+\left|\left\langle(\Id-A)\psi,\Phi_x\right\rangle\right|\le\langle A\psi,\psi\rangle^{\frac12}\left\|\Phi_x\right\|+\tau,\]
where
\[\tau^2=\left\langle(\Id-A)\Phi_x,\Phi_x\right\rangle=\Hr_{\Phi_x}\left(\T^2\setminus B_2(x,r)\right).\]
By (iv) of Lemma \ref{lemma:coherent},
\[\tau^2\le81N\int_{|z|\ge r}e^{-\frac\pi2N|z|^2}\,dz=162u^2\quad\text{and}\quad\left\|\Phi_x\right\|^2\le1+5e^{-\frac\pi2N}\le\frac54.\]
So $\tau\le13u$ and $\tau\le\|\Phi_x\|\le\frac98$. Squaring the first inequality and using $\langle A\psi,\psi\rangle\le1$ and $e^{-\frac\pi2N}\le u$, we obtain
\[\kappa\le\langle A\psi,\psi\rangle\left\|\Phi_x\right\|^2+\left(2\left\|\Phi_x\right\|+\tau\right)\tau\le\langle A\psi,\psi\rangle+5e^{-\frac\pi2N}+\frac{27}8\tau\le\langle A\psi,\psi\rangle+49u.\]
This proves the first inequality.

By Lemma \ref{lemma:positionHusimi} with $\rho=r/2$,
\[|\psi|^2\left(B_1(q,r)\right)\ge\Hr_\psi\left(B_1\left(q,\frac r2\right)\times\T\right)-2e^{-\frac\pi2Nr^2}\ge\Hr_\psi\left(B_2\left(x,\frac r2\right)\right)-2e^{-\frac\pi2Nr^2}.\]
The first inequality with $r/2$ in place of $r$ gives the second.
\end{proof}

\section*{Acknowledgements}
I benefited from discussions with Fr\'ed\'eric Faure, Andrew Hassell, St\'ephane Nonnenmacher, Ze\'ev Rudnick, and Steve Zelditch. This work began when I participated in the 2018 Special Year in Analysis at the Mathematical Sciences Institute of the Australian National University. I thank the institute for its hospitality.

I obtained the first versions of the results in Section \ref{sec:positive} in 2018 \cite{Han3}, those in Section \ref{sec:negativequasi} in 2020 \cite{Han4}, and the results in Section \ref{sec:negativeeigen} in 2026. This paper supersedes the two preprints. 

\section*{Use of artificial intelligence}
I used Claude (Anthropic) to assist with the exposition and the reference searches.


\newpage
\section*{List of symbols}
\begin{center}
\begin{tabular}{lp{9cm}r}
Symbol & Meaning & Page\\
\hline
\multicolumn{3}{l}{\textit{Classical dynamics}}\\
$M$ & cat map & \pageref{idx:catmap}\\
$\lambda$, $\sigma$ & Lyapunov exponent and sign of $\Tr M$ & \pageref{idx:Lyapunov}\\
$v_\pm$, $x_\pm$ & unstable and stable directions; coordinates & \pageref{idx:stable}\\
$C_M$ & constant in the growth of $M^t$ & \pageref{idx:CM}\\
$c_M$ & constant of bad approximability & \pageref{lemma:badapprox}\\
$u\wedge v$ & symplectic product & \pageref{idx:symplectic}\\
$\Leb$ & Lebesgue measure on $\T^2$ & \pageref{idx:Leb}\\
$\gamma$, $\mu_\gamma$ & periodic orbit and orbit measure & \pageref{idx:orbitmeasure}\\
$P(N)$ & period of $M$ modulo $N$ & \pageref{idx:period}\\
$\Pc_C$ & integers $N$ with short periods & \pageref{idx:shortperiod}\\
$T_E$ & Ehrenfest time & \pageref{idx:Ehrenfest}\\
\hline
\multicolumn{3}{l}{\textit{Quantization}}\\
$N$ & dimension of $\Hc_N$, inverse semiclassical parameter & \pageref{idx:N}\\
$\Hc_N$ & Hilbert space of quantum states & \pageref{idx:HN}\\
$\theta$ & Floquet angle & \pageref{idx:Floquet}\\
$T_v$ & translation operator on the plane & \pageref{idx:translation}\\
$T_N(n)$ & translation operator on the torus & \pageref{eq:TNn}\\
$\eb_n$, $\tilde a(n)$ & Fourier mode and Fourier coefficient & \pageref{idx:Fourier}\\
$\Op_N$ & Weyl quantization & \pageref{idx:Weyl}\\
$\Op_N^\aw$ & anti-Wick quantization & \pageref{idx:antiWick}\\
$\hat M$ & quantum cat map & \pageref{idx:quantumcat}\\
$Q$ & quantum period & \pageref{idx:qperiod}\\
$\phi_x$ & coherent state on the plane & \pageref{idx:coherent}\\
$\Phi_x$ & coherent state on the torus & \pageref{idx:coherenttorus}\\
$\Pi_N$ & periodization & \pageref{idx:periodization}\\
$W_\psi$ & Wigner function & \pageref{idx:Wigner}\\
$\Hr_\psi$ & Husimi function & \pageref{idx:Husimi}\\
$|\psi|^2$ & position density & \pageref{idx:position}\\
\hline
\multicolumn{3}{l}{\textit{Small scales}}\\
$B_1(q,r)$, $B_2(x,r)$ & balls in $\T$ and $\T^2$ & \pageref{idx:balls}\\
$\alpha$, $\beta$ & exponents of logarithmic and polynomial scales & \pageref{idx:exponents}\\
$\delta_N$ & error in equidistribution at small scales & \pageref{idx:QEscale}\\
$S(N)$ & subsets of full density & \pageref{idx:fulldensityN}\\
$|S|$ & number of elements of a finite set & \pageref{idx:card}\\
$\ve$ & order of a quasimode & \pageref{idx:quasimode}\\
$h$ & semiclassical parameter on manifolds & \pageref{idx:h}\\
\end{tabular}
\end{center}
\end{document}